\documentclass[10pt, a4paper]{amsart}
\usepackage{amssymb,graphicx}
\usepackage{mathrsfs, bbm, slashed}
\usepackage{color}
\usepackage{enumerate}
\usepackage{a4wide}
\usepackage[colorlinks=true,citecolor=blue,linkcolor=blue]{hyperref}

\newcommand{\R}{\mathbb{R}}

\newcommand{\C}{\mathbb{C}}

\newcommand{\Z}{\mathbb{Z}}

\newcommand{\N}{\mathbb{N}}
\newcommand{\T}{\mathbb{T}}

\newcommand{\bS}{\mathbb{S}}

\newcommand{\sF}{\mathscr{F	}}
\newcommand{\sL}{\mathscr{L}}
\newcommand{\sH}{\mathscr{H}}
\newcommand{\sK}{\mathscr{K}}

\newcommand{\cS}{\mathscr{S}}
\newcommand{\sX}{\mathscr{X}}

\newcommand{\fS}{\mathfrak{S}}

 \newcommand{\psido}{$\Psi$DO}
 \newcommand{\psidos}{$\Psi$DOs}
 \newcommand{\GL}{\op{GL}}
\newcommand{\acoupt}[2]{\ensuremath{(#1,#2)}}

\def\Xint#1{\mathchoice
{\XXint\displaystyle\textstyle{#1}}%
{\XXint\textstyle\scriptstyle{#1}}%
{\XXint\scriptstyle\scriptscriptstyle{#1}}%
{\XXint\scriptscriptstyle\scriptscriptstyle{#1}}%
\!\int}
\def\XXint#1#2#3{{\setbox0=\hbox{$#1{#2#3}{\int}$}
\vcenter{\hbox{$#2#3$}}\kern-.5\wd0}}

\def\dashint{\Xint-}

\newcommand{\bint}{\ensuremath{\dashint}}

\newcommand{\op}{\operatorname} 

\newcommand{\Sp}{\op{Sp}}

\newcommand{\Vol}{\op{Vol}}
\newcommand{\Tr}{\op{Tr}}

\newcommand{\Res}{\op{Res}}

\newcommand{\sa}{\textup{sa}}

\newcommand{\scal}[2]{\ensuremath{\left\langle #1 | #2 \right\rangle}} 
\newcommand{\acou}[2]{\ensuremath{\left\langle #1 , #2 \right\rangle}} 
\newcommand{\bigscal}[2]{\ensuremath{\big\langle #1 | #2 \big\rangle}} 
\newcommand{\bigacou}[2]{\ensuremath{\big\langle #1 , #2 \big\rangle}}

\newcommand{\car}{\mathbbm{1}}
\newcommand{\LLogL}{L\!\log\!L}

\def\Xint#1{\mathchoice
{\XXint\displaystyle\textstyle{#1}}%
{\XXint\textstyle\scriptstyle{#1}}%
{\XXint\scriptstyle\scriptscriptstyle{#1}}%
{\XXint\scriptscriptstyle\scriptscriptstyle{#1}}%
\!\int}
\def\XXint#1#2#3{{\setbox0=\hbox{$#1{#2#3}{\int}$ }
\vcenter{\hbox{$#2#3$ }}\kern-.6\wd0}}

\numberwithin{equation}{section}

\newtheorem{theorem}{Theorem}[section]
\newtheorem{proposition}[theorem]{Proposition}
\newtheorem{corollary}[theorem]{Corollary}

\newtheorem{lemma}[theorem]{Lemma}

\newtheorem{conjecture}[theorem]{Conjecture}

\theoremstyle{definition}
\newtheorem{definition}[theorem]{Definition}

\theoremstyle{remark}
\newtheorem{example}[theorem]{Example}
\newtheorem{remark}[theorem]{Remark}
 
\newtheorem*{claim*}{Claim}

\title{Dixmier Trace Formulas and Negative Eigenvalues of Schr\"odinger Operators on Curved Noncommutative Tori}
\date{\today}

\author{Edward McDonald}
 \address{Departments of Mathematics, Pennsylvania State University, State College, PA, USA}
 \email{eamcd92@gmail.com}

\author{Rapha\"el Ponge}
 \address{School of Mathematics, Sichuan University, Chengdu, China}
 \email{ponge.math@icloud.com}



\begin{document}
\begin{abstract}
In a previous paper we established Cwikel-type estimates on noncommutative tori and used them to get analogues in this setting of the Cwikel-Lieb-Rozenblum (CLR) and Lieb-Thirring inequalities for negative eigenvalues of fractional Schr\"odinger operators. In this paper, we focus on ``curved'' NC tori, where the role of the usual Laplacian is played by Laplace-Beltrami operators associated with arbitrary Riemannian metrics. The Cwikel-type estimates of our previous paper are extended to pseudodifferential operators and powers of Laplace-Beltrami operators. There are several applications of these estimates. First, we get $L_p$-versions of the usual formula for the trace of \psidos\ on NC tori, i.e., for combinations of \psidos\ with $L_p$-position operators. Next, we get $L_p$-versions of the analogues for NC tori Connes' trace theorem and Connes' integration formula. They give formulas for the NC integrals (a.k.a.\ Dixmier traces) of products of $L_p$-position operators with \psidos\ or powers of the Laplace-Beltrami operators. Moreover, by combining our Cwikel-type estimates with suitable versions of the Birman-Schwinger principle we get versions of the CLR and Lieb-Thirring inequalities for negative eigenvalues of fractional Schr\"odinger operators associated with powers of Laplace-Beltrami operators and $L_p$-potentials. As in the original Euclidean case the Lieb-Thirring inequalities imply a dual  Sobolev inequality for orthonormal families. Finally, we discuss spectral asymptotics and semiclassical Weyl's laws for the our classes of operators on curved NC tori. This superseded a previous conjecture of~\cite{MP:JMP22}. 
 \end{abstract}
\maketitle

\section{Introduction}     
Noncommutative tori are arguably the most well known examples of noncommutative spaces. In particular, they naturally appear in the noncommutative geometry approach to the quantum Hall effect~\cite{BES:JMP94} and topological insulators~\cite{BCR:RMP16, PS:Springer16}. 
Two-dimensional noncommutative tori arise from actions of $\Z$ on the circle $\bS^1$ by irrational rotations. More generally, an $n$-dimensional noncommutative torus $\T^n_\theta$ is the NC space whose $C^*$-algebra $C(\T^n_\theta)$ is generated by unitaries $U_1, \ldots, U_n$ subject to the relations, 
\begin{equation*}
 U_lU_j = e^{2i\pi \theta_{jl}}U_jU_l, \qquad j,l=1,\ldots, n, 
\end{equation*}
where $\theta=(\theta_{jl})$ is a given real anti-symmetric matrix. For $\theta=0$ we recover the $C^*$-algebra $C(\T^n)$ of continuous functions on the ordinary torus $\T^n=\R^n/\Z^n$. There is a natural $C^*$-action of  $\R^n$ on $C(\T^n_\theta)$ generated by the canonical derivations $\partial_1, \ldots, \partial_n$ such that $\partial_j(U_l)=i\delta_{jl}U_j$.  The smooth NC torus $C^\infty(\T^n_\theta)$ is the precisely the $*$-algebra of the smooth elements for that action.  The $L_p$-spaces $L_p(\T^n_\theta)$ are the NC $L_p$-spaces associated with the standard (normalized) trace $\tau$ of $C(\T^n_\theta)$. 
We refer to Section~\ref{sec:NCtori} for the main background on noncommutative tori.  

In~\cite{MP:JMP22, MSX:CMP19} analogues for NC tori of the estimates of Cwikel~\cite{Cw:AM77} were derived. These Cwikel-type estimates were used in~\cite{MP:JMP22} to get versions for noncommutative tori of the Cwikel-Lieb-Rozenblum and Lieb-Thirring inequalities for negative eigenvalues of fractional Schr\"odinger operators associated with powers of the ordinary Laplacian $\Delta=-(\partial_1^2+\cdots \partial_n^2)$. A semiclassical Weyl's law for such operators was further conjectured in~\cite{MP:JMP22}. This is part of a recent revival of interest in Cwikel estimates (see, e.g.,~\cite{Fr:JST14, HKRV:arXiv18, LeSZ:2020, SZ:arXiv20}). 

In this paper we deal with curved NC tori. Seminal articles of Connes-Tretkoff~\cite{CT:Baltimore11}  and Connes-Moscovici~\cite{CM:JAMS14} spurred an upsurge of activity on understanding the spectral geometry of curved noncommutative tori (see, e.g.,~\cite{CF:MJM19, DS:SIGMA15, FK:JNCG12, FK:JNCG15, FGK:JNCG19, GK:arXiv18, IM:JGP18, LM:GAFA16, Liu:arXiv20, Po:JMP20, SZ:arXiv19} and the references therein; see also~\cite{Co:Survey19, FK:Survey19} for recent surveys). By a curved NC torus it is meant a NC torus equipped with a Riemannian metric in the sense of~\cite{HP:JGP20, Ro:SIGMA13} (see also Section~\ref{sec:Riemannian}). The spectral geometry of a curved NC torus is then studied through the analysis of the corresponding Laplace-Beltrami operator. Until recently the main focus was on conformal deformations of the flat metric. The construction in~\cite{HP:JGP20} of the Laplace-Beltrami operator then made it possible to extend this program to arbitrary Riemannian metrics (see, e.g., \cite{Liu:arXiv20, Po:JMP20, SZ:arXiv19}). 

Throughout this paper we focus on Laplace-Beltrami operators $\Delta_{g,\nu}$ associated with an arbitrary Riemannian metric $g=(g_{ij})$ on $\T^n_\theta$ and a smooth positive density $\nu$ (i.e., a positive invertible element of $C^\infty(\T^n_\theta)$). 

In the terminology of~\cite{CT:Baltimore11} 
datum of $\nu$ defines a ``conformal'' weight $\hat{\tau}_\nu(x):=(2\pi)^n{\tau}(x\nu)$. Associated with this weight is the 
$L_2$-space $L_2(\T^n_\theta;\nu)$ given by the completion with respect to the inner product,   
\begin{equation}
 \scal{x}{y}_\nu:= \hat{\tau}(xy^*)=(2\pi)^n{\tau}(y^*\nu x), \qquad x,y\in C(\T^n_\theta), 
 \label{eq:Intro.nu-inner-product}
\end{equation}
where $\tau$ is the standard trace of $C(\T^n_\theta)$.  We have a natural $*$-representation $\lambda_\nu$ of $C(\T^n_\theta)$ in $L_2(\T^n_\theta;\nu)$.  For $\nu=1$ this is just the left-regular representation of $C(\T^n_\theta)$ in the original $L_2$-space $L_2(\T^n_\theta)$. The Laplace-Beltrami operator $\Delta_{g,\nu}$ is essentially selfadjoint on $L_2(\T^n_\theta;\nu)$ and has a compact resolvent belonging to the weak Schatten class $\sL_{n/2,\infty}$ (see Section~\ref{sec:Riemannian}). 

\subsection*{Cwikel-type estimates} The Cwikel-type estimates of~\cite{MP:JMP22} provide explicit upper bounds for the weak-Schatten quasi-norms of operators $\lambda(x)\Delta^{-n/2p}$ and $\Delta^{-n/4p}\lambda(x)\Delta^{-n/4p}$ in terms of suitable $L_q$-norms of $x$. They include explicit upper bounds for the best constants in those estimates (see also Section~\ref{sec:Cwikel-flat}). 
Namely, if either $p\neq 2$ and $q=\max(p,2)$, or $p=2<q$, then 
\begin{equation}\label{eq:Intro.Cwikel-flat}
 \left\| \lambda(x) \Delta^{-\frac{n}{2p}}\right\|_{\sL_{p,\infty}} \leq c_n(p,q)\|x\|_{L_{q}} \qquad \forall x \in L_q(\T^n_\theta), `
\end{equation}
where $c_n(p,q)$ is an explicit constant. There are similar estimates for the operators
$\Delta^{-n/4p}\lambda(x)\Delta^{-n/4p}$ if $p\neq 1$ and $q=\max(p,1)$, or $p=1<q$ (\emph{loc.\ cit.}). 

We extend the Cwikel-type estimates of~\cite{MP:JMP22}  to operators  $\lambda(x)P$ and $Q^*\lambda(x)P$, where $P$ and $Q$ are negative order pseudodifferential operators (\psidos) (see Proposition~\ref{prop.specific-Cwikel-Prho} for the precise statement). We refer to Section~\ref{sec:psdo_applications} (and the references therein) for the main background on \psidos\ on NC tori. As an application of~\eqref{eq:Intro.Cwikel-flat} we extend to operators of the form $\lambda(x)P$ and $Q^*\lambda(x)P$, with $x$ in suitable $L_p$-spaces, the usual trace-formula for \psidos\ of order~$<-n$ (see Proposition~\ref{prop:Lp-trace-formula}). Namely, if $x\in L_2(\T^n_\theta)$ and $P$ is a \psido\ of order~$<-n$ with symbol $\rho(\xi)$, then 
\begin{equation*}
  \Tr\big[ \lambda(x)P\big] = \sum_{k\in \Z^n} \tau\big[x\rho(k)\big]. 
\end{equation*}
We have a similar formula for operators $Q^*\lambda(x)P$, where $x\in L_p(\T^n_\theta)$, $p>1$, and $P$ and $Q$ are \psidos\ of order~$<-n/2$ (\emph{loc.\ cit.}). 

We seek for curved versions of the Cwikel-type estimates of~\cite{MP:JMP22}. For the applications presented in this paper we need to state them in terms of quasi-norms associated with the weighted inner-product~(\ref{eq:Intro.nu-inner-product}) rather than with the original inner product of $L_2(\T^n_\theta)$. These two inner products are equivalent, and so this only affects the constants in those inequalities. The Cwikel estimates for \psidos\ allow us to get Cwikel-estimates for powers of the Laplace-Beltrami operators $\Delta_{g,\nu}$ (see Theorem~\ref{thm:curved-Cwikel}). Namely, if $p\neq 2$ and $q=\max(p,2)$, or $p=2<q$, we have 
\begin{equation}
 \left\| \lambda_\nu(x) \Delta_{g,\nu}^{-\frac{n}{2p}}\right\|_{\sL_{p,\infty}^{(\nu)}} \leq C_{g,\nu}(p,q)\|x\|_{L_{q}} \qquad \forall x \in L_q(\T^n_\theta), 
 \label{eq:Intro.Cwikel-curved}
 \end{equation}
 where $\|\cdot \|_{\sL_{p,\infty}^{(\nu)}}$ is $\sL_{p,\infty}$-quasi-norm with respect to the inner product~(\ref{eq:Intro.nu-inner-product}). This includes explicit upper bounds for the best constants $C_{g,\nu}(p,q)$ (\emph{loc.\ cit.}). There are similar estimates for the compressions $\Delta_{g,\nu}^{-n/4q}\lambda_\nu(x)\Delta_{g,\nu}^{-n/4q}$ if $p\neq 1$ and $q=\max(p,1)$, or $p=1<q$ (\emph{loc.\ cit.}). 

\subsection*{Dixmier trace formulas} In the framework of noncommutative geometry~\cite{Co:NCG}, the role of the integral is played by positive (normalized) traces on the weak trace class $\sL_{1,\infty}$ (see also Section~\ref{sec:L2-Connes-trace}). An important class of such traces is provided by Dixmier traces~\cite{Di:CRAS66} (see also~\cite{Co:NCG, LSZ:Book, Po:Weyl}). An operator $A\in \sL_{1,\infty}$ is called \emph{measurable} (resp., \emph{strongly measurable}) if the value $\varphi(A)$ is independent of the trace $\varphi$ as it ranges over Dixmier traces (resp., positive normalized traces). Alternatively, $A$ is measurable if and only if
\begin{equation*}
 \bint A = \lim_{N\rightarrow \infty} \frac{1}{\log (N+1)} \sum_{j<N} \lambda_j(A) \ \text{exists},  
\end{equation*}
where $\{\lambda_j(A)\}$ is an eigenvalue sequence for $A$.  In this case $\bint A$ is called the NC integral of $A$. 

If $(M^n,g)$ is a closed Riemannian manifold, Connes' trace theorem~\cite{Co:CMP88} asserts that every \psido\ $P$ of order~$-n$ on $M$ is measurable and its NC integral is given by the integral of this principal symbol over the cosphere bundle $S^*M$ (i.e., its NC residue trace). It was further shown that such operators are actually strongly measurable (see~\cite{KLPS:AIM13}).  This result is also a consequence of the Weyl's law for negative order \psidos\ of Birman-Solomyak~\cite{BS:VLU77, BS:VLU79, BS:SMJ79}. 

One consequence of Connes' trace theorem is the integration formula,
\begin{equation*}
 \bint \Delta_g^{-\frac{n}{4}}f \Delta_g^{-\frac{n}{4}}= \bint f \Delta_g^{-\frac{n}{2}}= \hat{c}(n) \int_M f \nu(g)(x) \qquad \forall f\in C^\infty(M),
 \label{eq:Intro.Integration-Formula}
\end{equation*}
where $\nu(g)(x)$ is the Riemannian measure and $\hat{c}(n)=\frac{1}{n}(2\pi)^{-n}|\bS^{n-1}|$. This shows that the noncommutative integral recaptures the 
Riemannian measure. The formula for $f\Delta_g^{-n/2}$ actually holds for all $f\in L_2(M)$ (see~\cite{KLPS:AIM13}).  The formula for $\Delta_g^{-n/4}f\Delta_g^{-n/4}$ further holds for any $f$ in the Orlicz space $\LLogL(M)$ (see~\cite{Ro:JST22, SZ:arXiv21}; see also~\cite{Po:MPAG22}). This result does not hold in general for functions in $L_1(M)$ (see~\cite{KLPS:AIM13}).

Connes' trace theorem and Connes' integration formulas were extended to NC tori in~\cite{Po:JMP20} (see also~\cite{MSZ:MA19} for the integration formula in the flat case). Namely, if $P$ is a \psido\ of order~$-n$ on $\T^n_\theta$, then $P$ is strongly measurable, and we have 
\begin{equation*}
\bint P =\frac1n \tau\big[c_P\big], \qquad c_P:=\int_{\bS^{n-1}}\rho_{-n}(\xi) d^{n-1}\xi, 
\end{equation*}
where $\rho_{-n}(\xi)$ is the principal symbol of $P$. The integration formula involves the spectral Riemannian density,
\begin{equation}\label{eq:Intro.spectral-Riem-density}
 \tilde{\nu}(g):= \frac{1}{|\bS^{n-1}|} \int_{\bS^{n-1}}|\xi|^{-n}_g d\xi, \qquad |\xi|_g^2:=\sum  \xi_i \xi_j g^{ij}, 
\end{equation}
where $g^{ij}$ are the entries of the dual metric $g^{-1}=(g^{ij})$. If the entries of $g$ pairwise commute with each other (e.g., $g$ is conformally flat), then $\tilde{\nu}(g)$ is just $\sqrt{\det(g)}$. For all $x\in C(\T^n_\theta)$ and for arbitrary positive smooth densities $\nu$, we  have the integration formula, 
\begin{equation}\label{eq:Intro.Int-formula}
 \bint \lambda_\nu(x)\Delta_{g,\nu}^{-\frac{n}{2}} = c(n)\tau\big[x\tilde{\nu}(g)\big] \qquad \text{where}\ c(n)=\frac{1}{n}|\bS^{n-1}|. 
\end{equation}

We use the Cwikel-type estimates of this paper to get $L_p$-versions of the above results. First, we look at operators of the form $\lambda(x)P$ and $Q^*\lambda(x)P$ where $P$ and $Q$ are negative order \psidos\ (see Theorem~\ref{thm:Lp-Connes-trace}). More precisely, if $x\in L_2(\T^n_\theta)$ and $P$ is a \psido\ of order~$-n$, then $\lambda(x)P$ is strongly measurable, and we have
\begin{equation}
 \bint \lambda(x) P= \frac{1}{n}\tau\big[xc_P\big]. 
 \label{eq:Intro.Lp-integration-PsiDOs} 
\end{equation}
 We have a similar result for operators of the form $Q^*\lambda(x)P$, where $x\in L_p(\T^n_\theta)$, $p>1$, and $P$ and $Q$ are \psidos\ of order~$-n/2$ (\emph{loc.\ cit.}).  
 
 The NC integral does not depend on the choice of the inner product.  Therefore, by using the fact that powers of the Laplace-Beltrami operators $\Delta_{g,\nu}$ are \psidos\ allows us to get $L_p$-versions of the integration formula~(\ref{eq:Intro.Int-formula}) (see Theorem~\ref{thm:Lp-curved-formula}). More precisely, if $x\in L_2(\T^n_\theta)$, then $ \lambda_\nu(x)\Delta_{g,\mu}^{-n/2}$ is strongly measurable and its NC integral1 is computed by the formula~(\ref{eq:Intro.Int-formula}). If $x\in L_p(\T^n_\theta)$, $p>1$, then $\Delta_{g,\mu}^{-n/4} \lambda_\nu(x)\Delta_{g,\mu}^{-n/4}$ is strongly measurable as well, and we have
\begin{equation*}
 \bint \Delta_{g,\nu}^{-\frac{n}{4}} \lambda_\nu(x)\Delta_{g,\nu}^{-\frac{n}{4}} = c(n)\tau\big[x\tilde{\nu}(g)\big].  
\end{equation*}
 
\subsection*{CLR and Lieb-Thirring inequalities} For Schr\"odinger operators $H_V:=\Delta+V$ on $\R^n$, $n\geq 3$,  the CLR inequality~\cite{Cw:AM77,Li:BAMS76, Roz:1972} gives an upper-bound for the number of negative eigenvalues $N^{-}(H_V)$ (i.e., the number of bound states)  in terms of the $L_{n/2}$-norm of the negative part $V_{-}$. Namely,
\begin{equation}\label{eq:Intro.CLR}
N^{-}(\Delta+V) \leq C_n\int |V_{-}(x)|^{\frac{n}2}dx. 
\end{equation}
The Lieb-Thirring inequalities~\cite{LT:PRL75, LT:SMP76} for $n\geq 2$ and $\gamma>0$ further provide upper bounds for the $\gamma$-moments of the negative eigenvalues. That is, 
\begin{equation}\label{eq:Into.LTRn}
 \sum \lambda_j^{-}(H_V)^\gamma \leq L_{\gamma,n}\int |V_{-}(x)|^{\frac{n}2+\gamma}dx, 
\end{equation}
where $\{-\lambda_j^{-}(H_V)\}$ are the negative eigenvalues of $H_V$ counted with multiplicity. We also have Lieb-Thirring inequalities for $n=1$ and $\gamma\geq 1/2$ (see~\cite{LT:SMP76, We:CMP96}). The Lieb-Thirring inequality for  $\gamma=1$ is dual to a Sobolev inequality for (finite) orthonormal families in $L_2(\R^n)$. There are similar inequalities for fractional Schr\"odinger operators $\Delta^{n/2p}+V$ 
(see, e.g., \cite{Da:CMP83, Roz:1972, LS:JAM97, RS:SPMJ98}). All these inequalities have widespread applications in mathematical physics, including semiclassical Weyl's laws for $L_p$-potentials and Lieb's proof of the stability of matter (see~\cite{Fr:Survey20, FLW:Book, LS:Book} and the references therein). 

In~\cite{MP:JMP22} we obtained analogues of the CLR and Lieb-Thirring inequalities for fractional Schr\"odinger operators $\Delta^{n/2p}+\lambda(V)$ on $\T^n_\theta$.  This includes explicit upper bounds for the best CLR and LT constants. These inequalities were consequences of the Cwikel-type estimates~(\ref{eq:Intro.Cwikel-flat}) and suitable versions of the Birman-Schwinger principle. As in the Euclidean case the Lieb-Thirring inequality for $p=n/2$ and $\gamma=1$ is dual to a Sobolev inequality for orthonormal families.  

Likewise, the curved Cwikel-type estimates~(\ref{eq:Intro.Cwikel-curved}) enable us to get CLR and Lieb-Thirring inequalities for fractional Schr\"odinger operators 
$H_{V}^{(p)}:=\Delta_{g,\nu}^{n/2p}+\lambda_\nu(V)$ associated with powers of Laplace-Beltrami operators. As in~\cite{MP:JMP22} we get two CLR inequalities. First, we get an upper bound for the number of negative eigenvalues $N^{-}(H_{V}^{(p)})$ (Theorem~\ref{thm:CLR.CLR-NCtori}).  Namely, if $p\neq 1$ and $q=\max(p,1)$, or $p=1<q$, then, for all $V=V^*\in L_q(\T^n_\theta)$, we have
\begin{equation}\label{eq:Intro.CLR-inequality1}
 N^{-}\big(H_{V}^{(p)}\big)-1 \leq C_{g,\nu}(2p,2q)^{2p} \big(\big\|V_{-}\big\|_{L_q}\big)^p,  
\end{equation}
where $V_{-}=\frac{1}{2}(|V|-V)$ is the negative part of $V$ and $C_{g,\nu}(2p,2q)$ is the best constant in the inequality~(\ref{eq:Intro.Cwikel-curved}) for the pair $(2p,2q)$.  As in~\cite{MP:JMP22} the $-1$ summand in the l.h.s.~of~(\ref{eq:Intro.CLR-inequality1}) arises from removing the dimension of the nullspace of $\Delta_{g,\nu}$ (which is just $\C\cdot 1$). 

As the nullspace $\ker \Delta_{g,\nu}=\C\cdot 1$ is non-trivial, it is natural to restrict ourselves to its orthogonal complement in $L_2(\T^n_\theta;\nu)$, i.e., the subspace, 
\begin{equation*}
 \dot{L}_2(\T^n_\theta;\nu):=\left\{u\in L_2(\T^n_\theta);\ \tau[u\nu]=0\right\}. 
\end{equation*}
In other words, $\dot{L}_2(\T^n_\theta;\nu)$ consists of all $u\in L_2(\T^n_\theta;\nu)$ that have zero mean value with respect to $\nu$. We denote by $\dot{\Delta}_{g,\nu}$ and $\dot{\lambda}_\nu(V)$ the restrictions to $\dot{L}_2(\T^n_\theta;\nu)$ of  ${\Delta}_{g,\nu}$ and ${\lambda}_\nu(V)$, respectively. We get a homogeneous upper bound for the number of negative eigenvalues of the operators $\dot{H}^{(p)}_V:=\dot{\Delta}_{g,\nu}^{n/2p}+\dot{\lambda}_\nu(V)$ (Theorem~\ref{thm:CLR.CLR-NCtori2}). Namely, if $p\neq 1$ and $q=\max(p,1)$, or $p=1<q$, then, for all $V=V^*\in L_q(\T^n_\theta)$, we have
\begin{equation}\label{eq:Intro.CLR-inequality2}
 N^{-}\big(\dot{H}_{V}^{(p)}\big) \leq C_{g,\nu}(2p,2q)^{2p} \big(\big\|V_{-}\big\|_{L_q}\big)^p.   
\end{equation}

There are various ways to derive the Lieb-Thirring inequalities~(\ref{eq:Into.LTRn}) and the dual Sobolev inequality (see, e.g.,~\cite{Fr:Survey20, FLW:Book, LS:Book}). For $n\geq 3$, the simplest way is to deduce the Lieb-Thirring inequalities from the CLR inequality~(\ref{eq:Intro.CLR}). This was the strategy used in~\cite{MP:JMP22}. Likewise, the CLR inequality~(\ref{eq:Intro.CLR-inequality2}) allows us to get Lieb-Thirring inequalities for $n\geq 3$ (Theorem~\ref{thm:LT-inequality}). 
Namely, if $\gamma>0$ and $p>1$, then, for all $V=V^*\in L_{p+\gamma}(\T^n_\theta)$, we have
\begin{equation}\label{eq:Intro.LT-Inequality}
 \sum \lambda_j^{-}\big(\dot{H}_V^{(p)}\big)^{\gamma}\leq L_{g,\nu}(p,\gamma) \tau\big[(V_{-})^{p+\gamma}\big],  
\end{equation}
where $\{-\lambda_j^{-}(\dot{H}_V^{(p)})\}$ are the negative eigenvalues of $\dot{H}_V^{(p)}$, and the best constant $L_{g,\nu}(p,\gamma)$ is such that
\begin{equation*}
 L_{g,\nu}(p,\gamma) \leq \gamma \frac{\Gamma(p+1)\Gamma(\gamma)}{\Gamma(p+\gamma+1)}C_{g,\nu}(2p,2p)^{2p}.
\end{equation*}

As in the Euclidean case and in~\cite{MP:JMP22}, the Lieb-Thirring inequality~(\ref{eq:Intro.LT-Inequality}) for $p=n/2$ and $\gamma=1$ is dual to a Sobolev inequality (Theorem~\ref{thm:LT.Sobolev}). Namely,  for any orthonormal family $\{u_0, \ldots, u_N\}$ in $\dot{L}_2(\T^n_\theta;\nu)$ such that $\partial_ju_\ell \in {L}_2(\T^n_\theta;\nu)$, we have 
\begin{equation}\label{eq:Intro.Sobolev-ineq}
 \sum_{\ell\leq N} \big\| du_\ell\|^2_{g,\nu} \geq K_{g,\nu} \tau\left[\rho^{\frac{n+2}{n}}\right], \qquad \text{where}\ \rho:=(2\pi)^n\sum_{\ell \leq N}  \nu^{\frac12}u_{\ell}u_\ell^*\nu^{\frac12}.
\end{equation}
Here $du$ is the differential of $u$ and $\|\cdot \|_{g,\nu}$ is the $L_2$-norm on 1-forms defined by $g$ and $\nu$ (see Section~\ref{sec:Riemannian} and Section~\ref{CLR_section} for their precise definitions). Note that the constant $K_{g,\nu}$ is independent of $N$. If we set $q=2(n+2)n^{-1}$, then for $N=0$ the inequality becomes, 
\begin{equation*}
 \big\| du_0\|^2_{g,\nu} \geq K_{g,\nu} \big(\|u_0\|_{L_{q}(\T^n;\nu)}\big)^{q}, 
\end{equation*}
where  $L_{q}(\T^n;\nu)$ is the $L_q$-space associated with $\nu$ (see Section~\ref{sec:Riemannian}). Furthermore, the best constant $K_{g,\nu}$ is related to the best LT constant $L_{g,\nu}=L_{g,\nu}(n/2,1)$ by
\begin{equation*}
 K_{g,\nu}= \frac{n}{n+2} \bigg( \frac{n+2}{2}L_{g,\nu}\bigg)^{-\frac{2}{n}}. 
\end{equation*}

\subsection*{Spectral asymptotics}
The original motivation for the Cwikel estimates on $\R^n$, $n\geq 3$, came from a conjecture of Simon~\cite{Si:TAMS76} regarding semiclassical Weyl's laws for Schr\"odinger operators with $L_p$-potentials. Alternatively, the Birman-Schwinger principle reduces semiclassical Weyl's laws to (classical) spectral asymptotics for compact operators. In the case of sufficiently smooth potentials the latter are special cases of a fairly general Weyl's laws for negative order \psidos\ of Birman-Solomyak~\cite{BS:VLU77, BS:VLU79, BS:SMJ79}. We then can extend the results  to non-smooth potentials by using Cwikel-type estimates and Birman-Solomyak's perturbation theory (see, e.g.,~\cite{BS:JFAA70, BS:TMMS72,BS:AMST80}). Note also that  Birman-Solomyak's Weyl' law implies a stronger form of Connes' trace theorem (see, e.g., \cite{Po:Weyl, Ro:JST22, SZ:arXiv21}). 

In~\cite[Conjecture~8.8]{MP:JMP22} the authors conjectured a semiclassical Weyl's laws for semiclassical (fractional) Schr\"odinger operators $h^{n/p}\Delta^{n/2p}+\lambda(V)$ on $\T^n_\theta$ for suitable $L_q$-potentials $V$. In this paper, we give grounds for an analogue for (negative order) \psidos\ on $\T^n_\theta$ of Birman-Solomyak's Weyl's law (Conjecture~\ref{conj:Bir-Sol-WL}). Together with the Cwikel estimates of this paper this Weyl's law implies the following spectral asymptotics (Theorem~\ref{thm:Weyl-laws}). If $p\neq 1$ and $q=\max(p,1)$, or $p=1<q$, then, for every $x=x^*\in L_q(\T^n_\theta)$, we have
\begin{equation}\label{eq:Intro.spectral-asymp}
 \lim_{j\rightarrow \infty} j^{\frac1{p}} \lambda_j^{\pm} \left(\Delta_{g,\nu}^{-\frac{n}{4p}}\lambda_\nu(x)\Delta_{g,\nu}^{-\frac{n}{4p}}\right)  = 
 \left(\frac{1}{n} \int_{\bS^{n-1}} \tau\left[ \left(|\xi|_g^{-\frac{n}{2p}}x|\xi|_g^{-\frac{n}{4p}}\right)_{\pm}^p\right] d\xi\right)^{\frac1{p}}. 
 \end{equation}
where $\pm \lambda_j^\pm(\Delta_{g,\nu}^{-n/4p}\lambda_\nu(x)\Delta_{g,\nu}^{-n/4p})$ are the positive/negative eigenvalues of $\Delta_{g,\nu}^{-n/4p}\lambda_\nu(x)\Delta_{g,\nu}^{-n/4p}$, and  $|\xi|_g$ is defined as in~(\ref{eq:Intro.spectral-Riem-density}). This implies stronger forms of Connes' trace theorem for NC tori from~\cite{Po:SIGMA20} and its $L_p$-version~(\ref{eq:Intro.Lp-integration-PsiDOs}) (see Corollary~\ref{cor:Weyl.PsiDOs}). 

Thanks to the Birman-Schwinger principle the spectral asymptotics~(\ref{eq:Intro.spectral-asymp}) imply a semiclassical Weyl's law for fractional Schr\"odinger operators (Theorem~\ref{thm:SCWeyl}). Namely, if $p$ and $q$ are as above, then,  for all $V=V^*\in L_q(\T^n_\theta)$, we have
  \begin{equation*}
\lim_{h\rightarrow 0^+} h^{n}N^{-}\left(h^{\frac{n}{p}}\Delta_{g,\nu}^{\frac{n}{2p}}+\lambda_\nu(V)\right)= \frac{1}{n} \int_{\bS^{n-1}} \tau\left[ \left(|\xi|_g^{-\frac{n}{2p}}V|\xi|_g^{-\frac{n}{2p}}\right)_{-}^p\right] d\xi. 
\end{equation*}
In particular, for the flat Euclidean metric $g_0=(\delta_{ij})$ this delivers the conjecture of~\cite{MP:JMP22} mentioned above. A proof of that conjecture for $p=n/2$ and $n\geq 3$ is also given in~\cite{MSZ:arXiv21}. However, the approach of~\cite{MSZ:arXiv21} cannot be extended to non-flat metrics (see Remark~\ref{rmk:MSZ} on this point). 


\subsection*{Outline of the paper}
The paper is organized as follows. In Section~\ref{sec:NCtori}, we review the main background on noncommutative tori, including their $L_p$-spaces, Sobolev spaces, and pseudodifferential operators. 
In Section~\ref{sec:Riemannian},  we recall the main facts regarding Riemannian metrics and Laplace-Beltrami operators on NC tori.
In Section~\ref{sec:Cwikel-flat}, we obtain Cwikel-type estimates for \psidos\ and powers of Laplace-Beltrami operators. We use them to get $L_p$-trace formulas . In Section~\ref{sec:L2-Connes-trace}, we establish $L_p$-versions of Connes' trace theorem and Connes' integration formulas for NC tori. In Section~\ref{CLR_section}, we establish the CLR and Lieb-Thirring inequalities for fractional Schr\"odinger operators associated with powers of Laplace-Beltrami operators and get the dual Sobolev inequality~(\ref{eq:Intro.Sobolev-ineq}). In Section~\ref{sec:Weyl}, we discuss spectral asymptotics and semiclassical Weyl's laws for fractional Schr\"odinger operators with $L_p$-potentials. 

\subsection*{Acknowledgements} 
The authors wish to thank Rupert Frank, Galina Levitina, Grigori Rozenblum, Fedor Sukochev, and Dmitriy Zanin for various discussions related to the subject matter of the paper.


\section{Noncommutative Tori} \label{sec:NCtori}
In this section, we review the main definitions and properties of noncommutative $n$-tori, $n\geq 2$. We refer to~\cite{Co:NCG, HLP:IJM19a, Ri:CM90}, and the references therein, for a more comprehensive account.
 
Throughout this paper, we let $\theta =(\theta_{jk})$ be a real anti-symmetric $n\times n$-matrix, and denote by $\theta_1, \ldots, \theta_n$ its column vectors.  We also let  $L_2(\T^n)$ be the Hilbert space of $L_2$-functions on the ordinary torus $\T^n=\R^n\slash (2\pi \Z)^n$ equipped with the  inner product, 
\begin{equation} \label{eq:NCtori.innerproduct-L2}
 \scal{\xi}{\eta}= (2\pi)^{-n} \int_{\T^n} \xi(x)\overline{\eta(x)}d x, \qquad \xi, \eta \in L_2(\T^n). 
\end{equation}
 For $j=1,\ldots, n$, let $U_j:L_2(\T^n)\rightarrow L_2(\T^n)$ be the unitary operator defined by 
 \begin{equation*}
 \left( U_j\xi\right)(x)= e^{ix_j} \xi\left( x+\pi \theta_j\right), \qquad \xi \in L_2(\T^n). 
\end{equation*}
 We then have the relations, 
 \begin{equation} \label{eq:NCtori.unitaries-relations}
 U_kU_j = e^{2i\pi \theta_{jk}} U_jU_k, \qquad j,k=1, \ldots, n. 
\end{equation}

The \emph{noncommutative torus} is the noncommutative space whose $C^*$-algebra  $C(\T^n_\theta)$ and von Neuman algebra $L_\infty(\T^n_\theta)$ are generated by the unitary operators $U_1, \ldots, U_n$.  For $\theta=0$ we obtain the $C^*$-algebra $C(\T^n)$ of continuous functions on the ordinary $n$-torus $\T^n$ and the von Neuman algebra $L_\infty(\T^n)$ of essentially bounded measurable functions on $\T^n$. Note that~(\ref{eq:NCtori.unitaries-relations}) implies that $C(\T^n_\theta)$ (resp., $L_\infty(\T^n_\theta)$) is the norm closure (resp., weak closure) in $\sL(L_2(\T^n))$ of the linear span of the unitary operators, 
 \begin{equation*}
 U^k:=U_1^{k_1} \cdots U_n^{k_n}, \qquad k=(k_1,\ldots, k_n)\in \Z^n. 
\end{equation*}

 \subsection{GNS representation} Let $\tau:\sL(L_2(\T^n))\rightarrow \C$ be the state defined by
 \begin{equation*}
 \tau (T)= \scal{T1}{1}=(2\pi)^{-n}\int_{\T^n} (T1)(x) d  x, \qquad T\in \sL\left(L_2(\T^n)\right).
\end{equation*}
This induces a continuous tracial state on the von Neuman algebra $L_\infty(\T^n_\theta)$ such that $\tau(1)=1$ and $\tau(U^k)=0$ for $k\neq 0$. 
 The GNS construction allows us to get a $*$-representation of $L_\infty(\T^n_\theta)$.  More precisely, let $\scal{\cdot}{\cdot}$ be the inner product on $C(\T^n_\theta)$ defined by
\begin{equation}
 \scal{u}{v} = \tau\left( uv^* \right), \qquad u,v\in C(\T^n_\theta). 
 \label{eq:NCtori.cAtheta-innerproduct}
\end{equation}
We then denote by $L_2(\T^n_\theta)$ the Hilbert space given by the completion of $C(\T^n_\theta)$ with respect to this inner product. 

The action of $C(\T^n_\theta)$ on itself by left-multiplication then uniquely extends to a $*$-representation,
\begin{equation}\label{eq:left.multiplier.representation}
 \lambda: L_\infty(\T^n_\theta) \longrightarrow \sL\big(L_2(\T^n_\theta)\big). 
\end{equation}
For $\theta=0$ we recover the Hilbert space $L_2(\T^n)$ with the inner product~(\ref{eq:NCtori.innerproduct-L2}) and the representation of $L_\infty(\T^n)$ by bounded multipliers. 

The family $(U^k)_{k \in \Z^n}$ is actually an orthonormal basis of $L_2(\T^n_\theta)$. Thus, every $u\in L_2(\T^n_\theta)$ can be uniquely written as 
\begin{equation} \label{eq:NCtori.Fourier-series-u}
 u =\sum_{k \in \Z^n} u_k U^k, \qquad u_k:=\scal{u}{U^k}, 
\end{equation}
where the series converges in $L_2(\T^n_\theta)$. For $\theta =0$ we recover the Fourier series decomposition in  $L_2(\T^n)$. 

\subsection{The smooth algebra $C^\infty(\T^n_\theta)$} The natural action of $\R^n$ on $\T^n$ by translation gives rise to an action on $\sL(L_2(\T^n))$. This induces a $*$-action $(s,u)\rightarrow \alpha_s(u)$ on $C(\T^n_\theta)$ given on the basis element $U^k$ by 
\begin{equation*}
\alpha_s(U^k)= e^{is\cdot k} U^k \qquad  \text{for all $k\in \Z^n$ and $s\in \R^n$}. 
\end{equation*}
This action is strongly continuous, and so we obtain a $C^*$-dynamical system $(C(\T^n_\theta), \R^n, \alpha)$. We are especially interested in the subalgebra $C^\infty(\T^n_\theta)$ of smooth elements of this $C^*$-dynamical system (a.k.a.~\emph{smooth noncommutative torus}). Namely,  
\begin{equation*}
 C^\infty(\T^n_\theta):=\biggl\{ u \in C(\T^n_\theta); \ \alpha_s(u) \in C^\infty(\R^n; C(\T^n_\theta))\biggr\}. 
\end{equation*}
The unitaries $U^k$, $k\in \Z^n$, are contained in $C^\infty(\T^n_\theta)$, and so $C^\infty(\T^n_\theta)$ is a dense subalgebra of $C(\T^n_\theta)$ and $L_\infty(\T^n_\theta)$. Denote by $\cS(\Z^n)$ the space of rapid-decay sequences with complex entries. In terms of the Fourier series decomposition~(\ref{eq:NCtori.Fourier-series-u}) we have
\begin{equation*}
 C^\infty(\T^n_\theta)=\bigg\{ u=\sum_{k\in \Z^n} u_k U^k; (u_k)_{k\in \Z^n}\in  \cS(\Z^n)\bigg\}. 
\end{equation*}
For $\theta=0$ we recover the algebra $C^\infty(\T^n)$ of smooth functions on the ordinary torus $\T^n$ and the Fourier-series description of this algebra. 

For $j=1,\ldots, n$, let $\partial_j:C^\infty(\T^n_\theta)\rightarrow C^\infty(\T^n_\theta)$ be the  derivation defined by 
\begin{equation*}
 \partial_j(u) =\partial_{s_j} \alpha_s(u)|_{s=0}, \qquad u\in C^\infty(\T^n_\theta), 
\end{equation*}
For $\theta=0$ the derivation $\partial_j$ is just the derivation $\frac{\partial}{\partial x_j}$ on $C^\infty(\T^n)$. For $j,l=1,\ldots, n$, we have
\begin{equation*}
 \partial_j(U_l) = \left\{ 
 \begin{array}{ll}
 iU_j & \text{if $l=j$},\\
 0 & \text{if $l\neq j$}. 
\end{array}\right.
\end{equation*}
Given any multi-order $\beta \in \N_0^n$,  we set $\delta^\beta= \delta_1^{\beta_1} \cdots \delta_n^{\beta_n}$, wheres $\delta_j=\frac{1}{i}\partial_j$. 
We endow $C^\infty(\T^n_\theta)$ with the locally convex topology defined by the semi-norms,
\begin{equation}
 C^\infty(\T^n_\theta) \ni u \longrightarrow \left\|\delta^\beta (u)\right\| ,  \qquad \beta\in \N_0^n. 
\label{eq:NCtori.cAtheta-semi-norms}
\end{equation}
With the involution inherited from $C(\T^n_\theta)$ this turns $C^\infty(\T^n_\theta)$ into a (unital)  Fr\'echet $*$-algebra. The Fourier series~(\ref{eq:NCtori.Fourier-series-u}) of every $u\in C^\infty(\T^n_\theta)$ converges in $C^\infty(\T^n_\theta)$ with respect to this topology. In addition, it can be shown that $C^\infty(\T^n_\theta)$ is closed under holomorphic functional calculus (see, e.g., \cite{Co:AIM81, HLP:IJM19a}). 

\subsection{$L_p$-Spaces} 
The $L_p$-spaces of $\T^n_\theta$ are special instances of noncommutative $L_p$ spaces associated with a semi-finite faithful normal trace on a von Neumann algebra~\cite{Ku:TAMS58, Se:AM53} (see also~\cite{FK:PJM86, Ne:JFA74}). We refer to~\cite{FK:PJM86, Ku:TAMS58} for the main background on noncommutative $L_p$-spaces needed in this paper. 

By definition $L_\infty(\T^n_\theta)$ is a von Neumann algebra of bounded operators on $L_2(\T^n)$. Thus, a closed densely defined operator $T$ on $L_2(\T^n)$ is $L_\infty(\T^n_\theta)$-affiliated if every unitary in the commutant of $L_{\infty}(\T^n_\theta)$ in $\sL(L_2(\T^n))$ preserves the domain of $T$ and commutes with $T$. Furthermore, as $\tau$ is a finite faithful positive trace on $L_\infty(\T_\theta^n)$ every such operator is $\tau$-measurable in the sense of~\cite{FK:PJM86, Ne:JFA74}. Therefore, these operators form a $*$-algebra, where the sum and product of such operators are meant as the closures of their usual sum and product in the sense of unbounded operators (see~\cite{Ne:JFA74}). 

The space $L_1(\T^n_\theta)$ consists of all $L_\infty(\T^n_\theta)$-affiliated operators $x$ on $L_2(\T^n)$ such that
\begin{equation*}
 \tau\left(|x|\right):=\int_0^\infty \lambda d\tau(E_\lambda)<\infty, 
\end{equation*}
where $E_\lambda=\car_{[0,\lambda]}(|x|)$ is the spectral measure of $|x|$. We obtain a Banach space upon equipping $L_1(\T^n_\theta)$ with the norm, 
\begin{equation*}
 \|x\|_{L_1}:= \tau\big(|x|\big), \qquad x \in L_1(\T^n_\theta). 
\end{equation*}
We have a continuous inclusion with dense range of $L_\infty(\T^n_\theta)$ into $L_1(\T^n_\theta)$. The trace $\tau$ uniquely extends to a unitarily invariant continuous linear functional on $L_1(\T^n_\theta)$ such that
\begin{equation*}
    |\tau(x)|\leq \tau(|x|) = \|x\|_1\qquad \forall x \in L_1(\T^n_\theta).
\end{equation*}

For $p>1$, the space $L_p(\T^n_\theta)$ consists of all $L_\infty(\T^n_\theta)$-affiliated operators $x$ on $L_2(\T^n)$ such that $|x|^p\in L_1(\T^n_\theta)$. This is a Banach space with respect to the norm, 
\begin{equation*}
 \|x\|_{L_p}:= \tau(|x|^p)^{1/p},\qquad x \in L_p(\T^n_\theta).
\end{equation*}
We also have a continuous inclusion with dense range of $L_\infty(\T^n_\theta)$ into $L_p(\T^n_\theta)$.  In particular, for $p=2$ the above definition is consistent with the previous definition of $L_2(\T^n_\theta)$, since in both cases we get the completion of $L_\infty(\T^n_\theta)$ with respect to the same norm. 

We have the following version of H\"older's inequality. 

\begin{proposition}[\cite{FK:PJM86, Ku:TAMS58}]\label{prop:Holder}  Suppose that $p^{-1}+q^{-1}=r^{-1}\leq 1$. If $x\in L_p(\T^n_\theta)$ and $y\in L_q(\T^n_\theta)$, then $xy\in L_r(\T^n_\theta)$ with norm inequality, 
\begin{equation}
 \|xy\|_{L_r} \leq \|x\|_{L_p} \|y\|_{L_q}. 
\label{eq:Holder-Lp} 
\end{equation}
\end{proposition}

This implies that, for every $p\geq 1$, the left multiplication of $L_\infty(\T_\theta^n)$ on itself uniquely extends to an isometric representation, 
\begin{equation*}
\lambda: L_\infty(\T^n_\theta) \longrightarrow \sL\big(L_p(\T^n_\theta)\big).
\end{equation*}
For $p=2$ we recover the $*$-representation~(\ref{eq:left.multiplier.representation}). 

\subsection{Sobolev spaces} Given any $s\geq 0$, the Sobolev space $W_2^s(\T^n_\theta)$ is defined by 
\begin{equation}\label{sobolev_space_definition}
 W_2^s(\T^n_\theta):=\Big\{u =\sum_{k\in \Z^n} u_kU^k\in L_2(\T^n_\theta); \sum_{k\in \Z^n} (1+|k|^2)^s|u_k|^2<\infty\Big\}. 
\end{equation}
This is a Hilbert space with respect to the inner product and norm, 
\begin{equation}\label{eq:sobolev.norm.definition}
 \scal{u}{v}_s=\sum_{k\in \Z^n} (1+|k|^2)^{s}u_k\overline{v}_k, \qquad \|u\|_{W_2^s}=\bigg(\sum_{k\in \Z^n}  (1+|k|^2)^s|u_k|^2\bigg)^{\frac12}. 
\end{equation}
Note that $W_2^0(\T^n_\theta)=L_2(\T^n_\theta)$. 

Equivalently, let $\Delta=-(\partial_1^2+\cdots + \partial_n^2)$ be the Laplacian on $\T^n_\theta$. This is a non-negative selfadjoint operator on $L_2(\T^n_\theta)$ with domain $W_2^2(\T^n_\theta)$. We have 
\begin{equation*}
 \Delta\big(U^k)=|k|^2U^k, \qquad k\in \Z^n. 
\end{equation*}
In particular, $\Delta$ is isospectral to the Laplacian on the ordinary torus $\T^n$. Set $\Lambda=(1+\Delta)^{\frac12}$. Given any $s\geq 0$, we have 
\begin{equation*}
 W_2^s(\T^n_\theta):=\Big\{u\in L_2(\T^n_\theta);\ \Lambda^su\in L_2(\T^n_\theta)\Big\}, \qquad \|u\|_{W_2^s}=\|\Lambda^su\|_{W_2^0}. 
\end{equation*}
In particular, $\Lambda^s:W_2^s(\T^n_\theta)\rightarrow L_2(\T^n_\theta)$ is a unitary isomorphism.  Note also (see~\cite{Sp:Padova92, XXY:MAMS18}) that, given an integer $p\geq 0$, we have
\begin{equation*}
 W_2^p(\T^n_\theta)= \Big\{u\in L_2(\T^n_\theta); \ \delta^\alpha u\in L_2(\T^n_\theta) \  \ \forall \alpha \in \N_0^n, \, |\alpha|\leq p\big\}. 
\end{equation*}

We mention the following versions of Sobolev's embedding theorems. 

 \begin{proposition}[see~{\cite[Theorem 6.6]{XXY:MAMS18}}] \label{prop:Sobolev-embeddingLp} 
 Let $p\in [2,\infty)$. For every $s\geq n(1/2-p^{-1})$, we have a continuous inclusion of $W_2^s(\T^n_\theta)$ into $L_p(\T^n_\theta)$. This inclusion is compact if the inequality is strict. 
 \end{proposition}
 
\begin{proposition}[see {\cite[Appendix A]{HLP:IJM19b}}] \label{prop:Sobolev-embeddingLpC0}
 For any $s>n/2$, we have a compact inclusion of $W^s_2(\T^n_\theta)$ into $C(\T^n_\theta)$. 
\end{proposition}

\subsection{Pseudodifferential Operators on NC Tori}\label{sec:psdo_applications}
The pseudodifferential calculus on noncommutative tori is a special case of the pseudodifferential calculus for $C^*$-dynamical systems of Connes~\cite{Co:CRAS80} and Baaj~\cite{Ba:CRAS88}. A detailed account on the calculus on NC tori can be found in~\cite{HLP:IJM19a, HLP:IJM19b} (see also~\cite{Ta:JPCS18}). 

\begin{definition}[Standard Symbols; see~\cite{Ba:CRAS88, Co:CRAS80}]
$\bS^m (\T^n_\theta\times\R^n)$, $m\in\R$, consists of $C^\infty$-maps $\rho:\R^n\rightarrow C^\infty(\T^n_\theta)$ such that, for all multi-orders $\alpha$ and $\beta$, we have
\begin{equation} 
\label{eq:Symbols.standard-estimates}
\big\|\delta^\alpha \partial_\xi^\beta \rho(\xi)\big\| \leq C_{\alpha \beta} \left( 1 + | \xi | \right)^{m - | \beta |} \qquad \forall \xi \in \R^n .
\end{equation}
\end{definition}

\begin{remark}
 $\bS^m (\T^n_\theta\times\R^n)$ is a Fr\'echet space with respect to the topology generated by the semi-norms given by the best constants $C_{\alpha\beta}$ in~(\ref{eq:Symbols.standard-estimates}) (see~\cite{HLP:IJM19a}). 
\end{remark}

\begin{definition}[Homogeneous Symbols] 
$S_m (\T^n_\theta\times\R^n)$, $m \in \R$, consists of  $C^\infty$-maps $\rho:\R^n\backslash 0\rightarrow C^\infty(\T^n_\theta)$ that are homogeneous of degree $m$, i.e., 
$\rho( t \xi ) =t^m \rho(\xi)$ for all $\xi \in \R^n \backslash 0$ and $t > 0$. 
\end{definition}

\begin{definition}[Classical Symbols; see \cite{Ba:CRAS88}]\label{def:Symbols.classicalsymbols}
$S^m (\T^n_\theta\times\R^n )$, $m \in \R$, consists of $C^\infty$-maps $\rho:\R^n\rightarrow C^\infty(\T^n_\theta)$ that admit an asymptotic expansion,
\begin{equation*}
\rho(\xi) \sim \sum_{j \geq 0} \rho_{q-j} (\xi),  \qquad \rho_{q-j} \in S_{q-j} (\T^n_\theta\times\R^n ), 
\end{equation*}
where $\sim$ means that, for all $N\geq 0$ and multi-orders $\alpha$, $\beta$, we have
\begin{equation} \label{eq:Symbols.classical-estimates}
\Big\| \delta^\alpha \partial_\xi^\beta \big( \rho - \sum_{j<N} \rho_{q-j} \big)(\xi)\Big\| \leq C_{N\alpha\beta} | \xi |^{q-N-| \beta |} , 
\end{equation}
for all $\xi \in \R^n$ with $| \xi | \geq 1$. 
\end{definition}

\begin{remark} \label{rmk:Symbols.classical-inclusion}
We have $S^m(\T^n_\theta\times\R^n)\subset \bS^{m}(\T^n_\theta\times\R^n)$. 
\end{remark}

\begin{example}
 Any polynomial map of the form $\rho(\xi)=\sum_{|\alpha|\leq m} a_\alpha \xi^\alpha$, $a_\alpha\in C^\infty(\T^n_\theta)$, $m\in \N_0$, is in $S^m(\T^n_\theta\times\R^n)$.  
\end{example}

Given $\rho(\xi)\in \bS^m(\T^n_\theta\times\R^n)$, $m\in \R$. We let $P_\rho:C^\infty(\T^n_\theta) \rightarrow C^\infty(\T^n_\theta)$ be the linear operator defined by
\begin{equation} \label{eq:PsiDOs.PsiDO-definition}
P_\rho u = (2\pi)^{-n}\iint e^{is\cdot\xi}\rho(\xi)\alpha_{-s}(u)ds d\xi, \qquad u \in C^\infty(\T^n_\theta). 
\end{equation}
The above integral is meant as an oscillating integral (see~\cite{HLP:IJM19a}). Equivalently, for all $u=\sum_{k\in \Z^n} u_k U^k$ in $C^\infty(\T^n_\theta)$,  we have
\begin{equation} \label{eq:toroidal.Prhou-equation}
                P_{\rho}u = \sum_{k\in \Z^n} u_k \rho(k)U^k.  
\end{equation}
In any case, we get a continuous linear operator $P_\rho:C^\infty(\T^n_\theta) \rightarrow C^\infty(\T^n_\theta)$ (see~\cite{HLP:IJM19a}).

\begin{definition}
$\Psi^m(\T^n_\theta)$,  $m\in \R$, consists of all linear operators $P_\rho:C^\infty(\T^n_\theta)\rightarrow C^\infty(\T^n_\theta)$ with $\rho(\xi)$ in $S^m(\T^n_\theta\times\R^n)$.
If $\rho \in S^m(\T^n_\theta\times \R^n)$ is such that $P=P_\rho$ and $\rho \sim \sum_{j\geq 0} \rho_{q-j}$ where $\rho_q$ is non-zero, then $\rho_q$ is called the principal symbol of $P.$
\end{definition}
While the symbol $\rho$ is not determined uniquely by the operator $P_{\rho},$ the principal symbol $\rho_q$ is, see \cite[Remark 5.7]{HLP:IJM19b}.

\begin{example}\label{ex:PsiDO.Do-PsiDO}
 A differential operator on $\T^n_\theta$ is of the form $P=\sum_{|\alpha|\leq m}\lambda(a_\alpha)\delta^\alpha$, $a_\alpha\in C^\infty(\T^n_\theta)$ (see~\cite{Co:CRAS80, Co:NCG}). This is a \psido\ of order $m$ with symbol $\rho(\xi)= \sum_{|\alpha|\leq m} a_\alpha \xi^\alpha$ and principal symbol $\rho_m(\xi) = \sum_{|\alpha|=m} a_{\alpha}\xi^{\alpha}$ (see~\cite{HLP:IJM19a}). Note that for $m=0$ we get left-multiplication operators $\lambda(a)$ with $a\in C^\infty(\T^n_\theta)$. 
\end{example}

\begin{example}
 Given any $s\in \R$, we have $\Lambda^{s}=P_\rho$, with $\rho(\xi)=(1+|\xi|^2)^{s/2} \in S^{s}(\T^n_\theta\times\R^n)$. 
\end{example}

Suppose we are given symbols $\rho_1(\xi)\in \bS^{m_1}(\T^n_\theta\times\R^n)$, $m_1\in\R$, and $\rho_2(\xi)\in\bS^{m_2}(\T^n_\theta\times\R^n)$, $m_2\in\R$. As $P_{\rho_1}$ and $P_{\rho_2}$ are linear operators on $C^\infty(\T^n_\theta)$, the composition $P_{\rho_1}P_{\rho_2}$ makes sense as such an operator.  
In addition, we define the map $\rho_1\sharp\rho_2:\R^n\rightarrow C^\infty(\T^n_\theta)$ by
\begin{equation} \label{eq:Composition.symbol-sharp}
\rho_1\sharp\rho_2(\xi) = (2\pi)^{-n}\iint e^{it\cdot\eta}\rho_1(\xi+\eta)\alpha_{-t}[\rho_2(\xi)]dt d\eta , \qquad \xi\in \R^n , 
\end{equation}
where the above integral is meant as an oscillating integral (see~\cite{Ba:CRAS88, HLP:IJM19b}). 

\begin{proposition}[see \cite{Ba:CRAS88, Co:CRAS80, HLP:IJM19b}] \label{prop:Composition.sharp-continuity-standard-symbol}
For $j=1,2$ let $\rho_j(\xi)\in \bS^{m_j}(\T^n_\theta\times\R^n)$, $m_j\in \R$. 
\begin{enumerate}
 \item $\rho_1\sharp\rho_2(\xi)\in \bS^{m_1+m_2}(\T^n_\theta\times\R^n)$ and $ \rho_1\sharp\rho_2(\xi) \sim \sum\frac{1}{\alpha !}\partial_\xi^\alpha\rho_1(\xi)\delta^\alpha\rho_2(\xi)$. 

 \item The operators $P_{\rho_1}P_{\rho_2}$ and $P_{\rho_1\sharp \rho_2}$ agree.           
\end{enumerate}
\end{proposition}

\begin{corollary}\label{cor:PsiDOs.composition-classical}
For $j=1,2$, let $P_j\in \Psi^{m_j}(\T^n_\theta)$, $m_j\in \R$, have principal symbol $\rho_{m_j}(\xi)$. Then the composition  $P_1P_2$ is an operator in $\Psi^{m_1+m_2}(\T^n_\theta)$ whose principal symbol is $\rho_{m_1}(\xi)\rho_{m_2}(\xi)$. 
\end{corollary}

Given a linear operator $P: C^\infty(\T^n_\theta) \rightarrow C^\infty(\T^n_\theta)$, a formal adjoint is any linear operator $P^*: C^\infty(\T^n_\theta) \rightarrow C^\infty(\T^n_\theta)$ such that
\begin{equation*} 
\scal{P^*u}{v} =\scal{u}{Pv} \qquad \forall u,v \in C^\infty(\T^n_\theta). 
\end{equation*}
When it exists a formal adjoint is unique.

Let $\rho(\xi)\in\bS^m(\T^n_\theta\times \R^n)$, $m\in\R$, and set
\begin{equation} \label{eq:Adjoints.symbol-star}
\rho^\star(\xi) = (2\pi)^{-n} \iint e^{it\cdot\eta}\alpha_{-t}[\rho(\xi+\eta)^*]dtd\eta , \qquad \xi\in\R^n.
\end{equation}
where the integral is meant in the sense of oscillating integrals (see~\cite{HLP:IJM19b}). 

\begin{proposition}[{\cite{Ba:CRAS88, HLP:IJM19b}}] \label{prop:Adjoint}
Let $\rho\in\bS^m(\T^n_\theta\times \R^n)$. 
\begin{enumerate}
 \item $\rho^\star(\xi) \in \bS^m(\T^n_\theta\times \R^n)$ and $\rho^\star(\xi) \sim \sum \frac{1}{\alpha !}\delta^\alpha\partial_\xi^\alpha [\rho(\xi)^*]$. 
        
\item $P_{\rho^\star}$ is the formal adjoint of $P_\rho$. 
\end{enumerate}
\end{proposition}

\begin{corollary}[{\cite{Ba:CRAS88, HLP:IJM19b}}] 
If $P\in \Psi^{m}(\T^n_\theta)$, $m\in \R$, has principal symbol $\rho_m(\xi)$, then its formal adjoint $P^*$ is an operator in $\Psi^{m}(\T^n_\theta)$ whose principal symbol is $\rho_m(\xi)^*$. 
\end{corollary}

We gather the Sobolev space mapping properties of \psidos\ in the following statement. 

\begin{proposition}[see~\cite{HLP:IJM19b}] \label{prop:Sobolev-PsiDOs} 
Let $\rho(\xi)\in \bS^m(\T^n_\theta\times\R^n)$, $m\in \R$. 
\begin{enumerate}
 \item If $m\geq 0$, then $P_\rho$ extends to a unique continuous operator $P_\rho:W_2^{s+m}(\T^n_\theta)\rightarrow W_2^s(\T^n_\theta)$ for every $s\geq 0$. 
 
 \item If $m\leq 0$, then  $P_\rho$ extends to a unique continuous operator $P_\rho:W_2^{s}(\T^n_\theta)\rightarrow W_2^{s-m}(\T^n_\theta)$ for every $s\geq 0$. In particular, it extends to a bounded operator $P_\rho:L_2(\T^n_\theta)\rightarrow L_2(\T^n_\theta)$. 
\end{enumerate}
\end{proposition}

\begin{remark}\label{rmk:action-Cinfty-W2s} 
If $a\in C^\infty(\T^n_\theta)$, then $\lambda(a)$ is a zeroth order (pseudo)differential operator (\emph{cf.}\ Example~\ref{ex:PsiDO.Do-PsiDO}). Therefore, it follows from Proposition~\ref{prop:Sobolev-PsiDOs} that such operators induces bounded operators $\lambda(a):W_2^s(\T^n_\theta)\rightarrow W_2^s(\T^n_\theta)$ for all $s\geq 0$. 
\end{remark}

A detailed account on spectral theoretic properties of \psidos\ is given in~\cite{HLP:IJM19b}. In this paper we will only need the following result. 

\begin{proposition}[see~\cite{HLP:IJM19b}] \label{prop:PsiDOs.boundedness} 
 Let $\rho(\xi)\in \bS^m(\T^n_\theta\times\R^n)$, $m< 0$, and set $q=n|m|^{-1}$. 
 \begin{enumerate}
 \item  The operator $P_\rho$ is in $\sL_{q,\infty}$. 
 
 \item If $m<-n$, then $P_\rho$ is trace-class, and we have
 \begin{equation}
 \Tr \left[P_\rho\right]=\sum_{k\in \Z^n} \tau\left[\rho(k)\right]. 
 \label{eq:trace-formula-spur}
\end{equation}
\end{enumerate}
 \end{proposition}

\section{Riemannian Metrics and Laplace-Beltrami Operators on NC Tori} \label{sec:Riemannian}
In this section, we review the main definitions and properties regarding Riemannian metrics and Laplace-Beltrami operators on NC tori as introduced in~\cite{HP:JGP20, Ro:SIGMA13}. 

\subsection{Riemannian metrics} 
 In what follows, for $m\geq 1$, we denote by $\GL_m(C^\infty(\T^n_\theta))$ the group of invertible matrices in $M_m(C^\infty(\T^n_\theta))$ and denote by 
$\GL_m^+(C^\infty(\T^n_\theta))$ its subset of positive matrices. By a positive element of $M_m(C^\infty(\T^n_\theta))$ we mean a selfadjoint matrix with positive spectrum. For instance, $ \GL_1^+(C^\infty(\T^n_\theta))$ consists of positive invertible elements of $C^\infty(\T^n_\theta)$. 

Let $\sX(\T^n_\theta)$ be the free left-module over $C^\infty(\T^n_\theta)$ generated by the canonical derivations $\partial_1,\ldots, \partial_n$. This plays the role of the module of (complex) vector fields on the NC torus $C^\infty(\T^n_\theta)$ (\emph{cf}.~\cite{Ro:SIGMA13}). As  $\sX(\T^n_\theta)$ is a free module, the Hermitian metrics on $\sX(\T^n_\theta)$ are in one-to-one correspondence with matrices $h\in \GL^+_n(C^\infty(\T^n_\theta))$ (see~\cite{HP:JGP20, Ro:SIGMA13}). Namely, to any matrix $h=(h_{ij})$ in $ \GL^+_n(C^\infty(\T^n_\theta))$ corresponds the Hermitian metric, 
\begin{equation*}
 \acoupt{X}{Y}_h := \sum_{1\leq i,j\leq n}  X_ih_{ij} Y_j^*, \qquad X= \sum_i X_i \partial_i, \quad Y= \sum_j Y_j \partial_j. 
\end{equation*}

On a $C^\infty$-manifold a Riemannian metric is a Hermitian metric on the module of vector fields that takes real-values on real vector fields and real differential forms. Equivalently, in local coordinates its matrix and its inverse both have real entries. For the NC torus $\T_\theta^n$ the role of real-valued smooth functions is played by  selfadjoint elements of $C^\infty(\T^n_\theta)$. This leads to the following definition. 

\begin{definition}[\cite{HP:JGP20, Ro:SIGMA13}] A \emph{Riemannian metric} on $\T^n_\theta$ is a Hermitian metric on $\sX(\T^n_\theta)$ whose matrix $g=(g_{ij})$ is such that its entries $g_{ij}$ and the entries of its inverse $g^{-1}=(g^{ij})$ are selfadjoint elements of $C^\infty(\T^n)$. 
\end{definition}

\begin{remark}
 As usual we will often identify Riemannian metrics and their matrices. 
\end{remark}

\begin{example} 
 The standard flat metric is $g_0=(\delta_{ij})$.  
\end{example}

\begin{example}\label{ex:Riemannian.conf-flat}
 A conformal deformation of the flat metric is of the form $g=k^2g_0$ with $k\in \GL_1^+(C^\infty(\T^n_\theta))$. This is the kind of Riemannian metric considered in~\cite{CM:JAMS14, CT:Baltimore11}. 
\end{example}


\begin{example}[Self-compatible Riemannian metrics~\cite{HP:JGP20}] \label{ex:Riem.self-compatible}
 Let $g=(g_{ij})\in \GL_n^{+}(C^\infty(\T^n_\theta))$ be self-compatible in the sense of~\cite{HP:JGP20}, i.e., each entry $g_{ij}$ commute with every other entry $g_{kl}$.  
 In this case, if the entries are selfadjoint, then the inverse $g^{-1}$ has selfadjoint entries as well, and so $g$ is a Riemannian metric. This example includes the conformally flat metrics of Example~\ref{ex:Riemannian.conf-flat}. It also includes the \emph{functional metrics} of~\cite{GK:arXiv18}.  
\end{example}

\subsection{Smooth densities}\label{subsec:smooth.densities}
On ordinary manifolds densities (in the sense of differential geometry) are given in local coordinates by integration against positive smooth functions. On the NC torus 
$\T_\theta^n$ the role of positive functions is played by positive invertible elements of $C^\infty(\T^n_\theta)$, i.e., elements of $\GL_1^+(C^\infty(\T^n_\theta))$. Therefore, it is natural to think of any $\nu \in \GL_1^+(C^\infty(\T^n_\theta))$ as a smooth positive density on $\T_\theta^n$. 

Alternatively, if we think in terms of measures, we may focus on the corresponding weight $\hat{\tau}_{\nu}$ on $C(\T^n_\theta)$ given by 
\begin{equation*}
 \hat{\tau}_{\nu}(x):=(2\pi)^n\tau[x\nu], \qquad x\in C(\T^n_\theta). 
\end{equation*}
We obtain a faithful weight. The associated $L_2$-space $L_2(\T^n;\nu)$ is the completion of $C(\T^n_\theta)$ with respect to the inner-product, 
\begin{equation}
 \scal{x}{y}_\nu:=\hat{\tau}_{\nu}(xy^*)=(2\pi)^n\tau\big[y^*\nu x\big], \qquad x,y\in C(\T^n_\theta).
 \label{eq:Riemannian.inner-product-nu}  
\end{equation}
Equivalently, $L_2(\T^n;\nu)$ is the GNS representation space associated with $\hat{\tau}_{\nu}$ for the opposite algebra $C(\T^n_\theta)^{\textup{op}}$. We then define the \emph{volume} of $\T^n_\theta$ with respect of $\nu$ by 
\begin{equation*}
 \Vol_\nu(\T^n_\theta):=\hat{\tau}(1)=(2\pi)^n\tau[\nu]. 
\end{equation*}

The definition~(\ref{eq:Riemannian.inner-product-nu}) implies that 
\begin{equation*}
 \|x\|_{L_2(\T^n_\theta;\nu)} =(2\pi)^{\frac{n}{2}}\tau\big[x^*\nu x\big]^{\frac12}=\big\|(2\pi)^{\frac{n}{2}}\sqrt{\nu}x\big\|_{L_2}, \qquad x\in L_\infty(\T^n_\theta).  
\end{equation*}
Thus, a unitary isomorphism from $L_2(\T^n_\theta;\nu)$ onto $L_2(\T^n_\theta)$ is provided by the left-multiplication operator,
\begin{equation*}
 \hat{\lambda}\big(\sqrt{\nu}\big):=(2\pi)^{\frac{n}{2}}\lambda\big(\sqrt{\nu})
\end{equation*}
In particular, $L_2(\T^n;\nu)$ and $L_2(\T^n_\theta)$ agree as topological vector spaces. Thus, $L_2(\T^n_\theta;\nu)$ is just the Hilbert space $L_2(\T^n_\theta)$ with an equivalent inner product. 

As $ \hat{\lambda}(\sqrt{\nu})$ is a unitary isomorphism, we have a $*$-representation $\lambda_\nu:L_\infty(\T^n_\theta)\rightarrow \sL(L_2(\T^n_\theta;\nu))$ given by
\begin{equation*}
 \lambda_\nu(x)=\hat{\lambda}\big(\sqrt{\nu}\big)^{-1} \lambda(x) \hat{\lambda}\big(\sqrt{\nu}\big)=\lambda\big(\nu^{-\frac12}x\nu^{\frac12}\big), \qquad x \in L_\infty(\T^n_\theta). 
\end{equation*}
In other words, we obtain a $*$-representation by composing the representation $\lambda$ with the inner automorphism, 
\begin{equation}\label{eq:modular.automorphism.def}
 \sigma_\nu(x):=\nu^{-\frac12} x \nu^{\frac12}, \qquad x \in L_\infty(\T^n_\theta). 
\end{equation}

More generally, for $p\in(1,\infty)$ we denote by $L_p(\T^n_\theta;\nu)$ the Banach space $L_p(\T^n_\theta)$ equipped with the equivalent norm, 
\begin{equation*}
 \|x\|_{L_p(\T^n_\theta;\nu)}:= (2\pi)^{\frac{n}{2}}\big\| \sqrt{\nu}x\big\|_{L_p}. 
\end{equation*}
In particular, as in the $p=2$ case, the left-multiplication operator $\hat{\lambda}(\sqrt{v})=(2\pi)^{n/2}\lambda(\sqrt{v})$ provides us with 
an isometric isomorphism from $L_p(\T^n_\theta;\nu)$ onto $L_p(\T^n_\theta)$.   

\subsection{Riemannian densities} 
On an ordinary Riemannian manifold $(M^n,g)$ the Riemannian density is given in local coordinates by integration against $\sqrt{\det (g(x))}$. The volume of $(M,g)$ is then obtained as the integral of the Riemannian density. 

Following~\cite{HP:JGP20}, given any $h\in \GL_n^+(C^\infty(\T^n_\theta))$, its determinant is defined by 
\begin{equation*}
 \det (h) := \exp \big[ \Tr (\log h) \big]\in  \GL_1^+(C^\infty(\T^n_\theta)),
\end{equation*}
where $\log h$ is defined by holomorphic functional calculus and $\Tr:M_n(C^\infty(\T^n_\theta))\rightarrow C^\infty(\T^n_\theta)$ is the sum of the diagonal entries. This defines a positive invertible element of $C^\infty(\T^n_\theta)$ owing to the closedness of $C^\infty(\T^n_\theta)$ and $M_n(C^\infty(\T^n_\theta))$ under holomorphic functional calculus. We refer to~\cite{HP:JGP20} for an account on the main properties of this notion of determinant. 
In particular, if $h=(h_{ij})$ is self-compatible in the sense mentioned in Example~\ref{ex:Riem.self-compatible},  we have
 \begin{equation}
 \det (h)= \sum_{\sigma \in \fS_m} \varepsilon(\sigma) h_{1\sigma(1)} \cdots h_{m \sigma(m)}. 
 \label{eq:det.Leibniz}
\end{equation}

Given any Riemannian metric $g=(g_{ij})$ on $\T^n_\theta$ its \emph{Riemannian density} is defined by
\begin{equation*}
 \nu(g):=\sqrt{\det (g)}=\exp\big[\frac12 \Tr\left(\log (g)\right)\big]\in \GL_1^+(C^\infty(\T^n_\theta)). 
\end{equation*}
The corresponding volume with respect to $\nu(g)$ is simply denoted $\Vol_g(\T^n_\theta)$, i.e., 
\begin{equation*}
 \Vol_g(\T^n_\theta) =(2\pi)^n\tau\big[\nu(g)\big]. 
\end{equation*}

\begin{example}
 For the flat metric $g_0=(\delta_{ij})$, we have $\nu(g_0)=1$, and hence \[\Vol_{g_0}(\T^n_\theta)=(2\pi)^{n}\tau(1)=(2\pi)^n=|\T^n|.\]
\end{example}

\begin{example}
 Let $g=k^2g_0$, $k\in \GL_1^+(C^\infty(\T^n_\theta))$, be a conformal deformation of the flat Euclidean metric. Then 
 $\nu(g)=k^n$,  and  so we have $\Vol_g(\T^n_\theta)=(2\pi)^n\tau[k^n]$. 
\end{example}

\subsection{The Laplace-Beltrami operator}\label{subsec:laplace.beltrami}
Let $g=(g_{ij})$ be a Riemannian metric with inverse $g^{-1}=(g^{ij})$, and let 
$\nu\in \GL_1^+(C^\infty(\T^n_\theta))$ be a smooth positive density.   As in~\cite{HP:JGP20} we define the $C^\infty(\T^n_\theta)$-module of $1$-forms $\Omega^1(\T^n_\theta)$ as the $C^\infty(\T^n_\theta)$-linear dual of the module of vector fields $\sX(\T^n_\theta)$. That is, a $1$-form is just a $C^\infty(\T^n_\theta)$-linear map $\omega:\sX(\T^n_\theta)\rightarrow \C$. Here $\Omega^1(\T^n_\theta)$ is a right $C^\infty(\T^n_\theta)$-module; the action of $C^\infty(\T^n_\theta)$ is such that, for $\omega \in \Omega^1(\T^n_\theta)$ and $a\in C^\infty(\T^n_\theta)$, we have
\begin{equation*}
 \acou{\omega a}{X}=a\acou{\omega}{X} \qquad \forall X \in \sX(\T^n_\theta). 
\end{equation*}

A basis of $\Omega^1(\T^n_\theta)$ is given by the $1$-forms $\theta^1, \ldots, \theta^n$ such that
\begin{equation*}
 \theta^i(\partial_j)=\delta_j^i, \qquad i,j=1,\ldots, n. 
\end{equation*}
Thus, any $\omega\in  \Omega^1(\T^n_\theta)$ has a unique decomposition as $\omega=\sum \theta^i\omega_i$ with $\omega_i\in C^\infty(\T^n_\theta)$. If $a\in C^\infty(\T^n_\theta)$, then $\omega a=\sum  \theta^i\omega_ia$. We also observe that we have a left-action of $C^\infty(\T^n_\theta)$ on $\Omega^1(\T^n_\theta)$ given by
\begin{equation*}
 a\omega := \sum \theta^i (a\omega_i), \qquad a \in C^\infty(\T^n_\theta), \quad \omega =\sum  \theta^i \omega_i\in \Omega^1(\T^n_\theta). 
\end{equation*}
This turns $\Omega^1(\T^n_\theta)$ into a bimodule. In particular, the modular automorphism~(\ref{eq:modular.automorphism.def}) lifts to a linear isomorphism $\sigma_\nu:\Omega^1(\T^n_\theta)\rightarrow \Omega^1(\T^n_\theta)$ given by
\begin{equation*}
 \sigma_\nu(\omega):= \nu^{-\frac12}\omega \nu^{\frac12} = \sum \theta^i \big(\nu^{-\frac12}\omega_i \nu^{\frac12}\big), \qquad \omega =\sum  \theta^i \omega_i\in \Omega^1(\T^n_\theta). 
\end{equation*}

By duality the Hermitian metric $ \acoupt{\cdot}{\cdot}_g$ on $\sX(\T^n_\theta)$ defines a Hermitian metric $\acoupt{\cdot}{\cdot}_{g^{-1}}$ on $\Omega^1(\T^n_\theta)$ given by
\begin{equation*}
 \acoupt{\omega}{\eta}_{g^{-1}}= \sum \eta_i^* g^{ij} \omega_j, \qquad \omega=\sum \theta^i\omega_i, \quad \eta = \sum \theta^i \eta_i. 
\end{equation*}
We then endow $\Omega^1(\T^n_\theta)$ with the inner-product, 
\begin{equation}\label{eq:Riemannian.curved-one-form-inner-product}
 \scal{\omega}{\eta}_{g,\nu}:= \tau\big[\acoupt{\sigma_\nu^{-1}(\omega)}{\sigma_\nu^{-1}(\eta)}_{g^{-1}}\big]= \sum_{i,j} 
 \tau\left[   \eta_i^* \nu^{\frac12} g^{ij} \nu^{\frac12}\omega_j\right], \qquad \omega,\eta\in \Omega^1(\T^n_\theta).  
\end{equation}

The \emph{exterior differential} $d:C^\infty(\T^n_\theta)\rightarrow \Omega^1(\T^n_\theta)$ is given by
\begin{equation*}
 du := \sum \theta^i \partial_i(u), \qquad u \in C^\infty(\T^n_\theta). 
\end{equation*}
Following~\cite{HP:JGP20}, the  \emph{Laplace-Beltrami operator} $\Delta_{g,\nu}:C^\infty(\T^n_\theta) \rightarrow C^\infty(\T^n_\theta)$ is defined by
\begin{equation}\label{eq:Riemannian.Laplace-Beltrami-quadratic-form}
 \scal{\Delta_{g,\nu}u}{v}_\nu = \scal{du}{du}_{g,\nu} \qquad \text{for all}\ u,v\in C^\infty(\T^n_\theta). 
\end{equation}
If $\nu=\nu(g)$ we simply denote it by $\Delta_g$. Equivalently, we have
\begin{equation}
  \Delta_{g,\nu}u  =  -\nu^{-1} \sum_{1\leq i,j \leq n} \partial_i \big( \sqrt{\nu} g^{ij} \sqrt{\nu} \partial_j(u)\big), \qquad u \in C^\infty(\T^n_\theta).  
  \label{eq:Riemannian.Laplace-Beltrami} 
\end{equation}

\begin{example}
 For the flat metric $g_0=(\delta_{ij})$ the operator $\Delta_{g_0}$ is just the ordinary Laplacian $\Delta=-(\partial_1^2 + \cdots + \partial_n^2)$. 
\end{example}

\begin{example}
 Let $g=k^2\delta_{ij}$, $k\in \GL_1^+(C^\infty(\T^n_\theta))$, be a conformally flat metric. In this case, we have
\begin{equation*}
 \Delta_g=k^{-2} \Delta - \sum_{1\leq i \leq n} k^{-n}  \partial_i(k^{n-2}) \partial_i. 
\end{equation*}
In particular, in dimension $n=2$ we get $\Delta_g=k^{-2}\Delta$. This is reminiscent of the conformal invariance of the Laplace-Beltrami operator in dimension~2. 
\end{example}

In what follows, we set 
\begin{equation}
 |\xi|_g:=\big(\sum_{i,j} \xi_i g^{ij} \xi_j\big)^{\frac12}, \qquad \xi\in \R^n. 
\label{eq:Laplace.norm-g}
\end{equation}
Note that $|\xi|_g\in \GL_1^+(C^\infty(\T^n_\theta))$ if $\xi\neq 0$. In fact (see, e.g., \cite[Corollary~3.1]{HP:JGP20}) there are constants $c_1>0$ and $c_2>0$  such that\begin{equation}
 c_1|\xi|^2 \leq |\xi|_g^2 \leq c_2|\xi|^2 \qquad \forall \xi\in \R^n. 
 \label{eq:Laplacian.positivity-normg}
\end{equation}
Combining with~(\ref{eq:Riemannian.Laplace-Beltrami}) and Example~\ref{ex:PsiDO.Do-PsiDO} we immediately get the following result. 

\begin{proposition}[\cite{HP:JGP20}]
 The operator $\Delta_{g,\nu}$ is an elliptic 2nd order differential operator whose principal symbol is equal to 
$\sigma_\nu(|\xi|^2_g)=\nu^{-\frac12}|\xi|_g^2\nu^{\frac12}$. 
\end{proposition}

Thanks to this result the spectral theory for elliptic \psidos\ applies (see, e.g., \cite{HLP:IJM19b}). 
The following gathers the main spectral properties of  $\Delta_{g,\nu}$.  

\begin{proposition}[\cite{HP:JGP20}]\label{prop:Laplacian.properties}
 The following holds.
\begin{enumerate}
 \item The operator $\Delta_{g,\nu}$  with domain $W_2^2(\T^n_\theta)$ is selfadjoint on $L_2(\T^n;\nu)$ and has compact resolvent. 
 
 \item The spectrum of $\Delta_{g,\nu}$ consists of isolated non-negative eigenvalues with finite multiplicity. 
 
  \item Each eigenspace $\ker(\Delta_{g,\nu}-\lambda)$, $\lambda \in \Sp(\Delta_{g,\nu})$ is a finite dimensional subspace of $C^\infty(\T^n_\theta)$. In particular, $\ker \Delta_{g,\nu}=\C\cdot 1$. 
\end{enumerate}
\end{proposition}

Thanks to the selfadjointness and positivity of $\Delta_{g,\nu}$, for any  $s\in \R$ we may define the power $\Delta_{g,\nu}^s$ by Borel functional calculus. Equivalently, 
if $(u_{\ell})_{\ell\geq 0}$ is an orthonormal eigenbasis of $L_2(\T^n_\theta;\nu)$ such that $\Delta_{g,\nu}u_\ell=\lambda_\ell u_\ell$ and $u_0= 1$, then 
\begin{equation*}
 \Delta_{g,\nu}^su_0=0, \qquad \Delta_{g,\nu}^su_\ell=\lambda_\ell^s u_\ell, \quad \ell\geq 1. 
\end{equation*}
More generally, we can define complex powers $\Delta_{g,\nu}^z$ for any $z\in \C$. We similarly define the powers $(1+\Delta_{g,s})^{z}$, $z\in \C$. 

\begin{proposition}[\cite{LP:Part2, Po:JMP20}]\label{prop:powersLB} For every $s\in \R$, the powers $ \Delta_{g,\nu}^s$ and $(1+\Delta_{g,s})^{s}$ are operators in $\Psi^{2s}(\T^n_\theta)$ whose principal symbols are both equal to $\nu^{-\frac12}|\xi|_g^{2s}\nu^{\frac12}$. 
\end{proposition}

\begin{remark}
 For $\nu=\nu(g)$ the result for $ \Delta_{g,\nu}^s$ is part of the contents of~\cite[Lemma~10.3]{Po:JMP20} (see also~\cite{LP:JPDOA20}), but the proof remains valid for arbitrary smooth densities and for the powers $(1+\Delta_{g,s})^{s}$. More generally, it can be shown that the complex powers $\Delta_{g,\nu}^z$, $z\in \C$, and  
 $(1+\Delta_{g,s})^{z}$, $z\in \C$, are holomorphic families of \psidos\ of order $2z$ in the sense of~\cite{LNP:TAMS16, Po:SIGMA20} (see~\cite{LP:Part2}). 
\end{remark}

\subsection{The Sobolev spaces $W_2^s(\T^n_\theta;g,v)$} 
For $s\geq 0$, we denote by $W_2^s(\T^n_\theta;g,\nu)$ the Sobolev space with Bessel potential $(1+\Delta_{g,\nu})^{s/2}$, i.e.,
\begin{equation*}
 W_2^s(\T^n_\theta;g,\nu):=\left\{u\in L_2(\T^n_\theta;\nu);\ (1+\Delta_{g,\nu})^{\frac{s}{2}}u\in L_2(\T^n_\theta;\nu)\right\}. 
\end{equation*}
It is equipped with the Hilbert norm, 
\begin{equation*}
 \|u\|_{W^2_s(\T^n_\theta;g,\nu)}:=\big\| (1+\Delta_{g,\nu})^{\frac{s}{2}}u\|_{L_2(\T^n_\theta;\nu)}, \qquad u \in W_2^s(\T^n_\theta;g,\nu). 
\end{equation*}
For $s=0$ we recover the Hilbert space $L_2(\T^n_\theta;\nu)$. We denote by $W_2^{-s}(\T^n_\theta;g,\nu)$ the antilinear dual of $W_2^s(\T^n_\theta;g,\nu)$. The duality pairing of $W^{-s}_2(\T^n_\theta;g,\nu)$ with $W^s_2(\T^n_\theta;g,\nu)$ will be denoted by $\acou{\cdot}{\cdot}.$

Note that $\Delta_{g,\nu}^{-s/2}$ induces a bounded operator $\Delta_{g,\nu}^{-s/2}: L_2(\T^n_\theta;\nu)\rightarrow W_2^s(\T^n_\theta;g,\nu)$. By duality it further extends to a bounded operator $\Delta_{g,\nu}^{-s/2}:  W_2^{-s}(\T^n_\theta;g,\nu)\rightarrow  L_2(\T^n_\theta;\nu)$ given by
\begin{equation}\label{eq:Deltagnu-symmetricity}
 \acou{\Delta_{g,\nu}^{-\frac{s}{2}}u}{v}=\scal{u}{\Delta_{g,\nu}^{-\frac{s}{2}}v}_\nu, \qquad u\in W_2^{-s}(\T^n_\theta;g,\nu), \quad v\in  L_2(\T^n_\theta;\nu). 
\end{equation}

\begin{proposition}
 For any $s\geq 0$, the Sobolev spaces $W_2^s(\T^n_\theta)$ and $W_2^s(\T^n_\theta;g,\nu)$ agree as topological vector spaces. By duality, their antilinear duals 
  $W_2^{-s}(\T^n_\theta)$ and $W_2^{-s}(\T^n_\theta;g,\nu)$ agree as topological vector spaces as well.  
\end{proposition}
\begin{proof}
For $s=0$ we already know that $W_2^0(\T^n_\theta)=L_2(\T^n_\theta)$ and $W_2^s(\T^n_\theta;g,\nu)=L_2(\T^n_\theta;\nu)$ agree as topological vector spaces. Let $s>0$. As $(1+\Delta_{g,\nu})^{\pm s/2}$ are \psidos\ of order~$\pm s$, the composition $\Lambda^s(1+\Delta_{g,\nu})^{- s/2}$ and its inverse 
$(1+\Delta_{g,\nu})^{ s/2}\Lambda^{-s}$ are \psidos\ of order~$0$, and hence are bounded on $L_2(\T^n_\theta)$. We thus get an invertible bounded operator 
$\Lambda^s(1+\Delta_{g,\nu})^{- s/2}:L_2(\T^n_\theta;\nu)\rightarrow L_2(\T^n_\theta)$ with bounded inverse. It then follows that $W_2^s(\T^n_\theta)$ and $W_2^s(\T^n_\theta;g,\nu)$ agree as vector spaces and have equivalent norms. This gives the result. 
\end{proof}

Combining this result with the Sobolev's embeddings provided by Proposition~\ref{prop:Sobolev-embeddingLp} and Proposition~\ref{prop:Sobolev-embeddingLpC0} we obtain the following versions of those results. 

 \begin{corollary}\label{prop:Sobolev-embeddingLp-curved} 
 Let $p\in [2,\infty)$. For every $s\geq n(1/2-p^{-1})$, we have a continuous inclusion of $W_2^s(\T^n_\theta;g,\nu)$ into $L_p(\T^n_\theta;\nu)$. This inclusion is compact if the inequality is strict. 
 \end{corollary}
 
\begin{corollary}\label{prop:Sobolev-embeddingLpC0-curved}
 For any $s>n/2$, we have a compact inclusion of $W^s_2(\T^n_\theta;g,\nu)$ into $C(\T^n_\theta)$. 
\end{corollary}

\section{Cwikel-Type Estimates}\label{sec:Cwikel-flat}
In this section, we extend the Cwikel-type estimates of~\cite{MP:JMP22} to pseudodifferential operators and powers of Laplace-Beltrami operators.  

\subsection{Weak Schatten classes} \label{subsec:schatten}
We briefly review the main definitions and properties of  weak Schatten classes. We refer to~\cite{GK:AMS69, Si:AMS05} for further details. 

Let $\sH$ be a (separable) Hilbert space with inner product $\scal{\cdot}{\cdot}$. The algebra of bounded linear operators on $\sH$ is denoted $\sL(\sH)$ and its norm is denoted $\|\cdot\|$. We also denote by $\sK$ its closed ideal of compact operators on $\sH$. 

Given any compact operator $T$, we denote by $(\mu_j(T))_{j\geq 0}$ its singular value sequence, i.e., $\mu_j(T)$ is the $(j+1)$-th eigenvalue (counted with multiplicity) of the absolute value $|T|=\sqrt{T^*T}$. Recall that by the min-max principle we have
\begin{align}
 \mu_j(T)&=\min \left\{\|T_{|E^\perp}\|;\ \dim E=j\right\}.
 \label{eq:min-max} 
\end{align}

For $p\in (0,\infty)$, the weak Schatten class  $\sL_{p,\infty}$ is defined by
\begin{equation*}
 \sL_{p,\infty}:=\left\{T\in \sK; \ \mu_j(T)=\op{O}\big(j^{-\frac1p}\big)\right\}. 
\end{equation*}
 We shall also use the notation $\sL_{p,\infty}(\sH)$ whenever we need to highlight the background Hilbert space $\sH$.  
 
 The weak Schatten class  $\sL_{p,\infty}$ is a two-sided ideal. We equip it with the quasi-norm,
\begin{equation}\label{def:lp_infty_quasinorm}
 \|T\|_{\sL_{p,\infty}}:=\sup_{j\geq 0}\;(j+1)^{\frac{1}{p}}\mu_j(T), \qquad T\in \sL_{p,\infty}. 
\end{equation}
With this quasi-norm $\sL_{p,\infty}$ is a quasi-Banach ideal. In particular, we have
\begin{equation*}
 \|ATB\|_{\sL_{p,\infty}}\leq \|A\| \|T\|_{\sL_{p,\infty}} \|B\| \qquad A,B\in \sL(\sH), \ T\in \sL_{p,\infty}. 
\end{equation*}
For $p>1$ the above quasi-norm is equivalent to norm, and so in this case $\sL_{p,\infty}$ is a Banach ideal. 

We also have a version of H\"older's inequality for weak Schatten classes (see~\cite{GK:AMS69, Si:AMS05, SZ:PAMS21}). Namely, if $S\in \sL_{p,\infty}$ and $T\in \sL_{q,\infty}$ with $p^{-1}+q^{-1}=r^{-1}$, then $ST\in \sL_{r,\infty}$, and we have
\begin{equation}\label{eq:schatten.holder}
 \|ST\|_{\sL_{r,\infty}} \leq  p^{-\frac1{q}}q^{-\frac1{p}} (p+q)^{\frac1{p}+\frac1{q}}\|S\|_{\sL_{p,\infty}}  \|ST\|_{\sL_{q,\infty}}. 
\end{equation}
The above constant is sharp (see~\cite{SZ:PAMS21}). Note that for $p=q$ it is equal to $4^{\frac{1}{p}}$.

\subsection{Cwikel-type estimates on flat NC tori}  We briefly recall the Cwikel estimates on flat NC tori given in~\cite{MP:JMP22}. 

The H\"older's inequality  for the $L_p$-spaces of $\T^n_\theta$ (\emph{cf}.\ Proposition~\ref{prop:Holder}) implies the following extension result for the left-regular representation
$\lambda: L_\infty(\T^n_\theta)\rightarrow \sL(L_p(\T^n))$, $p\geq 1$. 

\begin{lemma}[\cite{Ku:TAMS58}]\label{lem:left-reg-Lp} 
Suppose that $p^{-1}+q^{-1}=r^{-1}\leq 1$.  Then the left-regular representation uniquely extends to a linear isometry, 
\begin{equation*}
 \lambda: L_p(\T^n_\theta) \longrightarrow \sL\big(L_q(\T^n_\theta),L_r(\T^n_\theta)\big).
\end{equation*}
In particular, if $p\geq 2$ and $p^{-1}+q^{-1}=2^{-1}$, then we get a linear isometry, 
\begin{equation*}
 \lambda: L_p(\T^n_\theta) \longrightarrow \sL\big(L_q(\T^n_\theta),L_2(\T^n_\theta)\big).
\end{equation*}
\end{lemma}

 Given any $x\in L_p(\T^n)$ with $p\geq 2$, the above proposition asserts that $\lambda(x)$ is a bounded operator from $L_q(\T_\theta^n)$ to $L_2(\T^n_\theta)$ with $p^{-1}+q^{-1}=1/2$. In particular, we may regard it as an unbounded operator on $L_2(\T^n_\theta)$ with domain $L_q(\T_\theta^n)$. 
 
 Note also that if $p>2$, then $q\in [2,\infty)$, and so by Proposition~\ref{prop:Sobolev-embeddingLp} we have a continuous embedding of $W_2^s(\T^n_\theta)$ into $L_q(\T^n_\theta)$ for $s\geq n(1/2-q^{-1})=n/p$. If $p=2$, then $q=\infty$ and it follows from Proposition~\ref{prop:Sobolev-embeddingLpC0} that we have continuous embedding of $W_2^s(\T^n_\theta)$ into $L_q(\T^n_\theta)$ for $s>n/2=n/p$. Therefore, we arrive at the following result. 
 
 \begin{lemma}\label{lem:Cwikel.boundedness-lambda}
 Let $x\in L_p(\T^n_\theta)$, $p\geq 2$, and assume that, either $s>n/p$, or $s=n/p$ and $p>2$. Then $\lambda(x)$ induces a bounded operator from $W_2^s(\T^n_\theta)$ to $L_2(\T_\theta^n)$, and so the operator $\lambda(x)\Delta^{-s/2}$ is bounded on $L_2(\T^n_\theta)$. 
\end{lemma}

Given any $q\geq 1$, we denote by $L_q(\T^n_\theta)^*$ the anti-linear dual of $L_q(\T^n_\theta)$, i.e., the space of continuous anti-linear forms on 
$L_q(\T^n_\theta)$. We observe that the left-regular regular representation of $L_\infty(\T^n_\theta)$ can be regarded as an isometric linear map $\lambda:L_\infty(\T^n_\theta)\rightarrow \sL(L_2(\T^n_\theta), L_2(\T^n_\theta)^*)$ such that
\begin{equation*}
 \acou{\lambda(x)u}{v}=\scal{\lambda(x)u}{v}=\tau\big[v^*xu\big], \qquad x\in L_\infty(\T^n_\theta), \quad u,v\in L_2(\T^n_\theta). 
\end{equation*}

\begin{lemma}[see~\cite{MP:JMP22}]\label{lem:left-reg-LpLqLq*}
Suppose that $p^{-1}+2q^{-1}=1$.  Then the left-regular representation uniquely extends to a continuous linear map $\lambda: L_p(\T^n_\theta) \rightarrow \sL(L_q(\T^n_\theta), L_q(\T^n_\theta)^*)$ such that
\begin{equation}\label{eq:left-reg-LpLqLq*}
 \acou{\lambda(x)u}{v}=\tau\big[v^*xu\big] \qquad \forall x\in L_p(\T^n_\theta)\quad  \forall u,v\in L_q(\T^n_\theta). 
\end{equation}
\end{lemma}

In particular, if $x\in L_p(\T^n_\theta)$ with $1\leq p<2$, then the above proposition shows that $\lambda(x)$ makes sense as a bounded operator from $L_q(\T^n_\theta)$ to $L_q(\T^n_\theta)^*$, with $p^{-1}+2q^{-1}=1$, i.e., $q^{-1}=\frac{1}{2}(1-p^{-1})$. 

For $s>0$, by duality the continuous inclusion
$W_2^s(\T^n_\theta)\hookrightarrow L_2(\T^n_\theta)$ gives rise to a continuous embedding $\iota^*:L_2(\T^n_\theta)\rightarrow W_2^{-s}(\T^n_\theta)$ such that
\begin{equation*}
 \acou{\iota^*u}{v}=\scal{u}{v}, \qquad u\in L_2(\T^n_\theta), \quad v\in W_2^{s}(\T^n_\theta). 
\end{equation*}
This allows us to identify $L_2(\T^n_\theta)$ with a dense subspace of $W_2^{-s}(\T^n_\theta)$. 

The operator $\Delta^{-s/2}$ induces a bounded operator $\Delta^{-s/2}:L_2(\T^n_\theta)\rightarrow  W^s_2(\T^n_\theta)$. It uniquely extends to a bounded operator $\Delta^{-s/2}:W_2^{-s}(\T^n_\theta) \rightarrow L_2(\T^n_\theta)$ given by
\begin{equation*}
 \bigscal{\Delta^{-\frac{s}{2}}u}{v}=\bigacou{u}{\Delta^{\frac{s}{2}}v}, \qquad u\in W_2^{-s}(\T^n_\theta), \quad v\in L_2(\T^n_\theta).
\end{equation*}

\begin{lemma}[see~\cite{MP:JMP22}]\label{lem:Boundedness.sandwich1}
 Let $x\in L_p(\T^n_\theta)$, $p\geq 1$, and assume that, either $s>n/2p$, or $s=n/2p$ and $p>1$. Then $\lambda(x)$ uniquely extends to a bounded operator $\lambda(x):W_2^s(\T^n_\theta)\rightarrow W_2^{-s}(\T^n_\theta)$, and so the operator $\Delta^{-s/2}\lambda(x)\Delta^{-s/2}$ is bounded on $L_2(\T^n_\theta)$. 
\end{lemma}
 
To sum up the operators $\lambda(x)\Delta^{-s/2}$ and $\Delta^{-s/2}\lambda(x)\Delta^{-s/2}$ are bounded provided that $x$ is in an appropriate $L_p$-space. As the following two results assert, they actually lie in suitable weak Schatten classes. 

\begin{proposition}[{\cite[Theorem~6.1]{MP:JMP22}}; see also~\cite{LeSZ:2020, MSX:CMP19}]\label{prop.specific-Cwikel}
Suppose that $p\neq 2$ and $q=\max(p,2)$, or $p=2<q$. If $x\in L_q(\T^n_\theta)$, then $\lambda(x)\Delta^{-n/2p}$ is in $\sL_{p,\infty}$, and we have 
 \begin{equation}
 \big\| \lambda(x) \Delta^{-\frac{n}{2p}}\big\|_{\sL_{p,\infty}} \leq c_n(p,q)\|x\|_{L_q}.
 \label{eq:Cwikel.specific1}  
 \end{equation}
Moreover, the best constant $c_n(p,q)$ satisfies
\begin{equation*}
 c_n(p,q)\leq c_0(p,q)\nu_0(n)^{\frac{1}{p}}, 
\end{equation*}
where we have set
\begin{gather}
 \nu_0(n):= \sup_{\lambda \geq 1} \lambda^{-\frac{n}{2}}\#\left\{k\in \Z^n\setminus 0; \ |k|\leq \sqrt{\lambda} \right\},\\
 \label{eq:c0.definition1}
 c_0(p,q):=  \left(\frac{130p}{p-2}\right)^{\frac12} \quad (p=q>2), \qquad c_0(p,q):=2^{\frac1p}(2-p)^{-\frac1p} \quad (p>2=q), \\
 \label{eq:c0.definition2}
 c_0(p,q):=2^{-\frac1{q}}q^{-\frac12}(q-2)^{\frac1q-\frac12}c_0(q,q) \quad (p=2<q). 
\end{gather}
\end{proposition}

\begin{remark}
For $p>2$ the estimate~(\ref{eq:Cwikel.specific1}) is obtained with an unspecified constant in~\cite{MSX:CMP19} as a special case of the results of~\cite{LeSZ:2020}. 
\end{remark}

\begin{remark}
 The estimate~(\ref{eq:Cwikel.specific1}) for $p=q=2$ does not hold (see Remark~\ref{rem:orlicz} below). 
\end{remark}

\begin{remark}\label{rem:general.cwikel}
The estimate~(\ref{eq:Cwikel.specific1}) is a special case of more general Cwikel estimates for operators of the form $\lambda(x)g(-i\nabla)$ with $g\in \ell_{p,\infty}(\Z^n)$ (see~\cite{MP:JMP22}). Here the operator $g(-i\nabla):L_2(\T^n_\theta)\rightarrow L_2(\T^n_\theta)$ is defined by 
   \begin{equation}\label{eq:fourier.multiplier}
        g(-i\nabla)U^k = g(k)U^k,\qquad k \in \Z^n.
    \end{equation}
In particular, if $g(k)=|k|^{-n/p}$ for $k\neq 0$ and $g(0)=0$, then $g(-i\nabla)=\Delta^{-n/2p}$. The constants $c_0(p,q)$ for $p\neq 2$ in~(\ref{eq:c0.definition1})-(\ref{eq:c0.definition2}) actually appear in~\cite{MP:JMP22} as upper bounds for the best constant for those general Cwikel estimates.   
\end{remark}

\begin{proposition}[{\cite[Theorem~6.5]{MP:JMP22}}]\label{prop:Specific-Cwikel.sandwiched}
Suppose that $p\neq 1$ and $q=\max(p,1)$, or $p=1<q$. If $x\in L_q(\T^n_\theta)$, then 
  $\Delta^{-{n}/{4p}} \lambda(x) \Delta^{-{n}/{4p}}$ is in $\sL_{p,\infty}$, and we have 
\begin{equation}
 \big\| \Delta^{-\frac{n}{4p}} \lambda(x)\Delta^{-\frac{n}{4p}}  \|_{\sL_{p,\infty}}\leq 2^{\frac{1}{p}}c_n(2p,2q)^2 \|x\|_{L_q}
 \label{eq:Cwikel.specific2}   
\end{equation}
where $c_n(2p,2q)$ is the best constant in the inequality~(\ref{eq:Cwikel.specific1}) for the pair $(2p,2q)$. If $x$ is positive, then the inequality holds without the $2^{1/p}$-factor.  
\end{proposition}

\begin{remark}\label{rem:orlicz}
For the ordinary torus $\T^n$, i.e., $\theta=0$, the estimate~(\ref{eq:Cwikel.specific2}) for $p=1$ continues to hold with a different constant if $x$ is in the Orlicz space $\LLogL(\T^n)$ (see~\cite{SZ:arXiv20,So:PLMS95}). However, it does not hold in general for $x\in L_1(\T^n)$ (see~\cite[Lemma~5.7]{LPS:JFA10}). 
\end{remark}

\subsection{Cwikel-type estimates for \psidos}\label{sec:psdo_cwikel} We shall now explain how to extend to pseudodifferential operators the 
Cwikel-type estimates provided by Proposition~\ref{prop.specific-Cwikel} and Proposition~\ref{prop:Specific-Cwikel.sandwiched}. 

Let $\rho(\xi)\in \bS^{-m}(\T^n_\theta\times\R^n)$, $m>0$. We know from Proposition~\ref{prop:PsiDOs.boundedness} that $P_\rho$ extends to a bounded operator $P_\rho:L_2(\T^n_\theta)\rightarrow W_2^{m}(\T^n_\theta)$. Let $x\in L_p(\T^n_\theta)$, $p\geq 2$, and assume that,
either $m>n/p$, or $m=n/p$ and $p>2$. In any case Lemma~\ref{lem:Cwikel.boundedness-lambda} ensures that $\lambda(x)$ induces a bounded operator from $W_2^m(\T^n_\theta)$ to $L_2(\T^n_\theta)$, and so the operator $\lambda(x)P_\rho$ is bounded on $L_2(\T^n_\theta)$. 

If $\sigma(\xi)\in\bS^{-m}(\T^n_\theta\times\R^n)$, then  by Proposition~\ref{prop:Adjoint} the formal adjoint $P_{\sigma}^*$ of $P_{\sigma}$ is $P_{\sigma^\star}$, where $\sigma^\star(\xi)\in \bS^{-m}(\T^n_\theta\times\R^n)$ is given by~(\ref{eq:Adjoints.symbol-star}). In particular, $P_{\sigma}^*$ uniquely extends to a bounded operator $P_{\sigma}^*:L_2(\T^n_\theta)\rightarrow W_2^m(\T^n_\theta)$. It then follows that $P_{\sigma}$ uniquely extends to a bounded $P_{\rho_2}:W_2^{-m}(\T^n_\theta)\rightarrow L_2(\T^n_\theta)$ such that
\begin{equation}\label{eq:Cwikel-Prho.dual-action-Prho}
 \scal{P_{\rho_2} u}{v}=\acou{u}{P_{\rho_2}^*v}, \qquad u\in W_2^{-m}(\T^n_\theta), \qquad v\in L_2(\T^n_\theta). 
\end{equation}

If $x\in L_p(\T^n_\theta)$, $p\geq 1$, and either $m>n/2p$, or $m=n/2p$ and $p>1$, then Lemma~\ref{lem:left-reg-LpLqLq*} ensures that 
$\lambda(x)$ uniquely extends to a bounded operator $\lambda(x):W_2^s(\T^n_\theta)\rightarrow W_2^{-s}(\T^n_\theta)$. It then follows that the operator 
$P_{\sigma}^*\lambda(x)P_\rho$ is bounded on $L_2(\T^n_\theta)$. 

In what follows, we denote by $\Pi_0$ the orthogonal projection onto $\ker \Delta=\C\cdot 1$. That is, 
\begin{equation*}
 \Pi_0u=\tau(u)1, \qquad u\in L_2(\T^n_\theta). 
\end{equation*}
Note that $\Pi_0$ is a rank 1 operator.

We have the following extension to \psidos\ of the Cwikel-type estimates~(\ref{eq:Cwikel.specific1}) and~(\ref{eq:Cwikel.specific2}). 

\begin{proposition}[Cwikel Estimates. \psido\ Version]\label{prop.specific-Cwikel-Prho}
 The following hold. 
\begin{enumerate}
 \item Suppose that $p\neq 2$ and $q=\max(p,2)$,  or $p=2<q$. If $x \in L_q(\T^n_\theta)$ and $\rho(\xi)$ is in $\bS^{-n/p}(\T^n_\theta\times\R^n)$, then 
 $\lambda(x) P_\rho\in \sL_{p,\infty}$, and we have 
  \begin{equation}
 \big\| \lambda(x) P_\rho\big\|_{\sL_{p,\infty}} \leq \big(c_n(p,q)+1\big)\|P_\rho\|_{\sL(L_2,W_2^{n/p})} \|x\|_{L_q}, 
 \label{eq:specific-Cwikel-Prho}
 \end{equation}
 where $c_n(p,q)$ is the best constant in~(\ref{eq:Cwikel.specific1}). 
 
 \item Suppose that, $p\neq 1$ and $q=\max(p,1)$,  or $p=1<q$. If $x\in L_q(\T^n_\theta)$, and $\rho(\xi)$ and $\sigma(\xi)$ are in $\bS^{-n/2p}(\T^n_\theta\times\R^n)$, then $P_{\sigma}^*\lambda(x) P_\rho\in \sL_{p,\infty}$, and we have 
 \begin{equation}
 \big\| P_{\sigma}^*\lambda(x) P_\rho\big\|_{\sL_{p,\infty}} \leq 2^{\frac{1}{p}}\big(c_n(2p,2q)+1\big)^2 \|P_\sigma\|_{\sL(L_2,W_2^{n/2p})} \|P_\rho\|_{\sL(L_2,W_2^{n/2p})} \|x\|_{L_q},
 \label{eq:specific-Cwikel-PrhoPsigma}
 \end{equation} 
 where $c_n(2p,2q)$ is the best constant in~(\ref{eq:Cwikel.specific1}) for the pair $(2p,2q)$. Moreover, if $x$ is positive and $\rho=\sigma$, then the inequality holds without the $2^{1/p}$-factor. 
\end{enumerate}
\end{proposition}
\begin{proof}
Suppose that, either $p\neq 2$ and $q=\max(p,2)$,  or $p=2<q$, and let $x\in L_q(\T^n_\theta)$. Let us first prove~(\ref{eq:specific-Cwikel-Prho}) for $\rho(\xi)=(1+|\xi|^2)^{-n/2p}$, i.e., $P_\rho=\Lambda^{-n/p}$. In the notation of~(\ref{eq:fourier.multiplier}) we have $\Lambda=g(-i\nabla)$ with $g(k)=(1+|k|^2)^{-n/4p}$, and so by Remark~\ref{rem:general.cwikel} and the results of~\cite{MP:JMP22} the operator $\lambda(x)\Lambda^{-n/p}$ is in $\sL_{p,\infty}$. 

In fact, we have
\begin{equation}\label{eq:one.sided.cwikel.kernel.splitting}
 \lambda(x)\Lambda^{-\frac{n}{p}}=  \lambda(x)(1-\Pi_0)\Lambda^{-\frac{n}{p}} +  \lambda(x)\Pi_0\Lambda^{-\frac{n}{p}}
 =  \lambda(x)\Delta^{-\frac{n}{2p}}\cdot \Delta^{\frac{n}{2p}}\Lambda^{-\frac{n}{p}} + \lambda(x)\Pi_0. 
\end{equation}
Here $\Delta^{\frac{n}{2p}}\Lambda^{-\frac{n}{p}}$ is a bounded operator of norm~$\leq 1$. Thus, by Proposition~\ref{prop.specific-Cwikel} the operator $\lambda(x)\Delta^{-\frac{n}{2p}}\cdot \Delta^{\frac{n}{2p}}\Lambda^{-\frac{n}{p}}$ is in $\sL_{p,\infty}$, and we have
\begin{equation}\label{eq:bound.on.x.Delta.n2p}
 \big\| \lambda(x)\Delta^{-\frac{n}{2p}}\cdot \Delta^{\frac{n}{2p}}\Lambda^{-\frac{n}{p}}\big\|_{\sL_{p,\infty}} \leq 
  \big\| \lambda(x)\Delta^{-\frac{n}{2p}}\big\|_{\sL_{p,\infty}}  \cdot  \big\| \Delta^{\frac{n}{2p}}\Lambda^{-\frac{n}{p}}\big\| \leq c_n(p,q)\|x\|_{L_q}.  
\end{equation}
Moreover, as $ \lambda(x)\Pi_0$ is an operator of rank~$\leq 1$, it is contained in the weak Schatten class $\sL_{p,\infty}$. Furthermore, by the min-max principle~(\ref{eq:min-max}) and the very definition of the weak Schatten class norm~(\ref{def:lp_infty_quasinorm}) we have
\begin{equation*}
 \big\|  \lambda(x)\Pi_0\big\|_{\sL_{p,\infty}} = \mu_0\left(\lambda(x)\Pi_0\right) 
 = \big\|  \lambda(x)\Pi_0\big\|= \|\tau\| \|x\|_{L_2} \leq \|x\|_{L_q}. 
\end{equation*}
Combining this with~(\ref{eq:one.sided.cwikel.kernel.splitting}) and~(\ref{eq:bound.on.x.Delta.n2p}) then gives
\begin{equation}\label{eq:bessel.potential.cwikel1}
 \big\| \lambda(x)\Lambda^{-\frac{n}{p}}\big\|_{\sL_{p,\infty}} \leq \big(c_n(p,q)+1\big)\|x\|_{L_q}. 
\end{equation}
 
 Let $\rho(\xi)\in \bS^{-n/p}(\T^n_\theta\times\R^n)$. As the operator $P_\rho:L_2(\T^n_\theta)\rightarrow W_2^{n/p}(\T^n_\theta)$ is bounded, the composition $\Lambda^{n/p}P_\rho$ is bounded on $L_2(\T^n_\theta)$. Moreover, the very definition~(\ref{eq:sobolev.norm.definition}) of the norm of  $W_2^{n/p}(\T^n_\theta)$ implies that $\|\Lambda^{n/p}P_\rho\|= \|P_\rho\|_{\sL(L_2,W_2^{n/p})}$. As $\lambda(x)P_\rho= \lambda(x)\Lambda^{-n/p}\cdot \Lambda^{n/p}P_\rho$, we deduce that 
that $\lambda(x)P_\rho\in \sL_{p,\infty}$, and by using~(\ref{eq:bessel.potential.cwikel1}) we get
\begin{equation*}
  \big\| \lambda(x) P_\rho\big\|_{\sL_{p,\infty}} \leq  \big\| \lambda(x) \Lambda^{-\frac{n}{p}}\big\|_{\sL_{p,\infty}} \big\| \Lambda^{\frac{n}{p}}P_\rho\| \leq 
  \big( c_n(p,q)+1\big) \|x\|_{L_q} \|P_\rho\|_{\sL(L_2,W_2^{n/p})}. 
\end{equation*}
This proves~(\ref{eq:specific-Cwikel-Prho})

The proof of the 2nd part is similar to the proof of the 2nd part of~\cite[Theorem~6.5]{MP:JMP22}. Suppose now that, $p\neq 1$ and $q=\max(p,1)$,  or $p=1<q$. Let $x\in L_q(\T^n_\theta)$ and assume that $\rho(\xi)$ and $\sigma(\xi)$ are symbols in $\bS^{-n/2p}(\T^n_\theta\times\R^n)$. We may write $x=yz$ with $y,z\in L_{2q}(\T^n_\theta)$ such that $\|y\|_{L_{2q}}=\|z\|_{L_{2q}}=(\|x\|_{L_q})^{1/2}$ (see, e.g., \cite[Lemma~6.4]{MP:JMP22}). By the first part $\lambda(y)P_\sigma$ and $\lambda(z)P_\rho$ are in the weak Schatten class $\sL_{2p,\infty}$. Moreover, by using~(\ref{eq:Cwikel-Prho.dual-action-Prho}) and arguing 
along similar lines as that of the proof of~\cite[Lemma~6.4]{MP:JMP22} it can be shown that
\begin{equation*}
 P_\sigma^*\lambda(x)P_\rho=\big(\lambda(y^*)P_\sigma\big)^*\lambda(z)P_\rho. 
\end{equation*}
Thus, by the H\"older inequality~(\ref{eq:schatten.holder}) the operator $ P_\sigma^*\lambda(x)P_\rho$ is in $\sL_{p,\infty}$, and we have
\begin{equation}\label{eq:PrhoPsigma.holder}
  \big\| P_{\sigma}^*\lambda(x) P_\rho\big\|_{\sL_{p,\infty}}  \leq 2^{\frac{1}{p}}
  \big\|\left(\lambda(y^*)P_{\sigma}\right)^*\big\|_{\sL_{2p,\infty}} \big\| \lambda(z)P_\rho\big\|_{\sL_{2p,\infty}}. 
\end{equation}
By using~(\ref{eq:specific-Cwikel-Prho}) and the fact that $\|y^*\|_{L_{2q}} =\|y\|_{L_{2q}}=\|z\|_{L_{2q}}=(\|x\|_{L_q})^{1/2}$ we get
\begin{gather*}
  \big\| \lambda(z)P_\rho\big\|_{\sL_{2p,\infty}}\leq \big(c_n(p,q)+1\big)\|z\|_{L_{2q}}= \big(c_n(p,q)+1\big)\big(\|x\|_{L_{2q}}\big)^{\frac12},\\
 \big\|\left(\lambda(y^*)P_{\sigma}\right)^*\big\|_{\sL_{2p,\infty}}  =\big\|\lambda(y^*)P_{\sigma}\big\|_{\sL_{2p,\infty}}  
 \leq \big(c_n(p,q)+1\big)\|y^*\|_{L_{2q}}= \big(c_n(p,q)+1\big)\big(\|x\|_{L_{2q}}\big)^{\frac12}. 
\end{gather*}
Combining these inequalities with~(\ref{eq:PrhoPsigma.holder}) gives~(\ref{eq:specific-Cwikel-PrhoPsigma}). 

If $x$ is positive, then we may take $y=z=\sqrt{x}$. If in addition $\sigma=\rho$, then~(\ref{eq:PrhoPsigma.holder}) becomes $P_\sigma^*\lambda(x)P_\rho=(\lambda(\sqrt{x})P_\sigma)^*(\lambda(\sqrt{x})P_\sigma)$, and so we get
\begin{equation*}
  \big\| P_{\sigma}^*\lambda(x) P_\rho\big\|_{\sL_{p,\infty}}= \big\| \big(\lambda(\sqrt{x})P_\sigma\big)^*\big(\lambda(\sqrt{x})P_\sigma\big)\big\|_{\sL_{p,\infty}}
  =  \big\|\lambda(\sqrt{x})P_\sigma\big\|_{\sL_{2p,\infty}}^2.
\end{equation*}
Combining this with the inequality~(\ref{eq:specific-Cwikel-Prho}) for $z=\sqrt{x}$ gives the inequality~(\ref{eq:specific-Cwikel-PrhoPsigma}) without the $2^{1/p}$-factor. The proof is complete. 
\end{proof}

 
 As a first application of the Cwikel-type estimates~(\ref{eq:specific-Cwikel-Prho})--(\ref{eq:specific-Cwikel-PrhoPsigma}) we shall now explain how they enable to get $L_p$-versions of the trace formula~(\ref{eq:trace-formula-spur}). 

Any operator in $\sL_{p,\infty}$ with $p<1$ is trace class. Thus,  Proposition~\ref{prop.specific-Cwikel-Prho} ensures that if $x\in L_2(\T^n_\theta)$ and $\rho(\xi)\in \bS^{-m}(\T^n_\theta\times \R^n)$ with $m>n$, then the operator $\lambda(x)P_\rho$ is trace-class. Likewise, if $x\in L_1(\T^n_\theta)$ and $\rho(\xi)$ and $\sigma(\xi)$ are symbols in $\bS^{-m}(\T^n_\theta\times\R^n)$ with $m>n/2$, then Proposition~\ref{prop.specific-Cwikel-Prho} also ensures that $P_\sigma^* \lambda(x) P_\rho$ is trace-class. We then have the following extension of the trace formula~(\ref{eq:trace-formula-spur}). 

\begin{proposition}[$L_p$-Trace Formulas]\label{prop:Lp-trace-formula}
 The following hold.
\begin{enumerate}
\item   If $\rho(\xi)\in \bS^{-m}(\T^n_\theta\times\R^n)$, $m>n$, and $x\in L_2(\T^n_\theta)$, then
\begin{equation}\label{trace_class_bound}
  \Tr\big[ \lambda(x)P_\rho\big] = \sum_{k\in \Z^n} \tau\big[x\rho(k)\big]. 
\end{equation}

\item If  $x\in L_1(\T^n_\theta)$ and $\rho(\xi)$ and $\sigma(\xi)$ are  in $\bS^{-m}(\T^n_\theta\times \R^n)$, $m>n/2$, then
\begin{equation}
   \Tr\big[ P_\sigma^*\lambda(x)P_\rho\big] = 
  \sum_{k\in \Z^n} \tau\big[x\rho(k)\sigma(k)^*\big]. 
   \label{eq:Cwikel-Prho.trace-Psigma*xPrho} 
\end{equation}
In particular, we have 
\begin{equation}
 \Tr\big[ P_\rho^*\lambda(x)P_\rho\big] = 
  \sum_{k\in \Z^n} \tau\big[x|\rho(k)|^2\big].
  \label{eq:Cwikel-Prho.trace-Pr*xPr} 
\end{equation}
\end{enumerate}
 \end{proposition}
 \begin{proof}
  Let $\rho(\xi)\in \bS^{-m}(\T^n_\theta\times\R^n)$, $m>n$, and $x\in L_2(\T^n_\theta)$.  As $\{U^k;\ k\in \Z^n\}$ is an orthonormal basis of $L_2(\T^n_\theta)$ and $\rho(\xi)$ takes values in $C^\infty(\T^n_\theta)$, we have 
    \begin{equation}\label{eq:Cwikel-Prho.trace-Uk}
         \Tr\big[\lambda(x)P_{\rho}\big] = \sum_{k\in \Z^n}\scal{ \lambda(x)P_\rho(U^k)}{U^k}= \sum_{k\in \Z^n}\scal{ x\rho(k)U^k}{U^k}= \sum_{k\in \Z^n}\tau\big[ (U^k)^*x\rho(k)U^k\big].  
     \end{equation}
Recall that $\tau$ is a unitarily invariant functional on $L_1(\T^n_\theta)$ and the basis elements $U^k$ are unitary. Thus, $\tau[ (U^k)^*x\rho(k)U^k]= \tau[x\rho(k)]$, and so we get~(\ref{trace_class_bound}). 

Suppose now that $x\in L_1(\T^n_\theta)$ and $\sigma(\xi)$ are symbols in $\bS^{-m}(\T^n_\theta\times\R^n)$ with $m>n/2$.  In the same way as in~(\ref{eq:Cwikel-Prho.trace-Uk}) we have  
\begin{equation}
 \Tr\big[ P_\sigma^*\lambda(x)P_\rho\big] = \sum_{k\in \Z^n} \scal{P_\sigma^* \lambda(x)P_\rho(U^k)}{U^k}. 
 \label{eq:Cwikel-Prho.trace-UkPsr}
\end{equation}
In view of~(\ref{eq:left-reg-LpLqLq*}) and~(\ref{eq:Cwikel-Prho.dual-action-Prho}) we have
\begin{equation}
  \scal{P_\sigma \lambda(x)P_\rho(U^k)}{U^k}=\acou{ \lambda(x)P_\rho(U^k)}{P_\sigma(U^k)}=\tau\left[ [P_\sigma(U^k)]^*xP_\rho(U^k)\right]. 
  \label{eq:Cwikel-Prho.pairingPsUkxPrUk}
\end{equation}
As $P_\rho(U^k)=\rho(k)U^k$ and $P_\sigma(U^k)=\sigma(k)U^k$, we get 
\begin{equation*}
  \scal{P_\sigma \lambda(x)P_\rho(U^k)}{U^k}=\tau\left[(U^k)^*\sigma(k)^*x\rho(k)U^k\right]=\tau\big[x\rho(k)\sigma(k)^*\big].
\end{equation*}
Combining this with~(\ref{eq:Cwikel-Prho.trace-UkPsr}) then gives~(\ref{eq:Cwikel-Prho.trace-Psigma*xPrho}). The proof is complete.  
\end{proof}
   
\subsection{Curved Cwikel estimates}\label{subsec:Cwikel-curved}
Let $g=(g_{ij})$ be a Riemannian metric  and $\nu\in \GL_1^+(C^\infty(\T^n_\theta))$ a smooth positive density on $\T^n_\theta$. We shall now explain how to extend the Cwikel estimates of Proposition~\ref{prop.specific-Cwikel} and Proposition~\ref{prop:Specific-Cwikel.sandwiched} to powers of the Laplace-Beltrami operator $\Delta_{g,\nu}$. Recall that $\Delta_{g,\nu}^{-s/2}$, $s>0$, gives rise  to bounded operators $\Delta_{g,\nu}^{-s/2}:L_2(\T^n_\theta;\nu)\rightarrow W_2^{s}(\T^n_\theta;g,\nu)$ and $\Delta_{g,\nu}^{-s/2}: W_2^{-s}(\T^n_\theta;g,\nu)\rightarrow L_2(\T^n_\theta;\nu)$.  
 
As mentioned in Subsection~\ref{subsec:smooth.densities}, the left-mutiplication operator $\hat{\lambda}(\sqrt{\nu})$ provides us with an isometric isomorphism from $L_p(\T^n_\theta;\nu)$ onto 
$L_p(\T^n_\theta)$ for any $p\in [1,\infty)$. Combining this with Lemma~\ref{lem:left-reg-Lp} for $p^{-1}+q^{-1}=r^{-1}\leq 1$ we get an isometric linear map,  
\begin{equation*}
 \lambda_\nu: L_p(\T^n) \longrightarrow \sL\big(L_q(\T^n_\theta;\nu), L_r(\T^n_\theta;\nu)\big), \qquad  
\end{equation*}where
\begin{equation}\label{eq:Riem.lambdanu}
 \lambda_\nu(x)=\hat{\lambda}\big(\sqrt{\nu}\big)^{-1}\lambda(x)\hat{\lambda}\big(\sqrt{\nu}\big)= \lambda\big(\sigma_\nu(x)\big), \qquad x\in L_p(\T^n). 
\end{equation}
Here $\sigma_\nu(x)$ is the inner automorphism~(\ref{eq:modular.automorphism.def}). In particular, if $x\in L_p(\T^n_\theta)$, $p\geq 2$, and $p^{-1}+q^{-1}=1/2$, then $\lambda_\nu(x)$ maps continuously $L_q(\T^n_\theta;\nu)$ to $L_2(\T^n_\theta;\nu)$.  

We are interested in operators  of the form
$\lambda_\nu(x)\Delta_{g,\nu}^{-s/2}$ and $\Delta_{g,\nu}^{-s/2}\lambda_\nu(x)\Delta_{g,\nu}^{-s/2}$ with $s>0$. For the applications we have in mind, it is important to regard them as operators on $L_2(\T^n_\theta;\nu)$ rather than as operators on $L_2(\T^n_\theta)$.  

In the same way as with Lemma~\ref{lem:Cwikel.boundedness-lambda}, by combining this with the continuous embeddings provided by Proposition~\ref{prop:Sobolev-embeddingLp} and Proposition~\ref{prop:Sobolev-embeddingLpC0} we get the following result. 

\begin{lemma}\label{lem:Cwikel.boundedness-lambda-curved}
 Let $x\in L_p(\T^n_\theta)$, $p\geq 2$, and assume that, either $s>np^{-1}$, or $s=np^{-1}$ and $p>2$. Then $\lambda_\nu(x)$ induces a bounded operator 
 $\lambda_\nu(x):W_2^s(\T^n_\theta;g,\nu)\rightarrow L_2(\T_\theta^n;\nu)$, and hence $\lambda_\nu(x) \Delta_{g,\nu}^{-s/2}$ is bounded on $L_2(\T^n_\theta;\nu)$. 
\end{lemma}

\begin{remark}
 If $x\in L_p(\T^n_\theta)$, $p\geq 2$, and $s$ are as in Lemma~\ref{lem:Cwikel.boundedness-lambda-curved}, then~(\ref{eq:Riem.lambdanu}) implies that
 \begin{equation}\label{eq:lambda-to-lambda-nu-onesided}
 \lambda_\nu(x)\Delta_{g,\nu}^{-s/2}=\lambda\big(\nu^{-\frac{1}{2}}x\nu^{\frac{1}{2}}\big)\Delta_{g,\nu}^{-s/2}.  
\end{equation}
\end{remark}

Let $L_q(\T^n_\theta;\nu)^*$, $q\geq 1$, be the anti-linear dual of $L_q(\T^n_\theta;\nu)$. By duality we have an isometric isomorphism 
$\hat{\lambda}(\sqrt{\nu})^*: L_q(\T^n_\theta)^*\rightarrow L_q(\T^n_\theta;\nu)^*$ given by 
\begin{equation*}
 \bigacou{\hat{\lambda}\big(\sqrt{\nu}\big)^*u}{v}=\bigacou{u}{\hat{\lambda}\big(\sqrt{\nu}\big)v}, \qquad u\in L_q(\T^n_\theta)^*, \quad v\in L_q(\T^n_\theta;\nu). 
\end{equation*}
Therefore, by Lemma~\ref{lem:left-reg-Lp}, if $p^{-1}+2q^{-1}=1$, then we have an isometric linear map, 
\begin{equation*}
 \lambda_\nu: L_p(\T^n_\theta) \longrightarrow \sL\big(L_q(\T^n_\theta;\nu), L_q(\T^n_\theta;\nu)^*\big), \qquad \lambda_\nu(x):=\hat{\lambda}\big(\sqrt{\nu}\big)^*\lambda(x) \hat{\lambda}\big(\sqrt{\nu}\big). 
\end{equation*}
More precisely, if $x\in L_p(\T^n_\theta)$ and $u,v\in L_q(\T^n_\theta;\nu)$, then 
\begin{equation}\label{eq:Integration.lambdanu-dual} 
 \acou{\lambda_\nu(x)u}{v}= \bigacou{\lambda(x)\hat{\lambda}\big(\sqrt{\nu}\big)u}{\hat{\lambda}\big(\sqrt{\nu}\big)v} =
 (2\pi)^n\bigacou{\lambda(x)\big(\nu^{\frac12} u\big)}{\nu^{\frac12}v} = (2\pi)^n\tau\big[v^*\nu^{\frac12}x\nu^{\frac12}u\big]. 
\end{equation}
In particular, if $x\in L_\infty(\T^n_\theta)$, then
\begin{equation*}
 \acou{\lambda_\nu(x)u}{v}=(2\pi)^n\tau\left[v^*\nu\big(\nu^{-\frac12}x\nu^{\frac12}\big)u\right]=\bigscal{\lambda\big(\nu^{-\frac12}x\nu^{\frac12}\big)u}{v}_\nu. 
\end{equation*}
Moreover, in the same with Lemma~\ref{lem:Boundedness.sandwich1}, by combining this with continuous embeddings provided by Proposition~\ref{prop:Sobolev-embeddingLp} and Proposition~\ref{prop:Sobolev-embeddingLpC0} we obtain the following result.

\begin{lemma}\label{lem:Boundedness.sandwich1-curved}
 Let $x\in L_p(\T^n_\theta)$, $p\geq 1$, and assume that, either $s>n/2p$, or $s=n/2p$ and $p>1$. Then $\lambda_\nu(x)$ uniquely extends to a bounded operator $\lambda_\nu(x):W_2^s(\T^n_\theta;g,\nu)\rightarrow W_2^{-s}(\T^n_\theta,g,\nu)$, and hence $\Delta_{g,\nu}^{-s/2}\lambda_\nu(x)\Delta_{g,\nu}^{-s/2}$ is bounded on $L_2(\T^n_\theta;\nu)$. 
\end{lemma}

\begin{remark}\label{rmk:lambda-to-lambda-nu-twosided}
 Let $p$ and $s$ be as in Lemma~\ref{lem:Boundedness.sandwich1-curved}. The linear map $x\rightarrow \Delta_{g,\nu}^{-s/2}\lambda_\nu(x)\Delta_{g,\nu}^{-s/2}$ is continuous from $L_p(\T^n_\theta)$ to $\sL(L_2(\T^n_\theta;\nu))$. If $x\in L_\infty(\T^n_\theta)$, then by~(\ref{eq:Riem.lambdanu}) we have 
 $ \Delta_{g,\nu}^{-s/2}\lambda_\nu(x)\Delta_{g,\nu}^{-s/2}= \Delta_{g,\nu}^{-s/2}\lambda\big(\nu^{-\frac{1}{2}}x\nu^{\frac{1}{2}}\big) \Delta_{g,\nu}^{-s/2}$. The density of  $L_\infty(\T^n_\theta)$ into $L_p(\T^n_\theta)$ then implies that
 \begin{equation}\label{eq:lambda-to-lambda-nu-twosided}
 \Delta_{g,\nu}^{-\frac{s}{2}}\lambda_\nu(x)\Delta_{g,\nu}^{-\frac{s}{2}}= \Delta_{g,\nu}^{-\frac{s}{2}}\lambda\big(\nu^{-\frac{1}{2}}x\nu^{\frac{1}{2}}\big) \Delta_{g,\nu}^{-\frac{s}{2}} \qquad \forall x\in L_p(\T^n_\theta).  
\end{equation}
\end{remark}

For $p>0$ we denote by $\sL_{p,\infty}^{(\nu)}$ the weak Schatten class $\sL_{p,\infty}(L_2(\T^n_\theta;\nu))$. We shall keep using the notation $\sL_{p,\infty}$ for the weak Schatten class $\sL_{p,\infty}(L_2(\T^n_\theta))$.  As $L_2(\T^n_\theta;\nu)$ and $L_2(\T^n_\theta)$ are the same Hilbert space with equivalent sinner products, $\sL_{p,\infty}^{(\nu)}$ and $\sL_{p,\infty}$ are the same quasi-Banach ideal with equivalent quasi-norms. In fact, as $\hat{\lambda}(\sqrt{\nu}):L_2(\T^n_\theta;\nu)\rightarrow L_2(\T^n_\theta)$ is a unitary isomorphism, an isometric isomorphism from $\sL_{p,\infty}^{(\nu)}$ onto $\sL_{p,\infty}$ is given by the map 
$T \longrightarrow \hat{\lambda}(\sqrt{\nu})T \hat{\lambda}(\sqrt{\nu})^{-1}= {\lambda}(\sqrt{\nu})T {\lambda}(\sqrt{\nu})^{-1}$.  In particular, 
\begin{equation*}
 \|T\|_{\sL_{p,\infty}^{(\nu)}}= \big\| {\lambda}\big(\sqrt{\nu}\big)T {\lambda}\big(\sqrt{\nu}\big)^{-1}\big\|_{\sL_{p,\infty}} \qquad \forall T\in  \sL_{p,\infty}^{(\nu)} . 
\end{equation*}

By Proposition~\ref{prop:powersLB} the power $\Delta_{g,\nu}^{-s}$ is a \psido\ of order~$-2s$. Thus, if we disregard inner products, then by using~(\ref{eq:lambda-to-lambda-nu-onesided}) and~(\ref{eq:lambda-to-lambda-nu-twosided}) and Proposition~\ref{prop.specific-Cwikel-Prho} we see that $\lambda_\nu(x)\Delta_{g,\nu}^{-n/2p}$ and $\Delta_{g,\nu}^{-n/4p}\lambda(x)\Delta_{g,\nu}^{-n/4p}$ are operators in $\sL_{p,\infty}$ if $x$ is in a suitable $L_q$-space, and we have explicit upper bounds for their $\sL_{p,\infty}$-quasi-norm. For our purpose, we regard these operators as operators on $L_2(\T^n_\theta;\nu)$, and so we seek for estimates for their $\sL_{p,\infty}^{(\nu)}$-quasi-norms.

\begin{theorem}[Curved Cwikel Estimates] \label{thm:curved-Cwikel} 
The following hold. 
 \begin{enumerate}
 \item Suppose that $p\neq 2$ and $q=\max(p,2)$, or $p=2<q$. If $x\in L_q(\T^n_\theta)$, then $\lambda_\nu(x)\Delta_{g,\nu}^{-n/2p}$ is in $\sL_{p,\infty}^{(\nu)}$, and we have
\begin{equation}
  \left\| \lambda_\nu(x)\Delta_{g,\nu}^{-\frac{n}{2p}}\right\|_{\sL_{p,\infty}^{(\nu)}} \leq C_{g,\nu}(p,q)\big\|x\big\|_{L_q}. 
 \label{eq:specific-Cwikel-Deltagnu}
\end{equation}
Moreover, the best constant $C_{g,\nu}(p,q)$ satisfies 
\begin{equation*}
 C_{g,\nu}(p,q)\leq  (2\pi)^{\frac{n}{2}} \big(c_n(p,q)+1)\|\nu\|^{\frac12} \big\|\Lambda^{\frac{n}{p}} \Delta_{g,\nu}^{-\frac{n}{2p}}\big\|_{\sL(L_2(\T^n_\theta;\nu),L_2(\T^n_\theta))}, 
\end{equation*}
where $c_n(p,q)$ is the best constant in the estimate~(\ref{eq:Cwikel.specific1}).

\item Suppose that $p\neq 1$ and $q=\max(p,1)$, or $p=1<q$. If $x\in L_q(\T^n_\theta)$, then the operator $\Delta_{g,\nu}^{-n/4p}\lambda_\nu(x)\Delta_{g,\nu}^{-n/4p}$ is in $\sL_{p,\infty}^{(\nu)}$, and we have
\begin{equation}
  \left\| \Delta_{g,\nu}^{-\frac{n}{4p}}\lambda_\nu(x) \Delta_{g,\nu}^{-\frac{n}{4p}}\right\|_{\sL_{p,\infty}^{(\nu)}} \leq 2^{\frac1p}C_{g,\nu}(2p,2q)^2
  \big\|x\big\|_{L_q}. 
\label{eq:specific-Cwikel-Deltagnu-2}
\end{equation}
where $C_{g,\nu}(2p,2q)$ is the best constant in the estimate~(\ref{eq:specific-Cwikel-Deltagnu}) for the pair $(2p,2q)$. Furthermore, the inequality holds without the $2^{1/p}$-factor if $x$ is positive. 
\end{enumerate}
\end{theorem}
\begin{proof}
Suppose that $p\neq 2$ and $q=\max(p,2)$, or $p=2<q$. The operator $A:=\Lambda^{n/p}\Delta_{g,\nu}^{-n/2p}$ is a \psido\ of order~$0$, since 
$\Lambda^{n/p}$ and $\Delta_{g,\nu}^{-n/2p}$ are \psidos\ of order $n/p$ and $-n/p$, respectively. Thus, $A$ is bounded on $L_2(\T^n_\theta)$. We actually regard it as a bounded operator $A:L_2(\T^n_\theta;\nu)\rightarrow L_2(\T^n_\theta)$. Let $x\in L_q(\T^n_\theta)$. We already know that $\lambda_\nu(x)\Delta_{g,\nu}^{-n/2p}$ is in $\sL_{p,\infty}^{(\nu)}$. Moreover, we have 
\begin{align*}
 \lambda_\nu(x)\Delta_{g,\nu}^{-\frac{n}{2p}} &=  \hat{\lambda}\big(\sqrt{\nu}\big)^{-1} \lambda(x) 
  \hat{\lambda}\big(\sqrt{\nu}\big)\Lambda^{-\frac{n}{p}} \Lambda^{-\frac{n}{p}}\Delta_{g,\nu}^{-\frac{n}{2p}}\\ 
  &= (2\pi)^{\frac{n}{2}}\hat{\lambda}\big(\sqrt{\nu}\big)^{-1}  \lambda\big(x \sqrt{\nu}\big) \Lambda^{-\frac{n}{p}} A.  
\end{align*}
As $\hat{\lambda}(\sqrt{\nu})^{-1}:L_2(\T^n_\theta)\rightarrow L_2(\T^n_\theta;\nu)$ is a unitary isomorphism and $A:L_2(\T^n_\theta;\nu)\rightarrow L_2(\T^n_\theta)$ is bounded, we get 
\begin{equation*}
   \left\|  \lambda_\nu(x) \Delta_{g,\nu}^{-\frac{n}{2p}}\right\|_{\sL_{p,\infty}^{(\nu)}} \leq   
   (2\pi)^{\frac{n}{2}}  \left\|\lambda\big(x \sqrt{\nu}\big) \Lambda^{-\frac{n}{p}}   \right\|_{\sL_{p,\infty}}  \|A\|_{\sL(L_2(\T^n_\theta;\nu),L_2(\T^n_\theta))}.  
 \end{equation*}
Combining this with the estimate~(\ref{eq:specific-Cwikel-Prho}) for $P_\rho=\Lambda^{-n/p}$ gives
\begin{align*}
  \left\|  \lambda_\nu(x) \Delta_{g,\nu}^{-\frac{n}{2p}}\right\|_{\sL_{p,\infty}^{(\nu)}} & \leq   
   (2\pi)^{\frac{n}{2}} \big(c_n(p,q)+1) \big\|x\sqrt{\nu}\big\|_{L_q} \|A\|_{\sL(L_2(\T^n_\theta;\nu),L_2(\T^n_\theta))}\\
  &\leq  (2\pi)^{\frac{n}{2}} \big(c_n(p,q)+1)\|\nu\|^{\frac12} \big\|x\big\|_{L_q} \|A\|_{\sL(L_2(\T^n_\theta;\nu),L_2(\T^n_\theta))}.
\end{align*}
This proves~(\ref{eq:specific-Cwikel-Deltagnu}).

Suppose now that either $p\neq 1$ and $q=\max(p,1)$, or $p=1<q$. Let $x\in L_q(\T^n_\theta)$. As in the proof of Proposition~\ref{prop.specific-Cwikel-Prho} we may write 
$x=yz$ with $y$ and $z$ in $L_{2q}(\T^n _\theta)$ such that $\|y\|_{L_{2q}}=\|z\|_{L_{2q}}=(\|x\|_{L_q})^{1/2}$. If $u_1,v_1\in W_2^{n/2p}(\T^n_\theta;g,\nu)$, then by~(\ref{eq:Integration.lambdanu-dual}) we have
\begin{align*}
 \acou{\lambda_\nu(x)u_1}{v_1} & = (2\pi)^n \tau\big[v_1^*\nu^{\frac12}yz\nu^{\frac12}u_1\big] \\
 &= (2\pi)^n\tau\left[ \big( \nu^{\frac12}y^*\nu^{-\frac12}v_1\big)^*\nu\big( \nu^{-\frac12}z\nu^{\frac12}u_1\big)\right] \\
 & = \scal{\lambda_\nu(z)u_1}{\lambda_\nu(y^*)v_1}_\nu.  
\end{align*}
 Combining this with~(\ref{eq:Deltagnu-symmetricity}) we then see that, for all  $u,v\in L_2(\T^n_\theta;\nu)$, we have
\begin{equation*}
 \bigscal{\Delta_{g,\nu}^{-\frac{n}{4p}}\lambda_\nu(x) \Delta_{g,\nu}^{-\frac{n}{4p}}u}{v}= 
 \bigacou{\lambda_\nu(x) \Delta_{g,\nu}^{-\frac{n}{4p}}u}{\Delta_{g,\nu}^{-\frac{n}{4p}}v} =
 \bigscal{\lambda_\nu(z) \Delta_{g,\nu}^{-\frac{n}{4p}}u}{\lambda_\nu(y^*)\Delta_{g,\nu}^{-\frac{n}{4p}}v}. 
\end{equation*}
This shows that
\begin{equation}\label{eq:two-sided-curved-cwikel-splitting}
 \Delta_{g,\nu}^{-\frac{n}{4p}}\lambda_\nu(x) \Delta_{g,\nu}^{-\frac{n}{4p}} = \left(\lambda_\nu(y^*) \Delta_{g,\nu}^{-\frac{n}{4p}}\right)^* 
 \left(\lambda_\nu(z) \Delta_{g,\nu}^{-\frac{n}{4p}}\right). 
\end{equation}

We may now proceed as in the proof of Proposition~\ref{prop:Specific-Cwikel.sandwiched}. By using~(\ref{eq:two-sided-curved-cwikel-splitting}) and the H\"older inequality~(\ref{eq:schatten.holder}) we get
\begin{equation*}
  \left\| \Delta_{g,\nu}^{-\frac{n}{4p}}\lambda_\nu(x) \Delta_{g,\nu}^{-\frac{n}{4p}}\right\|_{\sL_{p,\infty}^{(\nu)}} \leq 2^{\frac1p} 
   \left\| \lambda_\nu(y^*) \Delta_{g,\nu}^{-\frac{n}{4p}}\right\|_{\sL_{2p,\infty}^{(\nu)}}  
    \left\|\lambda_\nu(z) \Delta_{g,\nu}^{-\frac{n}{4p}}\right\|_{\sL_{2p,\infty}^{(\nu)}} . 
\end{equation*}
Thanks to the first part and the equalities $\|y^*\|_{L_{2q}}=\|y\|_{L_{2q}}=\|z\|_{L_{2q}}=(\|x\|_{L_q})^{1/2}$ we get 
\begin{gather*}
  \left\| \lambda_\nu(y^*) \Delta_{g,\nu}^{-\frac{n}{4p}}\right\|_{\sL_{2p,\infty}^{(\nu)}} \leq c_n(2p,2q)  \big\|y^*\big\|_{L_{2q}}=c_n(2p,2q)  \big(\big\|x\big\|_{L_{2q}}\big)^{\frac12},\\
    \left\|\lambda_\nu(z) \Delta_{g,\nu}^{-\frac{n}{4p}}\right\|_{\sL_{2p,\infty}^{(\nu)}} \leq c_n(2p,2q)  \big\|z\big\|_{L_{2q}}=c_n(2p,2q)  \big(\big\|x\big\|_{L_{2q}}\big)^{\frac12}.  
\end{gather*}
This gives the estimate~(\ref{eq:specific-Cwikel-Deltagnu-2}). 

If $x$ is positive, then we may take $y=z=\sqrt{x}$. In this case~(\ref{eq:two-sided-curved-cwikel-splitting}) becomes 
\begin{equation*}
  \Delta_{g,\nu}^{-\frac{n}{4p}}\lambda_\nu(x) \Delta_{g,\nu}^{-\frac{n}{4p}} = \left(\lambda_\nu(\sqrt{x}) \Delta_{g,\nu}^{-\frac{n}{4p}}\right)^* 
 \left(\lambda_\nu(\sqrt{x}) \Delta_{g,\nu}^{-\frac{n}{4p}}\right),
\end{equation*}
and so we have 
\begin{equation*}
 \left\| \Delta_{g,\nu}^{-\frac{n}{4p}}\lambda_\nu(x) \Delta_{g,\nu}^{-\frac{n}{4p}}\right\|_{\sL_{p,\infty}^{(\nu)}} 
 = \bigg(\left\|\lambda_\nu(y) \Delta_{g,\nu}^{-\frac{n}{4p}}\right\|_{\sL_{2p,\infty}^{(\nu)}}\bigg)^2
 \end{equation*}
Combining this with the inequality~(\ref{eq:specific-Cwikel-Deltagnu}) for $z=\sqrt{x}$ gives the inequality~(\ref{eq:specific-Cwikel-Deltagnu-2}) without the $2^{1/p}$-factor. The proof is complete. 
\end{proof}

 \section{Dixmier Trace Formulas}\label{sec:L2-Connes-trace}
In this section, as applications of the Cwikel estimates of the previous section, we establish $L_p$-versions of the analogues for NC tori of  Connes' trace theorem and Connes' integration formulas established in~\cite{Po:JMP20}. 
This will show that a very wide class of operators built out of \psidos\ and $L_p$-elements are actually integrable in the sense of NCG.

\subsection{Connes' integration in noncommutative geometry} \label{sec:quantised_calculus}
One of the main goals of noncommutative geometry~\cite{Co:NCG} is to translate into the Hilbert space formalism of quantum mechanics the main tools of differential geometry.  In this framework the notion of integral corresponds to positive traces on the weak trace class $\sL_{1,\infty}$. Such traces are annihilated by finite rank operators (see, e.g., \cite{LSZ:Book}) and are always continuous (see, e.g.,~ \cite[Remark~2.3]{Po:JMP20}), and so they are annihilated by the finite-rank operator closure $(\sL_{1,\infty})_0$. 

There is a whole zoo of positive traces on $\sL_{1,\infty}$ (see, e.g., \cite{LSZ:Book, LSZ:Survey19} and the references therein). One important class of such traces is provided by Dixmier traces $\Tr_\omega:\sL_{1,\infty} \rightarrow \C$ (see~\cite{Di:CRAS66}; see also~\cite{Co:NCG, LSZ:Book, Po:Weyl}). We say that an operator $A\in \sL_{1,\infty}$ is \emph{measurable} if the value of $\Tr_\omega(A)$ is independent of the choice of the Dixmier trace. Equivalently (see~\cite{LSZ:Book, Po:Weyl}), the operator $A$ is measurable if and only if it satisfies the following Tauberian condition, 
\begin{equation}\label{eq:NCG-measurability}
 \bint A:=  \lim_{N\rightarrow \infty} \frac{1}{\log (N+1)}\sum_{j<N}\lambda_j(A) \quad \text{exists}. 
\end{equation}
Here $(\lambda_j(A))_{j\geq 0}$ is any eigenvalue sequence of $A$ such that 
\begin{equation*}
|\lambda_0(A)|\geq |\lambda_1(A)|\geq \cdots \geq  |\lambda_j(A)|\geq \cdots \geq 0. 
\end{equation*}
where each eigenvalue is repeating according to its (algebraic) multiplicity. The limit $\bint A$ is then called the \emph{NC integral} of $A$. 
 
 As it turns out there are numerous positive traces on $\sL_{1,\infty}$ that are not Dixmier traces (see, e.g.,~\cite{SSUZ:AIM15}). Thus, it stands for reason to consider a stronger notion of measurability. We say that an operator $A\in \sL_{1,\infty}$ is \emph{strongly measurable}  (or  \emph{positively measurable}) if all positive normalized traces take on the same value on $A$. Here a trace $\varphi$ on $ \sL_{1,\infty}$ is called normalized if 
 \begin{equation*}
 \big(A\geq 0 \ \text{and} \ \lambda_j(A)=(j+1)^{-1} \big)\ \Longrightarrow \ \varphi(A)=1. 
\end{equation*}

The Dixmier traces are positive normalized traces. Thus, if $A\in \sL_{1,\infty}$ is strongly measurable, then $A$ is measurable and, for any positive normalized trace $\varphi$ on  $\sL_{1,\infty}$, we have
\begin{equation*}
 \varphi(A)=\bint A = \lim_{N\rightarrow \infty} \frac{1}{\log (N+1)}\sum_{j<N}\lambda_j(A).
\end{equation*}
As mentioned above, any positive trace is continuous and is annihilated by $(\sL_{1,\infty})_0$. It follows that the strongly measurable operators form a closed subspace of $\sL_{1,\infty}$ containing $(\sL_{1,\infty})_0$. 

We refer to~\cite[Section~7]{SSUZ:AIM15} for a characterization of strongly measurable operators in terms of their eigenvalue sequences and for examples of measurable operators that are not strongly measurable. 

It can be shown that every continuous trace on $\sL_{1,\infty}$  is a linear combination of 4 positive traces (see~\cite[Corollary~2.2]{CMSZ:ETDS19}). Therefore, in the definition of strong measurability we may substitute continuous (normalized) traces for positive (normalized) traces without altering the class of strongly measurable  operators. This implies that strong measurability is insensitive to the choice of the inner product of the background Hilbert space $\sH$. 
 
Let $(M^n,g)$ be a closed Riemannian manifold, and  denote by $\Psi^m(M)$, $m\in \Z$, the space of (classical) \psidos\ of order $m$ on $M$. \emph{Connes' trace theorem}~\cite{Co:CMP88} asserts that every operator $P\in \Psi^{-n}(M)$ is measurable, and we have
\begin{equation}
 \bint P= \frac{1}{n} \Res(P), 
 \label{eq:Connes-trace-thm} 
\end{equation}
 where $\Res$ is the noncommutative residue trace of Guillemin~\cite{Gu:AIM85} and Wodzicki~\cite{Wo:NCR}. That is,
 \begin{equation*}
 \Res(P)= (2\pi)^{-n} \int_{S^*M} \sigma_{-n}(P)(x,\xi) dxd^{n-1}\xi, 
\end{equation*}
 where $S^*M$ is the unit cosphere bundle and $\sigma_{-n}(P)(x,\xi)$ is the principal symbol of $P$. It was further shown by Kalton-Lord-Potapov-Sukochev~\cite{KLPS:AIM13} that every operator in $\Psi^{-n}(M)$ is actually strongly measurable. These results can also be deduced from the Weyl's law for negative order \psidos\ of Birman-Solomyak~\cite{BS:VLU77, BS:VLU79, BS:SMJ79}  (see also Section~\ref{sec:Weyl} on this point). 
 
The r.h.s.~of~(\ref{eq:Connes-trace-thm}) makes sense for any integer order \psido, since the noncommutative residue is a trace on the algebra $\Psi^\Z(M):=\bigcup_{m\in \Z}\Psi^m(M)$ and is actually the unique trace up to constant multiple if $M$ is connected. Therefore, using the r.h.s.\ allows us to extend the definition of the NC integral to \emph{all} integer order \psidos, including \psidos\ that are not even bounded on $L_2(M,g)$.

Let $\Delta_g$ be the Laplacian of $(M^n,g)$. Specializing the above results to $P=f\Delta_g^{-n/2}$ and $P=\Delta_g^{-n/4}f\Delta_g^{-n/4}$ with $f\in C^\infty(M)$
shows that these operators are strongly measurable. Moreover, we get \emph{Connes' integration formula}, 
\begin{equation*}
 \bint \Delta_g^{-\frac{n}{4}}f\Delta_g^{-\frac{n}{4}} =   \bint f \Delta_g^{-\frac{n}{2}} = \hat{c}(n) \int_M f(x) \sqrt{g(x)}d^nx, 
\end{equation*}
where we have set $ \hat{c}(n):=\frac{1}{n} (2\pi)^{-n}|\bS^{n-1}|$. This formula shows that the NC integral recaptures the Riemannian volume.  

The above formula for $f\Delta_g^{-n/2}$ actually holds for all $f\in L_2(M)$ (see~\cite{KLPS:AIM13}).  The formula for $\Delta_g^{-n/4}f\Delta_g^{-n/4}$ further holds for any $f$ in the Orlicz space $\LLogL(M)$ (see~\cite{Ro:JST22, SZ:arXiv21}; see also~\cite{Po:MPAG22}). However, this result does not hold in general for functions in $L_1(M)$ (see~\cite{KLPS:AIM13}).
 
\subsection{Integration formulas on NC tori}
If  $P\in \Psi^{m}(\T^n_\theta)$, $m\in \Z$, has symbol $\rho(\xi)\sim \sum \rho_{m-j}(\xi)$, then  its \emph{noncommutative residue} is defined by
\begin{equation*}
 \Res(P)=\tau[c_P], \qquad \text{where}\ c_P:= \int_{\bS^{n-1}} \rho_{-n}(\xi) d^{n-1}\xi. 
\end{equation*}
Here the integral is meant as a Riemann (surface) integral with values in the locally convex space $C^\infty(\T^n_\theta)$. It can be shown that the noncommutative residue is a trace on the algebra $\Psi^\Z(\T^n_\theta):=\cup_{m\in \Z}\Psi^\Z(\T^n_\theta)$ and is the only one up to a constant multiple (see~\cite{Po:SIGMA20}; see also~\cite{LNP:TAMS16}). 

Recall that any \psido\ of order~$\leq -n$ is in the weak trace-class $\sL_{1,\infty}$. We have the following version for NC tori of Connes's trace theorem. 

\begin{proposition}[see~\cite{Po:JMP20}]\label{nct_trace_theorem}
 If $P \in \Psi^{-n}(\T^n_\theta)$,  then $P$ is strongly measurable, and we have
\begin{equation*}
 \bint P=\frac1{n} \Res(P). 
\end{equation*}
More generally, given any $x\in C(\T^n_\theta)$, the operator $\lambda(x)P$ is strongly measurable,  and we have
\begin{equation}\label{smooth_ctt}
    \bint \lambda(x)P = \frac{1}{n}\tau\big[x c_P\big].
\end{equation}
 \end{proposition}
 
Recall that by the first part of Proposition~\ref{prop.specific-Cwikel-Prho} if $P\in \Psi^{-n}(\T^n_\theta)$ and $x\in L_2(\T^n_\theta)$, then $\lambda(x)P$ is in $\sL_{1,\infty}$. By the 2nd part of the same Proposition, if $P, Q\in \Psi^{-n/2}(\T^n_\theta)$ and $x\in L_p(\T^n_\theta)$, $p>1$, then $Q^*\lambda(x)P\in \sL_{1,\infty}$. 
 
 The following theorem extends Proposition~\ref{nct_trace_theorem} to $L_p$-spaces.  

\begin{theorem} \label{thm:Lp-Connes-trace} 
The following hold. 
\begin{enumerate}
 \item Let $P\in \Psi^{-n}(\T^n_\theta)$ and $x\in L_2(\T^n_\theta)$. Then $\lambda(x)P$ is strongly measurable, and we have
\begin{equation}\label{rough_ctt}
 \bint \lambda(x) P = \frac{1}{n}\tau\big[xc_P\big]. 
\end{equation}

\item Let $P,Q\in \Psi^{-n/2}(\T^n_\theta)$ and $x\in L_p(\T^n_\theta)$, $p>1$. Then $Q^*\lambda(x)P$ is 
strongly measurable, and we have
\begin{equation}\label{rough_ctt_2}
 \bint Q^*\lambda(x) P = \frac{1}{n}\tau\big[xc_{PQ^*}\big]. 
\end{equation}
\end{enumerate}
\end{theorem}
\begin{proof}
Let $P\in \Psi^{-n}(\T^n_\theta)$. Set $\Lambda(P)=\{\lambda(x)P; \ x\in C(\T^n_\theta)\}$ and $\Lambda_2(P)=\{\lambda(x)P; \ x\in L_p(\T^n_\theta)\}$. The first part of Proposition~\ref{prop.specific-Cwikel-Prho} and the density of  $C(\T^n_\theta)$ in $L_2(\T^n_\theta)$ ensure that $\Lambda_2(P)$ is a subspace of $\sL_{1,\infty}$ in which $\Lambda(P)$ is a dense subspace. Proposition~\ref{nct_trace_theorem} asserts that $\Lambda(P)$ is contained in the space of strongly measurable operators. As the latter is a closed subspace of $\sL_{1,\infty}$ it follows that every operator in $\Lambda_2(P)$ is strongly measurable. 

By Proposition~\ref{nct_trace_theorem} the formula~(\ref{rough_ctt}) holds for all $x\in C(\T^n_\theta)$. The r.h.s.~of~(\ref{rough_ctt}) is a continuous linear form on $L_2(\T^n_\theta)$. The l.h.s.~is a continuous linear form on $L_2(\T^n_\theta)$ as well thanks  to the 1st part of Proposition~\ref{nct_trace_theorem} and the continuity of the NC integral $\bint$ on the subspace (strongly) measurable operators. Combining this with the density of $C(\T^n_\theta)$ in $L_2(\T^n_\theta)$ we then deduce that the formula~(\ref{rough_ctt}) holds true for any $x\in L_2(\T^n_\theta)$. This proves the first part.

It remains to prove the 2nd part. Let $P,Q\in \Psi^{-n/2}(\T^n_\theta)$. Set 
$\Lambda^\infty(P,Q)=\{Q^*\lambda(x)P;\ x\in C^\infty(\T^n_\theta)\}$ and $\Lambda_p(P,Q)=\{Q^*\lambda(x)P;\ x\in L_p(\T^n_\theta)\}$, $p>1$. As above, the 2nd part of Proposition~\ref{prop.specific-Cwikel-Prho} and the density of $C^\infty(\T^n_\theta)$ in $L_p(\T^n_\theta)$ imply that  $\Lambda_p(P,Q)$ is a subspace of $\sL_{1,\infty}$ that contains $\Lambda^\infty(P,Q)$ as a dense subspace. Note that $\Lambda^\infty(P,Q)$ is a subspace of $\Psi^{-n}(\T^n_\theta)$, and so by Proposition~\ref{nct_trace_theorem} every operator in $\Lambda^\infty(P,Q)$ is strongly measurable. As above, the fact that the strongly measurable operators form a closed subspace of $\sL_{1,\infty}$ then ensures that every operator in $\Lambda_p(P,Q)$ is strongly measurable.  

If $x\in C^\infty(\T^n_\theta)$, then $Q^*\lambda(x)P$ is an operator in $\Psi^{-n}(\T^n_\theta)$, and so by~(\ref{smooth_ctt}) we have
\begin{equation*}
 \bint Q^*\lambda(x)P= \frac{1}{n} \tau\big[c_{Q^*xP}\big]. 
\end{equation*}
Let $\rho_{-n/2}(\xi)$ and $\sigma_{-n/2}(\xi)$ be the respective principal symbols of $P$ and $Q$. The principal symbols of $Q^*P$ and $Q^*xP$ are $\sigma_{-n/2}(\xi)^*\rho_{-n/2}(\xi)$ and $\sigma_{-n/2}(\xi)^*x\rho_{-n/2}(\xi)$, respectively. Thus, 
 \begin{align*}
  \tau\big[c_{Q^*xP}\big] & = \int_{\bS^{n-1}} \tau\big[\sigma_{-\frac{n}{2}}(\xi)^*x\rho_{-\frac{n}{2}}(\xi)\big]d\xi \\ 
  &= \int_{\bS^{n-1}} \tau\big[x\rho_{-\frac{n}{2}}(\xi)\sigma_{-\frac{n}{2}}(\xi)^*\big]d\xi\\
   & = \tau\bigg[x \int_{\bS^{n-1}}\rho_{-\frac{n}{2}}(\xi)\sigma_{-\frac{n}{2}}(\xi)^*d\xi\bigg]  =\tau\big[xc_{PQ^*}\big]. 
  \end{align*}
This proves the equality~(\ref{rough_ctt_2}) for $x\in C^\infty(\T^n_\theta)$. As with~(\ref{rough_ctt}), the r.h.s.~of~(\ref{rough_ctt_2}) is a continuous linear form on $L_p(\T^n_\theta)$. The 2nd part of Proposition~\ref{prop.specific-Cwikel-Prho} and the continuity of $\bint$ ensure that the l.h.s.~is a continuous linear form on $L_p(\T^n_\theta)$ as well. The density of $C^\infty(\T^n_\theta)$ in $L_p(\T^n_\theta)$ then implies that both sides agree for all $x\in L_p(\T^n_\theta)$. This gives the 2nd part. The proof is complete. 
\end{proof}

\begin{corollary}\label{cor:flat-integration-formula}
 The following hold. 
\begin{enumerate}
 \item If $x\in L_2(\T^n_\theta)$, then the operator $\lambda(x)\Delta^{-n}$ is strongly measurable, and we have
\begin{equation}\label{eq:flat-L2-integration-formula}
 \bint \lambda(x) \Delta^{-\frac{n}{2}} = c(n)\tau\big[x\big], \qquad c(n):=\frac{1}{n}\big|\bS^{n-1}\big|. 
\end{equation} 

\item If $x\in L_p(\T^n_\theta)$, $p>1$, then the operator $\Delta^{-n/4}\lambda(x)\Delta^{-n/4}$ is strongly measurable, and we have 
 \begin{equation}
 \bint  \Delta^{-\frac{n}{4}}\lambda(x) \Delta^{-\frac{n}{4}} = c(n)\tau\big[x\big]. 
 \label{eq:Flat-Lp-integration-formula}
\end{equation}
\end{enumerate}
\end{corollary}
\begin{proof}
 We only need to compute $c_{\Delta^{-n/2}}$. Note that $\Delta^{-n/2}$ is an operator in $\Psi^{-n}(\T^n_\theta)$ whose principal symbol is $|\xi|^{-n}$. This can be seen as a special case of Proposition~\ref{prop:powersLB}. A simpler way is to observe that~(\ref{eq:toroidal.Prhou-equation}) implies that $\Delta^{-n/2}=P_{\rho}$, where $\rho(\xi)$ is any smooth function such that $\rho(\xi)=|\xi|^{-n}$ for $|\xi|\geq 1$ and $\rho(0)=0$. In any case, we get 
 \begin{equation}
 c_{\Delta^{-\frac{n}{2}}}= \int_{\bS^{n-1}}|\xi|^{-n}d^{n-1}\xi = \int_{\bS^{n-1}}d^{n-1}\xi =\big|\bS^{n-1}\big|. 
 \label{eq:density-flat-Laplacian}
\end{equation}
 Therefore, applying the first part of Theorem~\ref{thm:Lp-Connes-trace} to $P=\Delta^{-n/2}$ gives the first part. Applying the 2nd part of 
 Theorem~\ref{thm:Lp-Connes-trace} to $P=Q=\Delta^{-n/4}$ gives the 2nd part. The proof is complete. 
\end{proof}

\begin{remark}
In~\cite{MSZ:MA19} the integration formula~(\ref{eq:flat-L2-integration-formula}) is established for all $x\in C(\T^n_\theta)$ by using the $C^*$-algebraic approach to the principal symbol of~\cite{MSZ:MA19, SZ:JOT18}. Thus, Corollary~\ref{cor:flat-integration-formula} provides us with an $L_p$-version of that result.  
\end{remark}

\subsection{Curved Integration Formulas}  \label{sec:curved-integration}
Let $g=(g_{ij})$ be a Riemannian metric on $\T^n_\theta$ and let $\nu\in \GL_1^+(C^\infty(\T^n_\theta))$ be a smooth positive density. 

For ordinary Riemmanian manifolds  the Riemannian density $\nu(g)=\sqrt{g}$ naturally appears in Connes's integration formula~(\ref{eq:Intro.Integration-Formula}). As shown in~\cite{Po:JMP20} for curved NC tori what comes into play is  
\begin{equation}
 \tilde{\nu}(g):=  \frac{1}{|\bS^{n-1}|} \int_{\bS^{n-1}} |\xi|_g^{-n} d^{n-1}\xi,
 \label{eq:spectral-riemannian-density}
\end{equation}
where $|\xi|_g$ is defined in~(\ref{eq:Laplace.norm-g}). 
It follows from the inequalities~(\ref{eq:Laplacian.positivity-normg}) that this defines a positive invertible element of $C^\infty(\T^n_\theta)$, i.e., a smooth positive density (see~\cite{Po:JMP20}). We call $\tilde{\nu}(g)$ the \emph{spectral Riemannian density} of $g$. We refer to~\cite{Po:JMP20} for a more detailed account on spectral Riemannian densities.  

\begin{example}[see~\cite{Po:JMP20}] Assume that $g=(g_{ij})$ is self-compatible in the sense mentioned in Example~\ref{ex:Riem.self-compatible}. Then, we have 
 \begin{equation}
\label{eq:Curved.tnug-compatible}
 \tilde{\nu}(g)=\nu(g)=\sqrt{\det (g)},
\end{equation}
where $\det(g)$ is given by~(\ref{eq:det.Leibniz}). In particular, if $g_0=k^2g$, $k\in \GL_1^+(C^\infty(\T^n_\theta))$, is a conformally flat metric, then $\tilde{\nu}(g)=\nu(g)=k^n$. 
\end{example}

\begin{example}
Suppose that $n = 2$,  and let $g$ be of the form
\begin{equation*}
g=
\begin{pmatrix}
a & 0 \\
0 & b
\end{pmatrix}, \qquad a,b\in \GL_1^+\left(C^\infty(\T^2_\theta)\right). 
\end{equation*}
We have
\begin{equation*}
 g= a^{\frac12} \hat{g} a^{\frac12} , \qquad \text{where}\ 
 \hat{g}=
\begin{pmatrix}
1 & 0 \\
0 & a^{-\frac12} b a^{-\frac12} 
\end{pmatrix}.
\end{equation*}
Note that $\hat{g}$ is a self-compatible Riemannian metric. Therefore, by~(\ref{eq:det.Leibniz}) and~(\ref{eq:Curved.tnug-compatible}) we have
\begin{equation*}
 \det(\hat{g})=a^{-\frac12} b a^{-\frac12} \qquad \text{and} \qquad 
 \tilde{\nu}(\hat{g})=\sqrt{\det(\hat{g})}=\big( a^{-\frac12} b a^{-\frac12}\big)^{\frac12}. 
\end{equation*}
Furthermore, as $g^{-1}=a^{-\frac12} \hat{g}^{-1} a^{-\frac12}$, we have
\begin{equation*}
 |\xi|_g^{2}= \sum g^{ij}\xi_i\xi_j = \sum a^{-\frac12} \hat{g}^{ij} a^{-\frac12}\xi_i\xi_j=a^{-\frac12}|\xi|_{\hat{g}}^2 a^{-\frac12}. 
\end{equation*}
Therefore, by the very definition~(\ref{eq:spectral-riemannian-density}) of $\tilde{\nu}(g)$ we get
\begin{equation*}
 \tilde{\nu}(g) =\frac{1}{2\pi} \int_{\bS^1} |\xi|_g^{-2} d\xi = \frac{1}{2\pi} \int_{\bS^1} a^{\frac12}|\xi|_{\hat{g}}^{-2} a^{\frac12}d\xi =a^{\frac12} \tilde{\nu}(\hat{g})a^{\frac12}= a^{\frac12}\big( a^{-\frac12} b a^{-\frac12}\big)^{\frac12}a^{\frac12}.  
\end{equation*}
This should be compared with $\nu(g)$, which in this case is $\exp(\log(a)+\log(b)).$
In particular, if $[a,b]=0$ we recover the formula $\tilde{\nu}(g)=\sqrt{ab}$ predicted by~(\ref{eq:Curved.tnug-compatible}). Note also that $a^{\frac12}\big( a^{-\frac12} b a^{-\frac12}\big)^{\frac12}a^{\frac12}$ is the geometric mean of $a$ and $b$ in the sense of Pusz-Woronowicz~\cite{PW:RMP75}.
\end{example}

The integration formula for curved NC tori of~\cite{Po:JMP20} then can be stated as follows.  

\begin{proposition}[\cite{Po:JMP20}] For every $x\in C(\T^n_\theta)$, the operator $ \lambda_{\nu}(x)\Delta_{g,\nu}^{-\frac{n}2}$ is strongly measurable, and we have
\begin{equation}
 \bint \lambda_{\nu}(x)\Delta_{g,\nu}^{-\frac{n}2} = c(n) \tau\big[x \tilde{\nu}(g)\big], \qquad c(n):=\frac{1}{n}|\bS^{n-1}|.  
 \label{eq:curved-integration-C}
\end{equation}
 \end{proposition}

\begin{remark}
 The formula~(\ref{eq:curved-integration-C}) is proved in~\cite{Po:JMP20} in the case $\nu=\nu(g)$. However, the result hold \emph{verbatim} for any smooth positive density $\nu\in \GL_1^+(C^\infty(\T^n_\theta))$ (\emph{cf}.~\cite[Remark~10.7]{Po:JMP20}). 
\end{remark}

\begin{remark}
 For the flat metric $g=g_0$ and the density $\nu=1$, the above integration formula is proved in~\cite{MSZ:MA19} by using the $C^*$-algebraic approach to the principal symbol of~\cite{MSZ:MA19, SZ:JOT18}.
\end{remark}

\begin{example}
 If $g=(g_{ij})$ is self-compatible, then $\tilde{\nu}(g)=\nu(g)=\sqrt{\det(g)}$, where $\det(g)$ is given by~(\ref{eq:det.Leibniz}). Thus, in this case the formula~(\ref{eq:curved-integration-C}) becomes, 
 \begin{equation}
 \bint \lambda_{\nu}(x)\Delta_{g,\nu}^{-\frac{n}2} = c(n) \tau\big[x \sqrt{\det(g)}\big], \qquad x\in C(\T^n_\theta). 
\end{equation}
\end{example}

Recall that if $x\in L_2(\T^n_\theta)$, then  by the first part of Theorem~\ref{thm:curved-Cwikel} the operator $\lambda_\nu(x)\Delta_{g,\nu}^{-n/2}$ is in the weak trace class $\sL_{1,\infty}^{(\nu)}$. By the 2nd part of the same proposition if $x\in L_p(\T^n_\theta)$, $p>1$, then the operator 
$\Delta_{g,\nu}^{-n/4}\lambda_\nu(x)\Delta_{g,\nu}^{-n/4}$ is also in $\sL_{1,\infty}^{(\nu)}.$

\begin{theorem}\label{thm:Lp-curved-formula} 
The following hold. 
\begin{enumerate}
 \item If $x\in L_2(\T^n)$, then $\lambda_\nu(x) \Delta_{g,\nu}^{-\frac{n}{2}}$ is strongly measurable, and we have 
\begin{equation}\label{curved_rough_ctt_1}
 \bint \lambda_\nu(x) \Delta_{g,\nu}^{-\frac{n}2} = c(n) \tau\big[x\tilde{\nu}(g)\big], \qquad c(n):=\frac1{n}|\bS^{n-1}|. 
\end{equation} 

\item If $x\in L_{p}(\T^n)$, $p>1$, then $\Delta_{g,\nu}^{-n/4}\lambda_\nu(x) \Delta_{g,\nu}^{-n/4}$ is strongly measurable, and we have 
\begin{equation}
 \bint  \Delta_{g,\nu}^{-\frac{n}4} \lambda_\nu(x) \Delta_{g,\nu}^{-\frac{n}4} = c(n) \tau\big[x\tilde{\nu}(g)\big].
 \label{eq:L1+-curved-formula} 
\end{equation} 
\end{enumerate}
 \end{theorem}
\begin{proof}
As mentioned in Section \ref{sec:quantised_calculus}, strong measurability and the NC integral do not depend on the choice of the inner product. As $L_2(\T^n_\theta)$ and $L_2(\T^n_\theta;\nu)$ are the same Hilbert space with equivalent inner products, we may regard  $\lambda_\nu(x)\Delta_{g,\nu}^{-\frac{n}2}$ and $\Delta_{g,\nu}^{-n/4}\lambda_\nu(x)\Delta_{g,\nu}^{-n/4}$ as operators in $L_2(\T^n_\theta)$. 

Let $x\in L_2(\T^n_\theta)$. By~(\ref{eq:lambda-to-lambda-nu-onesided}) we have $\lambda_\nu(x)\Delta_{g,\nu}^{-n/2}=\lambda(\nu^{-1/2}x\nu^{1/2})\Delta_{g,\nu}^{-n/2}$. Proposition~\ref{prop:powersLB} asserts that 
$\Delta_{g,\nu}^{-n/2}$ is an operator in $\Psi^{-n}(\T_\theta^n)$. Therefore, by Proposition~\ref{thm:Lp-Connes-trace} the operator $\lambda_\nu(x)\Delta_{g,\nu}^{-n/2}$ is strongly measurable, and we have
\begin{equation}\label{eq:initial-rough-ctt-computation}
 \bint \lambda_\nu(x) \Delta_{g,\nu}^{-\frac{n}2} = \bint \lambda\big(\nu^{-\frac12}x\nu^{\frac12}\big) \Delta_{g,\nu}^{-\frac{n}2} = 
 \frac{1}{n} \tau\left[(\nu^{-\frac12}x\nu^{\frac12})c_{\Delta_{g,\nu}^{-\frac{n}2}}\right]. 
\end{equation}

It remains to compute $c_{\Delta_{g,\nu}^{-n/2}}$. We know by Proposition~\ref{prop:powersLB} that $ \Delta_{g,\nu}^{-n/2}$ has principal symbol $\nu^{-1/2}|\xi|_g\nu^{1/2}$. Thus,
\begin{equation*}
 c_{\Delta_{g,\nu}^{-\frac{n}2}}= \int_{\bS^{n-1}}\nu^{-\frac12}|\xi|_g\nu^{\frac12}d^{n-1}\xi = \nu^{-\frac12} \bigg(\int_{\bS^{n-1}}|\xi|_gd^{n-1}\xi\bigg)\nu^{\frac12}=|\bS^{n-1}| \nu^{-\frac12}\tilde{\nu}(g)\nu^{\frac12}.  
\end{equation*}
Therefore, if we set $c(n)=n^{-1}|\bS^{n-1}|$, then we get
\begin{equation}\label{eq:trace_tau_computation}
 \frac{1}{n}\tau\left[(\nu^{-\frac12}x\nu^{\frac12})c_{\Delta_{g,\nu}^{-\frac{n}2}}\right]
 =c(n) \tau\left[\big(\nu^{-\frac12}x\nu^{\frac12}\big)\big(\nu^{-1/2}\tilde{\nu}(g)\nu^{1/2}\big)\right] = 
  c(n) \tau\big[x\tilde{\nu}(g)\big]. 
\end{equation}
Combining this with~(\ref{eq:initial-rough-ctt-computation}) gives~(\ref{curved_rough_ctt_1}).  

Let $x\in L_p(\T^n_\theta)$, $p>1$. By~(\ref{eq:lambda-to-lambda-nu-twosided}) we have 
$\Delta_{g,\nu}^{-\frac{n}4}\lambda_\nu(x)\Delta_{g,\nu}^{-\frac{n}4}= \Delta_{g,\nu}^{-\frac{n}4}\lambda(\nu^{-1/2}x\nu^{1/2})\Delta_{g,\nu}^{-n/4}$. We know by Proposition~\ref{prop:powersLB} that $\Delta_{g,\nu}^{-n/4}\in \Psi^{-n/2}(\T_\theta^n)$. Proposition~\ref{thm:Lp-Connes-trace} then ensures that $\Delta_{g,\nu}^{-\frac{n}4}\lambda_\nu(x)\Delta_{g,\nu}^{-\frac{n}4}$ is a strongly measurable operator, and we have
\begin{equation}\label{eq:curved_two_sided_rough_ctt_final_computation}
 \bint\Delta_{g,\nu}^{-\frac{n}4}\lambda_\nu(x)\Delta_{g,\nu}^{-\frac{n}4} = \bint   \Delta_{g,\nu}^{-\frac{n}4}\lambda\big(\nu^{-\frac12}x\nu^{\frac12}\big) \Delta_{g,\nu}^{-\frac{n}4} = 
 \frac{1}{n} \tau\left[(\nu^{-\frac12}x\nu^{\frac12})c_{\Delta_{g,\nu}^{-\frac{n}2}}\right]. 
\end{equation}
Note~(\ref{eq:trace_tau_computation}) holds for any $x\in L_1(\T^n_\theta)$, and hence it holds for any $x\in L_p(\T^n_\theta)$, $p>1$. Combining it with~(\ref{eq:curved_two_sided_rough_ctt_final_computation}) immediately gives~(\ref{eq:L1+-curved-formula}).  The proof is complete. 
\end{proof}

\section{CLR and Lieb-Thirring Inequalities on Curved NC Tori }\label{CLR_section}
In this section, we obtain versions for curved NC tori of the celebrated Cwikel-Lieb-Rozenblum (CLR) inequality~\cite{Cw:AM77, Li:BAMS76, Li:1980, Roz:1972, Roz:1976} and Lieb-Thirring (LT) inequalities~\cite{LT:PRL75, LT:SMP76} for number and $\gamma$-moments of negative eigenvalues of fractional Schr\"odinger operators associated with powers of  Laplace-Beltrami operators and $L_p$-potentials. As in the Euclidean setting the Lieb-Thirring inequality for $\gamma=1$ is shown to be dual to a Sobolev inequality. 

The inequalities of this section extends to the curved setting the inequalities~\cite{MP:JMP22} for Schr\"odinger operators associated with  powers of the ordinary Laplacian $\Delta$. For ordinary (flat) tori $\T^n$ Lieb-Thirringand Sobolev  inequalities were obtained by Ilyin~\cite{Il:JST12} for $n=2$ and by Ilyin-Laptev~\cite{IL:SM16} for $n\leq 19$ (see also \cite{IL:StPMJ20, ILZ:MN19, ILZ:JFA20}). 

\subsection{CLR inequality for curved NC tori (1st version)} 
Suppose that $p>1$ and $q=\max(p,1)$, or $p=1<q$, and let $V^*=V\in L_q(\T^n_\theta)$. Lemma~\ref{lem:Boundedness.sandwich1-curved} ensures that $\lambda_\nu(V)$ induces a bounded operator $\lambda_\nu(V):W_2^{-n/2p}(\T^n_\theta;g,\nu)\rightarrow W_2^{n/2p}(\T^n_\theta;g,\nu)$. Thus, we get a quadratic form on $W_2^{n/2p}(\T^n_\theta;g,\nu)$ given by

\begin{equation*}
 Q_{\lambda_\nu(V)}(u,v)=\acou{\lambda_\nu(V)u}{v},  \qquad u,v\in W_2^{\frac{n}{2p}}(\T^n_\theta). 
\end{equation*}
Note that by~(\ref{eq:Integration.lambdanu-dual}) we have
\begin{equation}\label{eq:Q_lambda_def}
  Q_{\lambda_\nu(V)}(u,v)= (2\pi)^n\tau\big[v^*\nu^{1/2}V\nu^{1/2}u\big],  \qquad u,v\in W_2^{\frac{n}{2p}}(\T^n_\theta).  
\end{equation}
In particular, $Q_{\lambda_\nu(V)}$ is symmetric and $Q_V\geq 0$ if $V\geq 0$.
More generally, we have
\begin{equation}
 V_1\leq V_2 \Longrightarrow \lambda_\nu(V_1)\leq \lambda_\nu(V_2). 
 \label{eq:CLR.monotonicity-lambda}
\end{equation}

Our main focus is on the fractional Schr\"odinger operators, 
\begin{equation*}
 H_{V}^{(p)}:= \Delta_{g,\nu}^{\frac{n}{2p}} +\lambda_\nu(V), \qquad p>0.  
\end{equation*}
Note that $Q_{\lambda_\nu(V)}$ has same domain as the quadratic form of the operator $H_0^{(p)}=\Delta_{g,\nu}^{n/2p}$, i.e., 
\begin{equation*}
 Q_{H_0^{(p)}}(u,v)=\bigscal{\Delta_{g,\nu}^{\frac{n}{4p}}u}{\Delta_{g,\nu}^{\frac{n}{4p}}v}= \bigacou{\Delta_{g,\nu}^{\frac{n}{2p}}u}{v}, \qquad u,v\in W_2^{\frac{n}{2p}}(\T^n_\theta;g,\nu),  
\end{equation*}
where we regard $\Delta^{n/2p}$ as an operator from $W_2^{-n/2p}(\T^n_\theta;g,\nu)$ to $W_2^{n/2p}(\T^n_\theta;g,\nu)$. In particular, $H_{V}^{(p)}$ makes sense as an operator from  $W_2^{-n/2p}(\T^n_\theta;g,\nu)$ to $W_2^{n/2p}(\T^n_\theta;g,\nu)$. 

We have 
\begin{equation*}
 \big(1+H_0^{(p)}\big)^{-\frac12} \lambda_\nu(V)  \big(1+H_0^{(p)}\big)^{-\frac12} 
 = A\Delta_{g,\nu}^{-\frac{n}{4p}}\lambda_\nu(V)\Delta_{g,\nu}^{-\frac{n}{4p}}A, 
\end{equation*}
where $A:=(1+H_0^{(p)})^{-\frac12}\Delta_{g,\nu}^{n/4p}= (\Delta_{g,\nu}^{-n/2p}+1)^{-1/2}$ is a bounded operator on $L_2(\T^n_\theta;\nu)$. Therefore, it follows from Theorem~\ref{thm:curved-Cwikel} that $(1+H_0^{(p)})^{-\frac12}\lambda_\nu(V) (1+H_0^{(p)})^{-\frac12}$ is in the weak Schatten class $\sL_{p,\infty}^{(\nu)}$, and hence is compact. That is, the symmetric quadratic form $Q_{\lambda_\nu(V)}$ is $H^{(p)}_0$-form compact, and so it is $H^{(p)}_0$-form bounded with zero $H$-bound (see~\cite[\S7.8]{Si:AMS15}). Thus, by the KLMN theorem (see, e.g., \cite{RS2:1975, Sc:Springer12}) the restriction of $H_V^{(p)}$ to $W_2^{n/p}(\T^n_\theta;g,\nu)$ is a bounded from below selfadjoint operator on $L_2(\T^n_\theta;\nu)$ whose quadratic form is precisely $Q_{H_0^{(p)}}+Q_{\lambda_\nu(V)}$. 

As $H_0^{(p)}$ has compact resolvent, $H^{(p)}_V$ has compact resolvent as well, and hence it has pure discrete spectrum (see, e.g., \cite[Theorem 7.8.4]{Si:AMS15}). We denote by $-\lambda_j^{-}(H^{(p)}_V)$ the \emph{negative} eigenvalues of $H_V^{(p)}$ in such a way that 
\begin{equation*}
\lambda_0^{-}\big(H_V^{(p)}\big)\geq \lambda_1^{-}\big(H_V^{(p)}\big)\geq\cdots>0 ,
\end{equation*}
where each eigenvalue is repeated according to multiplicity. We then introduce the counting function, 
\begin{equation}
 N^{-}\big(H_V^{(p)};\lambda\big) : = \#\left\{j; \ \lambda_j^{-}\big(H_V^{(p)}\big)< \lambda\right\}, 
 \qquad \lambda\geq 0. 
\label{eq:CLR.counting}
\end{equation}
For $\lambda=0$ we set $N^{-}(H_V^{(p)})=N^{-}(H_V^{(p)};0)$. Moreover, by Glazman's lemma (see, e.g., \cite[Theorem~10.2.3]{BS:Book}) we have 
\begin{equation}\label{eq:Glazman_reformulation}
  N^{-}(A;\lambda) =  \max\left\{ \dim F; \ F\in \sF^-\big(H_V^{(p)};\lambda\big)\right\}, 
\end{equation}
where $\sF^-(H_V^{(p)};\lambda)$ is the family of subspaces $F\subset W_2^{n/p}(\T^n_\theta;g,\nu)$ such that $Q_{H_V^{(p)}}(u,u)<-\lambda \|u\|^{2}_{L_2(\T^n_\theta;\nu)}$ for all $u\in F\setminus 0$. Combining this with~(\ref{eq:CLR.monotonicity-lambda}) yields the monotonicity principle, 
\begin{equation}
 V_1\leq V_2 \Longrightarrow  N^{-}\big(H_{V_2}^{(p)};\lambda\big) \leq N^{-}\big(H_{V_1}^{(p)};\lambda\big) \quad \forall \lambda\geq 0. 
 \label{eq:CLR.monotonicity-N}
\end{equation}

In what follows, we let $V_{\pm}=\frac12(|V|\pm V)$ be the positive and negative parts of $V$, so that $V=V_+-V_{-}$

\begin{theorem}[CLR Inequality for Curved NC Tori; 1st Version]\label{thm:CLR.CLR-NCtori}
Suppose that $q=\max(p,1)$ if $p\neq 1$, or $p=1<q$. If $V=V^*\in L_q(\T^n_\theta)$, then 
\begin{equation}\label{eq:CLR.CLR-inequality}
 N^{-}\left(H_V^{(p)}\right)-1 \leq C_{g,\nu}(2p,2q)^{2p} \big(\big\|V_{-}\big\|_{L_q}\big)^p ,
\end{equation}
where $C_{g,\nu}(2p,2q)$ is the best constant in the inequality~(\ref{eq:specific-Cwikel-Deltagnu}) for the pair $(2p,2q)$. 
\end{theorem}
\begin{proof}
 The proof goes along merely the same lines as that of the proof of~\cite[Theorem~8.1]{MP:JMP22}. As $V=V_+-V_{-}\geq -V_{-}$, the monotonicity principle~(\ref{eq:CLR.monotonicity-N}) implies that
\begin{equation}\label{eq:CLR.proof.monotonicity}
 N^{-}\big(H_{V}^{(p)}\big) \leq N^{-}\big(H_{-V_{-}}^{(p)}\big). 
\end{equation}
Let $\Pi_0^{(\nu)}$ the orthogonal projection onto $\ker \Delta_{g,\nu}=\C\cdot 1$ in $L_2(\T^n_\theta;\nu)$. That is, 
\begin{equation*}
 \Pi_0^{(\nu)}=\scal{u}{1}_\nu1 =\tau(u\nu)1, \qquad u\in L_2(\T^n_\theta;\nu). 
\end{equation*}
As $-V_{-}\leq 0$ the Birman-Schwinger principle in the version provided by~\cite[Theorem~7.9]{MP:JMP22} gives
\begin{equation*}
 N^{-}\big(H_{-V_{-}}^{(p)};\lambda\big) - N^{+}\big(\Pi_0^{(\nu)}V_{-}\Pi_0^{(\nu)}\big)\leq  
 \left(\left\| \Delta_{g,\nu}^{-\frac{n}{4p}}\lambda_\nu(V) \Delta_{g,\nu}^{-\frac{n}{4p}}\right\|_{\sL_{p,\infty}^{(\nu)}}\right)^p,  
\end{equation*}
where $N^{+}(\Pi_0^{(\nu)}V_{-}\Pi_0^{(\nu)})$ is the number of positive eigenvalues of $\Pi_0^{(\nu)}V_{-}\Pi_0^{(\nu)}$. As 
$\op{rk}(\Pi_0^{(\nu)}V_{-}\Pi_0^{(\nu)})\leq \op{rk}(\Pi_0^{(\nu)})=1$, we have $N^{+}(\Pi_0^{(\nu)}V_{-}\Pi_0^{(\nu)})\leq 1$. Thus, 
\begin{equation*}
 N^{-}\big(H_{V}^{(p)}\big) -1 \leq  N^{-}\big(H_{-V_{-}}^{(p)};\lambda\big) - N^{+}\big(\Pi_0^{(\nu)}V_{-}\Pi_0^{(\nu)}\big)\leq  
 \left(\left\| \Delta_{g,\nu}^{-\frac{n}{4p}}\lambda_\nu(V_{-}) \Delta_{g,\nu}^{-\frac{n}{4p}}\right\|_{\sL_{p,\infty}^{(\nu)}}\right)^p. 
\end{equation*}
 Combining this with the Cwikel-type estimate~(\ref{eq:specific-Cwikel-Deltagnu-2}) immediately gives the result. 
\end{proof}

\begin{remark}
 We have $Q_{H^{(p)}_V}(1,1)=\acou{1}{\lambda_\nu(V)1}=\hat{\tau}_\nu(V)$. Thus, if $\hat{\tau}_\nu(V)=(2\pi)^n\tau[V\nu]<0$, then the variational principle~(\ref{eq:Glazman_reformulation}) implies that, in this case, $N^{-}(H_V^{(p)})$ is always~$\geq 1$. 
\end{remark}

\subsection{CLR inequality on curved NC tori (2nd version)} 
As $\Delta_{g,\nu}$ has a non-trivial nullspace it is natural to restrict ourself to the orthogonal complement of $\ker \Delta_{g,\nu}=\C\cdot 1$, i.e., the subspace, 
\begin{equation*}
  \dot{L}_2\big(\T^n_\theta;\nu\big) := \left\{u \in L_2\big(\T^n_\theta;\nu\big);\ \scal{u}{1}_\nu=\tau[u\nu]=1\right\}.
\end{equation*}
Thus, we may regard $\dot{L}_2(\T^n_\theta;\nu)$ as the space of elements of $L_2(\T^n_\theta;\nu)$ that have zero mean value with respect to $\nu$. For $s\geq 0$, we also set
\begin{equation*}
  \dot{W}^s_2\big(\T^n_\theta;g,\nu\big) := W^s_2\big(\T^n_\theta;g,\nu\big)\cap \dot{L}_2\big(\T^n_\theta;\nu\big).
\end{equation*}
In addition, we denote by $\dot{W}^{-s}_2(\T^n_\theta;g,\nu)$ the anti-linear dual of $\dot{W}^{s}_2(\T^n_\theta;g,\nu)$. 

Let $\dot{\Delta}_{\nu,g}$ be the compression of $\Delta_{g,\nu}$ to $\dot{L}_2(\T^n_\theta)$. Its domain is  
$\dot{W}^{s}_2(\T^n_\theta;g,\nu)$, and it has same spectrum and eigenspaces as $\Delta_{g,\nu}$ except for the eigenvalue $\lambda=0$. 

Suppose that $p\neq 1$ and $q=\max(p,1)$, or $p=1<q$, and let $V=V^*\in L_q(\T^n_\theta)$. Let $\dot{\lambda}_\nu(V):\dot{W}^{-n/2p}(\T^n_\theta;g,\nu) \rightarrow \dot{W}^{n/2p}(\T^n_\theta;g,\nu)$ be the compression of $\lambda_\nu(V)$ to $\dot{L}_2(\T^n_\theta;\nu)$. More precisely, we mean the operator given by
\begin{equation*}
 \bigacou{\dot{\lambda}_\nu(V)u}{v}=Q_{\lambda_\nu(V)}(u,v)=\acou{\lambda_\nu(V)u}{v}, \qquad u,v\in  \dot{W}^{\frac{n}{2p}}_2\big(\T^n_\theta;g,\nu\big). 
\end{equation*}
That is, the quadratic form $Q_{\dot{\lambda}_\nu(V)}$ is the restriction of $Q_{\lambda_\nu(V)}$ to $\dot{W}_2^{n/2p}(\T^n_\theta;g,\nu)$. In particular, the monotonicity principle~(\ref{eq:CLR.monotonicity-N}) holds \emph{verbatim} for the operators $\dot{\lambda}_\nu(V)$. 

With respect to the orthogonal splitting $L^2(\T^n_\theta;\nu)=\dot{L}^2(\T^n_\theta;\nu)\oplus \C\cdot1$ we have
\begin{equation}\label{eq:LT.dotDeltaV}
\Delta_{g,\nu}^{-\frac{n}{4p}}  \lambda_\nu(V)\Delta_{g,\nu}^{-\frac{n}{4p}} = 
\begin{pmatrix}
 \dot{\Delta}_{g,\nu}^{-\frac{n}{4p}}  \dot{\lambda}_\nu(V) \dot{\Delta}_{g,\nu}^{-\frac{n}{4p}} & 0\\ 0 &0
\end{pmatrix}. 
\end{equation}
In particular, we see that $ \dot{\Delta}_{g,\nu}^{-{n}/{4p}}\dot{\lambda}_\nu(V) \dot{\Delta}_{g,\nu}^{-{n}/{4p}}$ is in the weak Schatten class $\sL_{p,\infty}(\dot{L}_{2}(\T^n_\theta;\nu))$ and has same  $\sL_{p,\infty}$-quasi-norm as $\Delta_{g,\nu}^{-{n}/{4p}}  \lambda_\nu(V)\Delta_{g,\nu}^{-{n}/{4p}}$. 
Thus, in the same way as with $H^{(p)}_{V}$ the operator $\dot{H}_{V}^{(p)}:=\dot{\Delta}_{g,\nu}^{n/2p}+\dot{\lambda}_{\nu}(V)$ restricts on 
$\dot{W}_2^{n/p}(\T^n_\theta;g,\nu)$ to a selfadjoint operator on $L_2(\T^n_\theta;\nu)$ which is bounded from below and has compact resolvent. We define its counting function as in~(\ref{eq:CLR.counting}). The monotonicity principle~(\ref{eq:CLR.monotonicity-N}) holds \emph{verbatim} for the operators $\dot{H}_V^{(p)}$. 

\begin{theorem}[CLR Inequality on Curved NC Tori; 2nd Version]\label{thm:CLR.CLR-NCtori2}
Suppose that $p\neq 1 $ and $q = \max\{p,1\}$, or $p=1<q$.  If $V=V^* \in L_q(\T^n_\theta)$, then 
\begin{equation}
 N^{-}\big(\dot{H}_{V}^{(p)}\big) \leq C_{g,\nu}(2p,2q)^{2p} \big(\big\|V_{-}\big\|_{L_q}\big)^p, 
 \label{eq:CLR.2} 
\end{equation}
where $C_{g,\nu}(2p,2q)$ is the best constant in the inequality~(\ref{eq:specific-Cwikel-Deltagnu}) for the pair $(2p,2q)$.
\end{theorem}
\begin{proof}
 The proof follows along similar lines as that of Theorem~\ref{thm:CLR.CLR-NCtori}. In the same way as in~(\ref{eq:CLR.proof.monotonicity}) we have $N^{-}(\dot{H}_V^{(p)})\leq N^{-}(\dot{H}_{-V_{-}}^{(p)})$. As $-V_{-}\leq 0$ and $0$ is not in the spectrum of $\dot{H}_0^{(p)}=\dot{\Delta}^{n/2p}_{g,\nu}$, the abstract Birman-Schwinger principle in the version of~\cite{BS:AMST89} (see also~\cite{MP:JMP22, Si:AMS15}) implies that
\begin{equation*}
 N^{-}\big(\dot{H}_{V}^{(p)}\big) \leq   \left(\left\| \dot{\Delta}_{g,\nu}^{-\frac{n}{4p}}\dot{\lambda}_\nu(V_{-}) \dot{\Delta}_{g,\nu}^{-\frac{n}{4p}}\right\|_{\sL_{p,\infty}(\dot{L}(\T^n_\theta;\nu))}\right)^p
 = \left(\left\| \Delta_{g,\nu}^{-\frac{n}{4p}}\lambda_\nu(V_{-}) \Delta_{g,\nu}^{-\frac{n}{4p}}\right\|_{\sL_{p,\infty}^{(\nu)}}\right)^p. 
\end{equation*}
Combining this with the Cwikel-type estimate~(\ref{eq:specific-Cwikel-Deltagnu-2}) immediately gives the result. 
\end{proof}

\begin{remark}
 We can recover from~(\ref{eq:CLR.2}) the other CLR inequality~(\ref{eq:CLR.CLR-inequality}), since, in the same way as in~\cite[Remark~8.13]{MP:JMP22}, it can be shown that $N^{-}(H_V^{(p)})\leq N^{-}(\dot{H}_V^{(p)})+1$. 
\end{remark}

\subsection{Lieb-Thirring inequalities}
The Lieb-Thirring inequalities extend the CLR inequality~(\ref{eq:CLR.2}) to $\gamma$-moments, i.e.,  
\begin{equation*}
\sum_{j\geq 0} \lambda_j^{-}\big(\dot{H}_V^{(p)}\big)^{\gamma}, \qquad \gamma\geq 0. 
\end{equation*}
In particular, for $\gamma=0$ we recover the number of negative eigenvalues $N^{-}(\dot{H}_V^{(p)})$. 

\begin{theorem}[Lieb-Thirring Inequalities on Curved NC Tori]\label{thm:LT-inequality}
Let $\gamma>0$ and $p>1$. If $V=V^*\in L_{p+\gamma}(\T^n_\theta)$, then
 \begin{equation}\label{eq:LT-Inequality}
\sum_{j\geq 0} \lambda_j^{-}\big(\dot{H}_V^{(p)}\big)^{\gamma}\leq L_{g,\nu}(p,\gamma) \tau\big[(V_{-})^{p+\gamma}\big],  
\end{equation}
where the best constant $L_{g,\nu}(p,\gamma)$ satisfies
\begin{equation*}
 L_{g,\nu}(p,\gamma) \leq \gamma \frac{\Gamma(p+1)\Gamma(\gamma)}{\Gamma(p+\gamma+1)}C_{g,\nu}(2p,2p)^{2p}.
\end{equation*}
Here $C_{g,\nu}(2p,2p)$ is the best constant for the inequality~(\ref{eq:specific-Cwikel-Deltagnu}) for the pair $(2p,2p)$. 
\end{theorem}
\begin{proof}
The proof goes along merely the same lines as that of the proof of~\cite[Theorem~8.14]{MP:JMP22}. We only highlight the main steps. We have the standard formula, 
\begin{equation*}
\sum_{j\geq 0} \lambda_j^{-}\big(\dot{H}_V^{(p)}\big)^{\gamma}
  = \gamma \int_0^\infty t^{\gamma-1}N^{-}\big(\dot{H}_V^{(p)};t\big)dt. 
\end{equation*}
For $t\geq 0$, the CLR inequality~(\ref{thm:CLR.CLR-NCtori2}) gives
\begin{equation*}
 N^{-}\big(\dot{H}_V^{(p)};t\big) = N^{-}\big(\dot{H}_{V+t}^{(p)}\big)\leq C_{g,\nu}(2p,2p)^{2p}\tau\big[(V+t)_{-}^p\big]. 
\end{equation*}
Thus, 
\begin{equation*}
  \sum_{j\geq 0} \lambda_j^{-}\big(\dot{H}_V^{(p)}\big)^{\gamma} 
  \leq C_{g,\nu}(2p,2p)^{2p} \gamma  \int_0^\infty t^{\gamma-1} \tau\big[(V+t)_{-}^p\big]\,dt.
\end{equation*}
It is shown in the proof of~\cite[Theorem~8.14]{MP:JMP22} that
\begin{equation*}
  \int_0^\infty t^{\gamma-1} \tau\big[(V+t)_{-}^p\big]\,dt = \frac{\Gamma(p+1)\Gamma(\gamma)}{\Gamma(p+\gamma+1)}  \tau\big[(V_{-})^{p+\gamma}\big]. 
\end{equation*}
This gives the result. 
\end{proof}

\begin{remark}
By proceedings along the lines of~\cite{Il:JST12} we can further get Lieb-Thirring inequalities for $p=1$ with even better constants. See~\cite[Remark~18.7]{MP:JMP22} for the discussion on this in the flat case.  
\end{remark}

\subsection{A Sobolev inequality}
The formula~(\ref{eq:Riemannian.curved-one-form-inner-product}) defines an inner product $\acou{\cdot}{\cdot}_{g,\nu}$ on the $\C^\infty(\T^n_\theta)$-module of 1-forms $\Omega^1(\T^n_\theta)$. We denote $\|\cdot\|_{g,\nu}$ the corresponding norm. In particular, for all $u\in \C^\infty(\T^n_\theta)$, we have 
\begin{equation*}
 \big\|du\|^2_{g,\nu}=\acou{du}{du}_{g,\nu}=  \sum_{i,j} 
 \tau\left[   \partial_i(u)^* \nu^{\frac12} g^{ij} \nu^{\frac12}\partial_j(u)\right]. 
\end{equation*}
Note that the above formula continues to make sense for all $u\in W_2^1(\T^n_\theta;g,\nu)$. 

On $\R^n$ the Lieb-Thirring inequality  for $\gamma=1$ and $p=n/2$ is equivalent to a Sobolev inequality for orthonormal families (see~\cite[Theorem 4]{LT:SMP76}; see 
also~\cite{Fr:Survey20, Si:AMS15a}). Likewise, as a consequence of the above Lieb-Thirring inequalities we shall obtain the following Sobolev's inequality. 

\begin{theorem}[Sobolev Inequality on Curved NC Tori]\label{thm:LT.Sobolev}  
Assume $n\geq 3$. For any family $\{u_0, \ldots, u_N\}$ in $\dot{W}_2^1(\T^n_\theta)$ which is orthonormal in $\dot{L}_2s(\T^n_\theta;\nu)$, we have
\begin{equation}\label{eq:LT.Sobolev-ineq}
 \sum_{\ell\leq N} \big\| du_\ell\|^2_{g,\nu} \geq K_{g,\nu} \tau\left[\rho^{\frac{n+2}{n}}\right], \qquad \text{where}\ \rho:=(2\pi)^n\sum_{\ell \leq N}  \nu^{\frac12}u_{\ell}u_\ell^*\nu^{\frac12}
\end{equation}
In particular, if we set $q=2(n+2)n^{-1}$, then for $N=0$ we get
\begin{equation}\label{eq:LT.Sobolev-ineq-special-case}
 \big\| du\|^2_{g,\nu} \geq K_{g,\nu} \big(\|u\|_{L_{q}(\T^n;\nu)}\big)^{q} \qquad \forall u \in \dot{W}_2^1(\T^n_\theta). 
\end{equation}
 Moreover, the best constant $K_{g,\nu}$ is related to the best Lieb-Thirring constant $L_{g,\nu}=L_{g,\nu}(n/2,1)$ by
\begin{equation}\label{eq:LT.Ln-Kn}
 K_{g,\nu}= \frac{n}{n+2} \bigg( \frac{n+2}{2}L_{g,\nu}\bigg)^{-\frac{2}{n}}. 
\end{equation}
 \end{theorem}
 \begin{proof}
The proof goes along similar lines as that of the proof of~\cite[Theorem~8.19]{MP:JMP22}. For any $V\in L_{n/2+1}(\T^n_\theta)$, set $H_{V}:=H_{V}^{(n/2)}=\Delta_{g,\nu}+\lambda_\nu(V)$ and $\dot{H}_V:=\dot{H}_V^{(n/2)}$. By the very definition~(\ref{eq:Riemannian.Laplace-Beltrami-quadratic-form}) of $\Delta_{g,\nu}$ we have 
\begin{equation*}
 Q_{\Delta_{g,\nu}}(u,u)=  \|du\|_{g,\nu}^2 \qquad \forall u\in C^\infty(\T^n_\theta). 
\end{equation*}
This formula continues to hold for all $u\in W_2^1(\T^n_\theta;g,\nu)$. Moreover, by~(\ref{eq:Q_lambda_def}) we have
\begin{equation*}
 Q_{\lambda_\nu(V)}(u,u)=\acou{\lambda_\nu(V)u}{u}=(2\pi)^n\tau\big[u^*\nu^{\frac12}V\nu^{\frac12}u\big]=(2\pi)^n\tau\big[\nu^{\frac12}uu^*\nu^{\frac12}V\big]. 
\end{equation*}

Assume $V\geq 0$, and let $\{u_0, \ldots, u_N\}$ be  any family in $\dot{W}_2^1(\T^n_\theta)$ which is orthonormal in $\dot{L}_2(\T^n_\theta;\nu)$. We have
\begin{align}
 \sum_{\ell \leq N} Q_{\dot{H}_{-V}}(u_\ell,u_\ell)&= \sum_{\ell \leq N} Q_{{H}_{-V}}(u_\ell,u_\ell)\nonumber \\
 & =  \sum_{\ell \leq N} Q_{\Delta_{g,\nu}}(u_\ell,u_\ell) -  \sum_{\ell \leq N} Q_{\lambda_\nu(V)}(u_\ell,u_\ell)\nonumber \\
 & =  \sum_{\ell \leq N} \|du_\ell\|_{g,\nu}^2 - \sum_{\ell \leq N} (2\pi)^n\tau\big[\nu^{\frac12}u_\ell u_\ell^*\nu^{\frac12}V\big]\nonumber \\
 &= \sum_{\ell \leq N} \|du_\ell\|_{g,\nu}^2 - \tau\big[\rho V\big]\label{eq:LT.quadratic-form-expression}, 
\end{align}
where we have set $\rho:= (2\pi)^n\sum  \nu^{1/2}u_{\ell}u_\ell^*\nu^{1/2}$. Moreover, as $\{u_0, \ldots, u_N\}$ is an orthonormal family, in the same way as in 
the proof of~\cite[Theorem~8.19]{MP:JMP22} (and in various other proofs of the Sobolev inequality), the variational principal for negative eigenvalues gives 
\begin{equation}\label{eq:LT.spectral-lower-bound}
 \sum_{\ell \leq N} Q_{\dot{H}_{-V}}(u_\ell,u_\ell) \geq -\sum_{j\geq 0}\lambda_j^{-}\big(\dot{H}_{-V}\big). 
\end{equation}

As $-V\leq 0$, the Lieb-Thirring inequality~(\ref{eq:LT-Inequality}) for $\gamma=1$ and $p=n/2$ delivers 
\begin{equation*}
\sum_{j\geq 0}\lambda_j^{-}\big(\dot{H}_{-V}\big) \leq L_{g,\nu}\tau\big[V^{\frac{n}{2}+1}\big]. 
\end{equation*}
Combining this with~(\ref{eq:LT.quadratic-form-expression}) and~(\ref{eq:LT.spectral-lower-bound}) leads to the inequality
\begin{equation*}
  \sum_{\ell \leq N} \|du_\ell\|_{g,\nu}^2  \geq \tau\big[\rho V-L_{g,\nu}V^{\frac{n}{2}+1}\big]. 
\end{equation*}

The above inequality is true for all $0\leq V\in L_{\frac{n}{2}+1}(\T^n_\theta)$. In the same way as in the proof of~\cite[Theorem~8.19]{MP:JMP22} if we take 
$V=(\frac12(n+2)L_{g,\nu})^{-2/n}\rho^{2/n}$, then $V\in L_{\frac{n}{2}+1}(\T^n_\theta)$, and we have 
$\rho V - L_{g,\nu} V^{\frac{n}{2}+1}=\frac{n}{n+2}(\frac12 (n+2)L_{g,\nu})^{-2/n}\rho^{(n+2)/n}$. This gives the inequality~(\ref{eq:LT.Sobolev-ineq}) with a best constant $K_{g,\nu}$ such that 
\begin{equation}\label{eq:LT.Kgnu-lower-bound}
  K_{g,\nu}\geq \frac{n}{n+2} \bigg( \frac{n+2}{2}L_{g,\nu}\bigg)^{-\frac{2}{n}}.
\end{equation}

Note that $\rho = \sum (\hat{\lambda}(\sqrt{\nu})u_\ell)(\hat{\lambda}(\sqrt{\nu})u_\ell)^*= \sum |(\hat{\lambda}(\sqrt{\nu})u_\ell)^*|^2$. Thus, if we set  $q=2(n+2)n^{-1}$, then for $N=0$ the inequality~(\ref{eq:LT.Sobolev-ineq}) means that, for all $u\in \dot{W}_2^1(\T^n_\theta;g,\nu)$, we have
\begin{equation*}
 \|du\|^2_{g,\nu}\geq K_{g,\nu} \tau\left[ \big|(\hat{\lambda}(\sqrt{\nu})u)^*\big|^{q}\right]. 
\end{equation*}
Observe that
\begin{equation}
  \tau\left[ \big|(\hat{\lambda}(\sqrt{\nu})u)^*\big|^{q}\right]= \big\| (\hat{\lambda}(\sqrt{\nu})u)^*\big\|_{L_q}^q= \big\| \hat{\lambda}(\sqrt{\nu})u\big\|_{L_q}^q =
   \big( \big\|u\big\|_{L_q(\T^n_\theta;\nu)}\big)^q. 
\end{equation}
 Therefore, we obtain~(\ref{eq:LT.Sobolev-ineq-special-case}).

Conversely, let $V=V^*\in L_{n/2+1}(\T^n_\theta)$  and assume that $N^{-}(\dot{H}_V)\neq 0$ (otherwise the LT inequalities~(\ref{eq:LT-Inequality}) are trivially satisfied). 
 Set $N=N^{-}(\dot{H}_V)-1$, and let $\{u_0,\ldots, u_{N}\}$ be an orthonormal family of eigenvectors in $\dot{W}_2^{2}(\T^n_\theta)$ 
such that $\dot{H}_Vu_j=-\lambda_j^{-}(H_V)u_j$ for $j\leq N$. In particular, we have 
\begin{equation*}
 \sum_{j\geq 0}\lambda_j^{-}\big(\dot{H}_{V}\big)= - \sum_{j\leq N} 
 Q_{\dot{H}_{V}}(u_j,u_j)\leq -  \sum_{j\leq N} Q_{\dot{H}_{-V_{-}}}(u_j,u_j). 
\end{equation*}
Combining this with~(\ref{eq:LT.quadratic-form-expression}) then gives
\begin{equation*}
   \sum_{j\geq 0}\lambda_j^{-}\big(\dot{H}_{V}\big)\leq  - \sum_{j\leq N} \|du_j\|_{g,\nu}^2 +\tau\big[\rho V_{-}\big]. 
\end{equation*}

Set  $q=1+n/2$ and $r=1+2/n$. Note that $q^{-1}+r^{-1}=1$. The Sobolev inequality~(\ref{eq:LT.Sobolev-ineq}) delivers
\begin{equation*}
  \sum_{j\leq N} \|du_j\|_{g,\nu}^2 \geq K_{g,\nu} \tau\big[\rho^r\big]. 
\end{equation*}
Moreover, in the same way as in the proof of~\cite[Theorem~8.19]{MP:JMP22} it can be shown that, for any $\epsilon>0$, we have 
\begin{equation*}
 \tau\big[\rho V_{-}\big] \leq \frac{1}{q}\epsilon^q \tau\big[V_{-}^q\big] + \frac{1}{r}\epsilon^{-r} \tau\big[\rho^r\big]. 
\end{equation*}
Thus, 
\begin{equation}\label{eq:LT.Sobolev-eps}
   \sum_{j\geq 0}\lambda_j^{-}\big(\dot{H}_{V}\big) \leq   \left(\frac{1}{r}\epsilon^{-r}- K_{g,\nu}\right)  \tau\big[\rho^r\big] +  \frac{1}{q}\epsilon^q \tau\big[V_{-}^q\big]  .
\end{equation}
In the same way as in the proof of~\cite[Theorem~8.19]{MP:JMP22}, for $\epsilon=(rK_{g,\nu})^{-1/r}$ the first summand in the r.h.s.~of~(\ref{eq:LT.Sobolev-eps}) vanishes. In this case 
$q^{-1}\epsilon^q=\frac{2}{n+2}(\frac{n}{n+2}K_{g,\nu})^{-n/2}$, and so we get 
\begin{equation*}
 \sum_{j\geq 0}\lambda_j^{-}\big(\dot{H}_{V}\big) \leq \frac{2}{n+2}\left(\frac{n}{n+2}K_{g,\nu}\right)^{-\frac{n}{2}}\!\!\!\cdot\tau\big[V_{-}^q\big] .  
\end{equation*}
This recovers the Lieb-Thirring inequality~(\ref{eq:LT-Inequality}) for $\gamma=1$ and $p=n/2$ and shows that~(\ref{eq:LT.Kgnu-lower-bound}) is actually an equality. This completes the proof of~(\ref{eq:LT.Ln-Kn}).
\end{proof}

\section{Spectral Asymptotics}\label{sec:Weyl}
In this section, as a further application of the Cwikel-type estimates of this paper, we explain how these estimate lead to spectral asymptotics for Birman-Schwinger-type operators and fractional Schr\"odinger operators associated with powers of Laplace-Beltrami operator and $L_p$-potentials. In particular, this supersedes a conjecture in the prequel paper~\cite{MP:JMP22}. 

In what follows, if $A$ a selfadjoint compact operator we denote by $\pm \lambda_j^\pm(A)$ its positive/negative eigenvalues in such a way that
\begin{equation*}
 \lambda_0^\pm(A)\geq  \lambda_1^\pm(A) \geq \lambda_2^\pm(A) \geq \cdots >0, 
\end{equation*}
where each eigenvalue is repeated according to multiplicity. In the 70s Birman-Solomyak~\cite{BS:VLU77, BS:VLU79, BS:SMJ79} established a very general Weyl's law for negative order  \psidos\ on closed manifolds and Euclidean spaces. By the Birman-Schwinger principle these asymptotics imply a semiclassical Weyl's laws for fractional Schr\"odinger operators. They also imply a stronger form of Connes' trace theorem~(\ref{eq:Connes-trace-thm}). A ``soft proof'' of Birman-Solomyak's Weyl's law on closed manifolds is given in~\cite{Po:Weyl}. 

We conjecture the following analogue for NC tori of Birman-Solomyak's Weyl's law. 

\begin{conjecture}\label{conj:Bir-Sol-WL}
 Let $P=P^*\in \Psi^{-m}(\T^n_\theta)$, $m>0$, have principal symbol $\rho_{-m}(\xi)$, and set $p=nm^{-1}$. Then
\begin{equation}\label{eq:Bir-Sol-WL}
  \lim_{j\rightarrow \infty} j^{\frac1{p}} \lambda_j^{\pm} (P)= 
 \left(\frac{1}{n} \int_{\bS^{n-1}} \tau\left[\big(\rho_{-m}(\xi)_{\pm}\big)^p\right] d\xi\right)^{\frac1{p}}. 
\end{equation}
\end{conjecture}

In fact, Weyl's laws for positive/negative eigenvalues for (selfadjoint) elliptic \psidos\ are established in the forthcoming manuscript~\cite{LP:Part2}. As a result, 
the spectral asymptotics~(\ref{eq:Bir-Sol-WL}) hold if $P$ is elliptic. We then can expect to extend them to general \psidos\ by using the same kind of considerations as in~\cite{Po:Weyl}. However, as~\cite{LP:Part2} is still in its final stage of preparation, we prefer to leave the above statement as a conjecture. 

\begin{corollary}\label{cor:Weyl.PsiDOs}
 Assume Conjecture~\ref{conj:Bir-Sol-WL} holds. Suppose that $p\neq 1$ and $q=\max\{p,1\}$, or $p=1<q$, and set $m=n/2p$. Let $P \in \Psi^{m}(\T^n_\theta)$ have principal symbol $\rho_{-m}(\xi)$. Then, for every $x=x^*\in L_q(\T^n_\theta)$, we have
\begin{equation*}
 \lim_{j\rightarrow \infty} j^{\frac1{p}} \lambda_j^{\pm} \big(P^*\lambda(x)P\big)  = 
 \left(\frac{1}{n} \int_{\bS^{n-1}} \tau\big[ \left(\rho_{-m}(\xi)^*x\rho_{-m}(\xi)\right)_{\pm}^p\big] d\xi\right)^{\frac1{p}}. 
 \end{equation*}
\end{corollary}
\begin{proof}
For $x\in L_q(\T^n_\theta)_\sa:=\{x\in L_q(\T^n_\theta); \ x^*=x\}$ set
\begin{equation*}
 \omega_\pm(x)= \left(\frac{1}{n} \int_{\bS^{n-1}} \tau\big[ \left(\rho_{-m}(\xi)^*x\rho_{-m}(\xi)\right)_{\pm}^p\big] d\xi\right)^{\frac1{p}}. 
\end{equation*}
Note that $(x,\xi)\rightarrow \tau[ (\rho_{-m}(\xi)^*x\rho_{-m}(\xi))_{\pm}^p]$ are continuous functions on $L_q(\T^n_\theta)_{\sa}\times \bS^{n-1}$, because $(x,\xi)\rightarrow \rho_{-m}(\xi)^*x\rho_{-m}(\xi)$ is a continuous map from $L_q(\T^n_\theta)_{\sa}\times \bS^{n-1}$ to $L_q(\T^n_\theta)_{\sa}$ and the functions $y\rightarrow \tau[y_\pm^p]=2^{-p}(\||y|\pm y\|_{L_p})^p$ are continuous on $L_q(\T^n_\theta)_{\sa}$ (since $q\geq p$). It then follows that $\omega_\pm(x)$ is a continuous function on $L_q(\T^n_\theta)_{\sa}$.  

If  $x=x^*\in C^\infty(\T^n_\theta)$, then $P^*\lambda(x)P$ is a selfadjoint operator in $\Psi^{2m}(\T^n_\theta)$ whose principal symbol is $\rho_m(\xi)^*x\rho_m(\xi)$, and so by~(\ref{eq:Bir-Sol-WL}) we have
\begin{equation*}
 \lim_{j\rightarrow \infty} j^{\frac1{p}} \lambda_j^{\pm} \big(P^*\lambda(x)P\big)  = \omega_\pm(x). 
\end{equation*}
By the 2nd part of Theorem~\ref{thm:curved-Cwikel} the map $x\rightarrow P^*\lambda(x)P$ is a continuous linear map from $L_q(\T^n_\theta)$ to $\sL_{p,\infty}$. Thus, if
$x\in L_q(\T^n_\theta)_{\sa}$ and $(x_\ell)_{\ell\geq 0}$  is a sequence of selfadjoint elements of $C^\infty(\T^n_\theta)$ converging to $x$ in $L_q(\T^n_\theta)$, then $P^*\lambda(x_\ell)P\rightarrow P^*\lambda(x)P$ in $\sL_{p,\infty}$. Birman-Solomyak's perturbation theory (see~\cite[Theorem~4.1]{BS:JFAA70}) then ensures that
\begin{equation*}
  \lim_{j\rightarrow \infty} j^{\frac1{p}} \lambda_j^{\pm} \big(P^*\lambda(x)P\big)  = 
  \lim_{\ell \rightarrow \infty}  \lim_{j\rightarrow \infty} j^{\frac1{p}} \lambda_j^{\pm} \big(P^*\lambda(x_\ell)P\big)=   \lim_{\ell \rightarrow \infty}\omega(x_\ell)=\omega(x). 
\end{equation*}
This proves the result. 
\end{proof}

From now on we let $g=(g_{ij})$ be a Riemannian metric  and  $\nu\in \GL_1^+(C^\infty(\T^n_\theta))$ a smooth positive density on $\T^n_\theta$. 

\begin{theorem}\label{thm:Weyl-laws}
 Assume Conjecture~\ref{conj:Bir-Sol-WL} holds. Suppose that $p\neq 1$ and $q=\max\{p,1\}$, or $p=1<q$. For all $x=x^*\in L_q(\T^n_\theta)$, we have
 \begin{equation}\label{eq:Weyl.spectral-asymp}
 \lim_{j\rightarrow \infty} j^{\frac1{p}} \lambda_j^{\pm} \left(\Delta_{g,\nu}^{-\frac{n}{4p}}\lambda_\nu(x)\Delta_{g,\nu}^{-\frac{n}{4p}}\right)  = 
 \left(\frac{1}{n} \int_{\bS^{n-1}} \tau\left[ \left(|\xi|_g^{-\frac{n}{2p}}x|\xi|_g^{-\frac{n}{4p}}\right)_{\pm}^p\right] d\xi\right)^{\frac1{p}}. 
 \end{equation}
\end{theorem}
\begin{proof}
 Recall that $\hat{\lambda}(\sqrt{\nu}):L_2(\T^n_\theta;\nu)\rightarrow L_2(\T^n_\theta)$ is a unitary isomorphism. Under this isomorphism $\Delta_{g,\nu}$ corresponds to the operator, 
 \begin{equation*}
 \hat{\Delta}_{g,\nu}:=\hat{\lambda}\big(\sqrt{\nu}\big)\Delta_{g,\nu}\hat{\lambda}\big(\sqrt{\nu}\big)^{-1}= \nu^{\frac12}\Delta_{g,\nu}\nu^{-\frac12}. 
\end{equation*}
We get an elliptic 2nd order differential operator whose principal symbol is $\nu^{1/2}(\nu^{-1/2}|\xi|_g^2\nu^{1/2})\nu^{-1/2}=|\xi|^2_g$. It is formally selfadjoint with respect to the original inner product~(\ref{eq:NCtori.innerproduct-L2}) of $L_2(\T^n_\theta)$. In particular, we get a selfadjoint operator on $L_2(\T^n_\theta)$ with domain $W_2^2(\T^n_\theta)$. 
If $s>0$, then $\hat{\Delta}_{g,\nu}^{-s}=\hat{\lambda}(\sqrt{\nu})^{-1} \Delta_{g,\nu}^{-s}\hat{\lambda}(\sqrt{\nu})$, and  so by using Proposition~\ref{prop:powersLB} we see that  $\hat{\Delta}_{g,\nu}^{-s}$ is a (selfadjoint) \psido\ of order $-2s$ whose principal symbol is $|\xi|_g^{-2s}$. This result remains valid for any power 
$\hat{\Delta}_{g,\nu}^{z}$, $z\in \C$. 

If $x\in L_\infty(\T^n_\theta)$, then by using~(\ref{eq:Riem.lambdanu}) we get
\begin{align*}
 \Delta_{g,\nu}^{-\frac{n}{4p}}\lambda_\nu(x)\Delta_{g,\nu}^{-\frac{n}{4p}} &=  \hat{\lambda}\big(\sqrt{\nu}\big)^{-1}
 \hat{\Delta}_{g,\nu}^{-\frac{n}{4p}} \hat{\lambda}\big(\sqrt{\nu}\big) \cdot \lambda\big(\nu^{-\frac12}x\nu^{\frac12}\big)\cdot
\hat{\lambda}\big(\sqrt{\nu}\big)^{-1}\hat{\Delta}_{g,\nu}^{-\frac{n}{4p}} \hat{\lambda}\big(\sqrt{\nu}\big)\\
 &= \hat{\lambda}\big(\sqrt{\nu}\big)^{-1}  \hat{\Delta}_{g,\nu}^{-\frac{n}{4p}}\lambda(x)\hat{\Delta}_{g,\nu}^{-\frac{n}{4p}} \hat{\lambda}\big(\sqrt{\nu}\big). 
\end{align*}
The l.h.s.\ and the lower r.h.s\ are both continuous maps from $L_q(\T^n_\theta)$ to $\sL(L_2(\T^n_\theta;\nu))$ (\emph{cf}.~Remark~\ref{rmk:lambda-to-lambda-nu-twosided}). Therefore, the above equalities continue to hold for all $x\in L_q(\T^n_\theta)$. 

It follows from all this that the operators $\Delta_{g,\nu}^{-n/4p}\lambda_\nu(x)\Delta_{g,\nu}^{-n/4p}$ and $\hat{\Delta}_{g,\nu}^{-n/4p}\lambda_\nu(x)\hat{\Delta}_{g,\nu}^{-n/4p}$ have exactly the same positive and negative eigenvalues with the same multiplicities. As $\hat{\Delta}_{g,\nu}^{-n/4p}$ is a selfadjoint \psido\ of order $-n/2p$ whose principal symbol is $|\xi|_g^{-n/2p}$, by using Corollary~\ref{cor:Weyl.PsiDOs} we get
\begin{align*}
 \lim_{j\rightarrow \infty} j^{\frac1{p}} \lambda_j^{\pm} \left(\Delta_{g,\nu}^{-\frac{n}{4p}}\lambda_\nu(x)\Delta_{g,\nu}^{-\frac{n}{4p}}\right)  & = 
  \lim_{j\rightarrow \infty} j^{\frac1{p}} \lambda_j^{\pm} \left(\hat{\Delta}_{g,\nu}^{-\frac{n}{4p}}\lambda(x)\hat{\Delta}_{g,\nu}^{-\frac{n}{4p}}\right) \\
&=  \left(\frac{1}{n} \int_{\bS^{n-1}} \tau\left[ \left(|\xi|_g^{-\frac{n}{2p}}x|\xi|_g^{-\frac{n}{4p}}\right)_{\pm}^p\right] d\xi\right)^{\frac1{p}}. 
\end{align*}
The proof is complete. 
\end{proof}

\begin{remark}
 For $p=1$ the spectral asymptotics~(\ref{eq:Weyl.spectral-asymp}) imply that $\Delta_{g,\nu}^{-n/4}\lambda_\nu(x)\Delta_{g,\nu}^{-n/4}$ is a measurable operator in the sense of~(\ref{eq:NCG-measurability}). They also allow us to recover the integration formula~(\ref{eq:L1+-curved-formula}). In fact, it can be further shown that Weyl's laws of the form~(\ref{eq:Weyl.spectral-asymp}) imply that $\Delta_{g,\nu}^{-n/4}\lambda_\nu(x)\Delta_{g,\nu}^{-n/4}$ is actually strongly measurable (see~\cite{Po:Weyl}). Thus, Theorem~\ref{thm:Weyl-laws} yields a stronger form of Theorem~\ref{thm:Lp-curved-formula}. 
\end{remark}

If $p$ and $q$ are as in Theorem~\ref{thm:Weyl-laws} and $V=V^*\in L_q(\T^n_\theta)$, then we can define the (semiclassical) fractional Schr\"odinger operators $h^{n/p}\Delta_{g,\nu}^{n/2p}+\lambda_\nu(V)$, $h>0$, as in the previous section. The abstract form of the Birman-Schwinger principle~\cite[Lemma~1.4]{BS:AMST89} implies that 
\begin{equation*}
 \lim_{h\rightarrow 0^+} h^{n}N^{-}\left(h^{\frac{n}{p}}\Delta_{g,\nu}^{\frac{n}{2p}}+\lambda_\nu(V)\right)= 
  \lim_{j\rightarrow \infty} j \lambda_j^{-} \left(\Delta_{g,\nu}^{-\frac{n}{4p}}\lambda_\nu(x)\Delta_{g,\nu}^{-\frac{n}{4p}}\right)^p. 
\end{equation*}
This formula goes back to Birman-Solomyak (see, e.g., \cite[Theorem~10.1]{BS:TMMS72} and~\cite[Appendix~6]{BS:AMST80}). Readers who are unfamiliar with the Birman-Schwinger principle may also find a proof of the above equality in~\cite{Po:Weyl}. Combining it with Theorem~\ref{thm:Weyl-laws} we immediately arrive at the following semiclassical Weyl's laws. 

\begin{theorem}\label{thm:SCWeyl}
 Under the assumptions of Theorem~\ref{thm:Weyl-laws}, for all $V=V^*\in L_q(\T^n_\theta)$, we have
 \begin{equation}\label{eq:SCWeyl-fractional}
\lim_{h\rightarrow 0^+} h^{n}N^{-}\left(h^{\frac{n}{p}}\Delta_{g,\nu}^{\frac{n}{2p}}+\lambda_\nu(V)\right)= \frac{1}{n} \int_{\bS^{n-1}} \tau\left[ \left(|\xi|_g^{-\frac{n}{2p}}V|\xi|_g^{-\frac{n}{2p}}\right)_{-}^p\right] d\xi. 
\end{equation}
\end{theorem}

In the case $\nu=1$ and $g$ is the flat Euclidean metric $g_0=(\delta_{ij})$, the Laplace-Beltrami operator $\Delta_{g,\nu}$ is just the ordinary Laplacian $\Delta=-(\partial_1^2+\cdots + \partial_n^2)$. In this case, Theorem~\ref{thm:SCWeyl} specializes to the following result. 

\begin{corollary}
 Under the assumptions of Theorem~\ref{thm:Weyl-laws}, for all $V=V^*\in L_q(\T^n_\theta)$, we have
 \begin{equation}\label{eq:SCWeyl-fractional-flat}
\lim_{h\rightarrow 0^+} h^{n}N^{-}\left(h^{\frac{n}{p}}\Delta^{\frac{n}{2p}}+\lambda(V)\right)= c(n)\tau\big[(V_{-})^p  \big], \qquad c(n):=\frac1{n}|\bS^{n-1}|. 
\end{equation}
\end{corollary}
\begin{proof}
 We have $|\xi|_{g_0}=|\xi|$, and so for $g=g_0$ the r.h.s.~of~(\ref{eq:SCWeyl-fractional}) is equal to
 \begin{equation*}
  \frac{1}{n} \int_{\bS^{n-1}} \tau\left[ \left(|\xi|^{-\frac{n}{2p}}V|\xi|^{-\frac{n}{2p}}\right)_{-}^p\right] d\xi=  \frac{1}{n} \int_{\bS^{n-1}} \tau\left[ \left(V\right)_{-}^p\right] d\xi
  = \frac{1}{n}\tau\big[(V_{-})^p\big]|\bS^{n-1}|. 
\end{equation*}
This gives the result. 
\end{proof}

\begin{remark}
The semiclassical Weyl's law~(\ref{eq:SCWeyl-fractional-flat}) was conjectured in~\cite{MP:JMP22}.  We refer to~\cite{MSZ:arXiv21} for an alternative proof for $p=n/2$ and $n\geq 3$. Note that the approach of~\cite{MSZ:arXiv21} does not extend to non-flat Riemannian metrics or to the dimension $n=2$  (see Remark~\ref{rmk:MSZ} below). 
\end{remark}

It also worth looking at the special case of conformally flat metrics. 

\begin{corollary}
 Suppose that $g=k^2g_0$, $k\in \GL_1^+(C^\infty(\T^n_\theta))$.  Under the assumptions of Theorem~\ref{thm:Weyl-laws}, for all $V=V^*\in L_q(\T^n_\theta)$, we have
 \begin{equation}\label{eq:SCWeyl-conf-flat}
\lim_{h\rightarrow 0^+} h^{n}N^{-}\left(h^{\frac{n}{p}}\Delta_g^{\frac{n}{2p}}+\lambda_{\nu(g)}(V)\right)= c(n) \tau\left[ \left(k^{\frac{n}{2p}}Vk^{\frac{n}{2p}}\right)_{-}^p\right].
\end{equation}
where $c(n)$ is as in~(\ref{eq:SCWeyl-fractional-flat}). 
\end{corollary}
\begin{proof}
As $g=k^2g_0$ we have $|\xi|_g=|\xi|_{g_0}k^{-1}=|\xi|k^{-1}$, and so in this case the r.h.s.~of~(\ref{eq:SCWeyl-fractional}) is
 \begin{align*}
  \frac{1}{n} \int_{\bS^{n-1}} \tau\left[ \left((k^{-1}|\xi|)^{-\frac{n}{2p}}V(k^{-1}|\xi|)^{-\frac{n}{2p}}\right)_{-}^p\right] d\xi &=  \frac{1}{n} \int_{\bS^{n-1}} 
  \tau\left[ \left(k^{\frac{n}{2p}}Vk^{\frac{n}{2p}}\right)_{-}^p\right] d\xi\\ 
  &= \frac{1}{n} \tau\left[ \left(k^{\frac{n}{2p}}Vk^{\frac{n}{2p}}\right)_{-}^p\right] |\bS^{n-1}|. 
\end{align*}
The result then follows from Theorem~\ref{thm:SCWeyl}. 
\end{proof}

\begin{remark}\label{rmk:MSZ}
 For $p=n/2$ the asymptotic~(\ref{eq:SCWeyl-conf-flat}) gives
\begin{equation*}
\lim_{h\rightarrow 0^+} h^{n}N^{-}\left(h^{2}\Delta_g+\lambda_{\nu(g)}(V)\right)= c(n) \tau\left[ (kVk)_{-}^{\frac{n}{2}}\right].
\end{equation*}
As $\tilde{\nu}(g)=\nu(g)=k^n$, the integration formula~(\ref{eq:L1+-curved-formula}) gives 
\begin{equation*}
\bint (V_{-})^{\frac{n}{2}}\Delta_g^{-n/2}=c(n)\tau\left[(V_{-})^{\frac{n}{2}}k^{n}\right].  
\end{equation*}
In general, $\tau[(V_{-})^{n/2}k^{n}] \neq \tau[ (kVk)_{-}^{n/2}]$, and so we have
\begin{equation*}
\lim_{h\rightarrow 0^+} h^{n}N^{-}\left(h^{2}\Delta_g+\lambda_{\nu(g)}(V)\right)\neq \bint (V_{-})^{\frac{n}{2}}\Delta_g^{-\frac{n}{2}}. 
\end{equation*}
This shows that the semiclassical Weyl's law for spectral triples established in~\cite{MSZ:arXiv21} need not hold for non-flat Riemannian metrics on $\T^n_\theta$.  
\end{remark}


\begin{thebibliography}{99}
\bibitem{Ba:CRAS88} Baaj, S.: \emph{Calcul pseudo-diff\'erentiel et produits crois\'es de $C^*$-alg\'ebres. I, II}. 
C.\ R.\ Acad.\ Sc.\ Paris, s\'er.~I, \textbf{307} (1988), 581--586, 663--666.

\bibitem{BES:JMP94} Bellissard, J.; van Elst, A.; Schulz-Baldes, H.: 
\emph{The noncommutative geometry of the quantum Hall effect}. 
J.\ Math.\ Phys.\ \textbf{35} (1994), 5373--5451.

\bibitem{BS:JFAA70} Birman, M.Sh.; Solomyak, M.Z.: \emph{The leading term of the spectral asymptotics for ``non-smooth'' elliptic problems}.  Funct.\ Anal.\ Appl.\ \textbf{4} (1970), 265--275.

\bibitem{BS:TMMS72} Birman, M.Sh.; Solomyak, M.Z.: \emph{Spectral asymptotics of
nonsmooth elliptic operators. I}.  Trudy Moskov.\ Mat.\ Obshch.\
\textbf{27} (1972), 3--52 (Russian). 
Trans.\ Moscow Math.\ Soc.\ \textbf{27} (1972), 1--52 (1975) (English).  

\bibitem{BS:VLU77} Birman, M.Sh.; Solomyak, M.Z.:  \emph{Asymptotic behavior of the spectrum of pseudodifferential operators with anisotropically homogeneous symbols}. (Russian) Vestnik Leningrad.\ Univ.\ \textbf{13} (1977), no.\ 3, 13--21. (English) Vestn.\ Leningr.\ Univ.\, Math.\ \textbf{10} (1982), 237--247.

\bibitem{BS:VLU79} Birman, M.S.; Solomyak, M.Z.: \emph{Asymptotic behavior of the spectrum of pseudodifferential operators with anisotropically homogeneous symbols. II}.  Vestnik Leningrad.\ Univ.\ Mat.\ Mekh.\ Astronom.\ 13, no.\ 3 (1979), 5--10 (Russian). 

\bibitem{BS:SMJ79} Birman, M.S., Solomyak, M.Z.:  \emph{Asymptotics of the spectrum of variational problems on solutions of elliptic equations}. 
Sib.\ Math J.\ \textbf{20} (1979), 1--15. 

\bibitem{BS:AMST80} Birman, M.Sh., Solomyak, M.Z.:
 \emph{Quantitative analysis in Sobolev imbedding theorems and applications to spectral theory}.   Izdanie Inst.\ Mat.\ Akad.\ Nauk Ukrain.\ SSR, Kiev, 1974 (Russian). American Mathematical Society Translations, Series 2, 114. American Mathematical Society, Providence, R.I., 1980.
 
 \bibitem{BS:Book} Birman, M.Sh., Solomyak, M.Z.: \emph{Spectral theory of self-adjoint operators in Hilbert space}, D.\ Reidel, Boston, MA, 1987.

\bibitem{BS:AMST89} Birman, M.Sh.; Solomyak, M.Z.: 
\emph{Schr\"odinger operator. Estimates for number of bound states as function-theoretical problem.} 
Spectral theory of operators (Novgorod, 1989), 1--54. Amer.\ Math.\ Soc.\ Transl.\ Ser.\ 2, \textbf{150}, Amer.\ Math.\ Soc., Providence, RI, 1992.
 
 \bibitem{BCR:RMP16} Bourne, C.; Carey, A.L.; Rennie, A.: \emph{A non-commutative framework to topological insulators}. 
Rev.\ Math.\ Phys.\ \textbf{28} (2016), no.\ 2, 1650004, 51 pp..s 

\bibitem{Co:CRAS80} Connes, A.: \emph{C*-alg\`ebres et g\'eom\'etrie differentielle}.  C.\ R.\ Acad.\ Sc.\ Paris, s\'er.~A, \textbf{290} (1980), 599--604.

\bibitem{Co:AIM81} Connes, A.: \emph{An analogue of the Thom isomorphism for crossed products of a $C^*$-algebra by an action of $\R$}. Adv.\ Math.\ \textbf{39} (1981), 31--55.

\bibitem{Co:CMP88} Connes, A.:  \emph{The action functional in non-commutative geometry}. 
Comm.\ Math.\ Phys.\ \textbf{117} (1988), 673--683.

\bibitem{Co:NCG} Connes, A.: \emph{Noncommutative geometry}. Academic Press, Inc.,\ San Diego, CA, 1994.

 \bibitem{Co:Survey19} Connes, A.: \emph{Noncommutative geometry, the spectral standpoint}. 
New Spaces in Physics, Volume 2, pp.\ 23--84, Cambridge University Press. 

\bibitem{CMSZ:ETDS19} Connes, A.;  McDonald, E.; Sukochev, F.;  Zanin, D.: \emph{Conformal trace theorem for Julia sets of quadratic polynomials}. 
Ergodic Theory and Dynamical Systems \textbf{39} (2019), 2481--2506. 

\bibitem{CF:MJM19} Connes, A.; Fathizadeh, F.: \emph{The term $a_4$ in the heat kernel expansion of noncommutative tori}. 
M\"unster J.\ Math.\ \textbf{12} (2019), 239--410. 

\bibitem{CM:JAMS14} Connes, A.; Moscovici, H.: \emph{Modular curvature for noncommutative two-tori}. J.\ Amer.\ Math.\ Soc.\ \textbf{27} (2014), no.\ 3, 639--684.

\bibitem{CT:Baltimore11} Connes, A.; Tretkoff, P.: \emph{The Gauss-Bonnet theorem for the noncommutative two torus}. Noncommutative geometry, arithmetic, and related topics, pp.\ 141--158, Johns Hopkins Univ.\ Press, Baltimore, MD, 2011.

\bibitem{Cw:AM77} Cwikel, M.: \emph{Weak type estimates for singular values and the number of bound states of Schr\"odinger operators}. Ann.\ of Math.\ (2), 106(1):93--100, 1977.

\bibitem{DS:SIGMA15} Dabrowski, L.;  Sitarz, A.: \emph{An asymmetric noncommutative torus}. SIGMA
Symmetry Integrability Geom.\ Methods Appl.\ \textbf{11} (2015), Paper 075, 11pp.. 

\bibitem{Da:CMP83} Daubechies, I.: \emph{An uncertainty principle for fermions with generalized kinetic energy}. Comm.\ Math.\ Phys.\ \textbf{90} (1983), 511--520. 

\bibitem{Di:CRAS66} Dixmier, J.: \emph{Existence de traces non normales.} C.\ R.\ Acad.\ Sci.\ Paris\ S\'r.\ A-B \textbf{262} (1966), A1107--A1108. 


\bibitem{FK:PJM86} Fack, T.;  Kosaki, H.: \emph{Generalized $s$-numbers of $\tau$-measurable operators}. Pacific J.\ Math.\ \textbf{123} (1986) 269--300.

\bibitem{FK:JNCG12} Fathizadeh, F.; Khalkhali, M.: \emph{The Gauss-Bonnet theorem for noncommutative two tori with a general conformal structure}. J.\ Noncommut.\ Geom.\ \textbf{6} (2012), no.\ 3, 457--480.

\bibitem{FK:JNCG15} Fathizadeh, F.; Khalkhali, M.: \emph{Scalar curvature for noncommutative four-tori}. 
J.\ Noncommut.\ Geom.\ \textbf{9} (2015), no.\ 2, 473--503. 

\bibitem{FK:Survey19} Fathizadeh, F.; Khalkhali, M.: \emph{Curvature in noncommutative geometry}. \emph{Advances in Noncommutative Geometry}, Springer, 2019, pp.~321--420. 

\bibitem{FGK:JNCG19} Floricel, R.; Ghorbanpour, A.; Khalkhali, M.: \emph{The Ricci curvature in noncommutative geometry}. 
J.\ Noncommut.\ Geom.\ \textbf{13} (2019), 269--296. 

\bibitem{Fr:JST14}  Frank, R.: \emph{Cwikel's theorem and the CLR inequality}, J.\ Spectr.\ Theory \textbf{4} (2014) 1--21.

\bibitem{Fr:Survey20}  Frank, R.: \emph{The Lieb-Thirring inequalities: Recent results and open problems}. Preprint \texttt{arXiv:2007.09326}, 46 pp..
 
\bibitem{FLW:Book} Frank, R.; Laptev, A.; Weidl. T.: \emph{Lieb-Thirring inequalities}. To appear. 

\bibitem{GK:AMS69} Gohberg, I. C., Krein, M. G.: \emph{Introduction to the theory of linear nonselfadjoint operators}. Translations of Mathematical Monographs, Vol. 18 American Mathematical Society, Providence, R.I. 1969 xv+378 pp. 

\bibitem{GK:arXiv18} Ghorbanpour, A.; Khalkhali, M.: \emph{Spectral geometry of functional metrics on noncommutative tori}. Preprint \texttt{arXiv:1811.04004}, 46 pp..

\bibitem{Gu:AIM85} Guillemin, V.: \emph{A new proof of Weyl's formula on the asymptotic distribution of eigenvalues}. 
 Adv.\ Math.\ \textbf{55} (1985), no. 2, 131--160. 
 
\bibitem{HLP:IJM19a} Ha, H.; Lee, G.; Ponge, R.: \emph{Pseudodifferential calculus on noncommutative tori, I. Oscillating integrals}. Int.\ J.\ Math.\ \textbf{30} (2019), 1950033 (74 pages). 

\bibitem{HLP:IJM19b} Ha, H.; Lee, G.; Ponge, R.: \emph{Pseudodifferential calculus on noncommutative tori, II. Main properties}. Int.\ J.\ Math.\ \textbf{30} (2019), 1950034 (73 pages).  

\bibitem{HP:JGP20} Ha, H.; Ponge, R.: \emph{Laplace-Beltrami operators on noncommutative tori}. J.\ Geom.\ Phys.\ \textbf{150} (2020), 103594 (25 pages).

\bibitem{HKRV:arXiv18} Hundertmark, D.; Kunstmann, P.; Ried, T.; Vugalter, S.: 
 \emph{Cwikel's bound reloaded}. Preprint,  \texttt{arXiv:1809.05069}, 30 pp..
 
\bibitem{Il:JST12} Ilyin, A.: \emph{Lieb-Thirring inequalities on some manifolds.} J.\ Spectr.\ Theory \textbf{2} (2012), no.\ 1, 57--78. 

\bibitem{IL:SM16} Ilyin, A.A.; Laptev, A.A.: \emph{Lieb-Thirring inequalities on the torus}. Sb.\ Math.\ \textbf{207} (2016), 1410--1434. 

\bibitem{IL:StPMJ20} Ilyin, A.; Laptev, A.: \emph{Lieb-Thirring inequalities on the sphere.} St.\ Petersburg Math.\ J.\ \textbf{31} (2020), no. 3, 479--493 

\bibitem{ILZ:MN19} Ilyin, A.A.; Laptev, A.A.; Zelik, S.V.:  \emph{On the Lieb-Thirring constant on the torus}. Math.\ Notes \textbf{106} (6) (2019), 1019--1023.

\bibitem{ILZ:JFA20} Ilyin, A.A.; Laptev, A.A.; Zelik, S.V.:  \emph{Lieb-Thirring constant on the sphere and on the torus}. J.\ Funct.\ Anal.\ \textbf{279} (2020), 108784. 

\bibitem{IM:JGP18} Iochum B., Masson T.: \emph{Heat asymptotics for nonminimal Laplace type operators and application to noncommutative tori}. 
J.\ Geom.\ Phys.\ \textbf{129} (2018), 1--24.


\bibitem{KLPS:AIM13} Kalton, N.; Lord, S.; Potapov, D.; Sukochev, F.: \emph{Traces of compact operators and the noncommutative residue}. Adv.\ Math.\ \textbf{235} (2013), 1--55. 
 
\bibitem{Ku:TAMS58} Kunze, R.A.:  \emph{$L^p$-Fourier transforms on locally compact unimodular groups}. Trans.\ Amer.\ Math.\ Soc.\ \textbf{89} (1958), 519--540. 


 \bibitem{LP:JPDOA20} Lee, G.; Ponge, R.: \emph{Functional calculus for elliptic operators on noncommutative tori, I}.  
 J.\ Pseudo-Differ.\ Oper.\ Appl.\ \textbf{11} (2020), 935--1004. 

\bibitem{LP:Part2} Lee, G.; Ponge, R.: \emph{Functional calculus for elliptic operators on noncommutative tori, II}. In preparation.  


\bibitem{LM:GAFA16} Lesch, M.; Moscovici, H.: \emph{Modular curvature and Morita equivalence}. 
Geom.\ Funct.\ Anal.\ \textbf{26} (2016), no.\ 3, 818--873. 

\bibitem{LS:JAM97} Levin, D.; Solomyak, M.: \emph{The Rozenblum-Lieb-Cwikel inequality for Markov generators}. J.\ Anal.\ Math.\ \textbf{71} (1997), 173--193. 

\bibitem{LeSZ:2020} Levitina, G; Sukochev, F.; Zanin, D.: \emph{Cwikel estimates revisited}. Proc.\ London Math.\ Soc.\ \textbf{120} (2020), 265--304.

\bibitem{LNP:TAMS16} L\'evy, C.; Neira-Jim\'enez, C.; Paycha, S.:
\emph{The canonical trace and the noncommutative residue on the noncommutative torus}. Trans. Amer. Math. Soc. 368 (2016), no. 2, 1051--1095.

\bibitem{Li:BAMS76} Lieb, E. H.: \emph{Bounds on the eigenvalues of the Laplace and Schr\"odinger operators.} Bull.\ Amer.\ Math.\ Soc.\ \textbf{82} (1976), 751--752. 

\bibitem{Li:1980} Lieb, E. H.: \emph{The number of bound states of one-body Schr\"odinger operators and the Weyl problem}. Proceedings of Symposia in Pure Mathematics XXXVI. American Mathematical Society, Providence, RI, 1980. 241--252.

\bibitem{LS:Book} Lieb, E.; Seiringer, R.: \emph{The Stability of matter in quantum mechanics}. Cambridge University Press, Cambridge, 2010.

\bibitem{LT:PRL75} Lieb, E.H.; Thirring, W.E.: \emph{Bound on kinetic energy of fermions which proves stability of matter}. Phys.\ Rev.\ Lett.\ \textbf{35} (1975), 687--689.

\bibitem{LT:SMP76} Lieb, E.H.; Thirring, W.E.: \emph{Inequalities for the moments of the eigenvalues of the Schr\"odinger Hamiltonian and their relation to Sobolev inequalities}. In: Studies in Mathematical Physics, Princeton University Press, 1976, 269--303.

\bibitem{Liu:arXiv20}  Liu, Y.: \emph{General rearrangement lemma for heat trace asymptotic on noncommutative tori},
Preprint \texttt{arXiv:2004.05714}, 43 pp..

\bibitem{LPS:JFA10} Lord, S.; Potapov, D.; Sukochev, F.: \emph{Measures from Dixmier traces and zeta functions.} J.\ Funct.\ Anal.\ \textbf{259} (2010), no. 8, 1915--1949. 

\bibitem{LSZ:Book} Lord, S.; Sukochev, F.; Zanin, D.: \emph{Singular traces: Theory and applications}. de Gruyter Studies in Mathematics, 46, 2013.

\bibitem{LSZ:Survey19} Lord, S.; Sukochev, F.; Zanin, D.: \emph{Advances in Dixmier traces and  applications}. \emph{Advances in noncommutative geometry}. 491--583. Springer International Publishing, 2019.


\bibitem{MP:JMP22} McDonald, E; Ponge, R.: \emph{Cwikel estimates and negative eigenvalues of Schr\"odinger operators on noncommutative tori}. 
J.\ Math.\ Phys.\ \textbf{61} (2022), 043503 (37 pages). 

\bibitem{MSX:CMP19} McDonald, E; Sukochev, F.; Xiong, X.: \emph{Quantum differentiability on quantum tori}. Comm.\  Math.\ Phys.\ \textbf{371} (2019), 1231--1260.

\bibitem{MSZ:MA19} McDonald, E.; Sukochev, F.; Zanin, D.: \emph{A $C^{\ast}$-algebraic approach to the principal symbol II}. Math.\ Ann.\ \textbf{374} (2019), 273--322.

\bibitem{MSZ:arXiv21} McDonald, E.; Sukochev, F.; Zanin, D.: \emph{Semiclassical Weyl law and exact spectral asymptotics in noncommutative geometry}. 
Preprint \texttt{arXiv:2106.02235},  31 pp.. 

\bibitem{Ne:JFA74} Nelson, E.: \emph{Notes on non-commutative integration}. J.\ Funct.\ Anal.\ \textbf{15} (1974), 103--116.

\bibitem{Po:JMP20} Ponge, R.: \emph{Connes's trace Theorem for curved noncommutative tori. Application to scalar curvature}. J.\ Math.\ Phys.\ \textbf{61} (2020), 042301, 27 pp..  

\bibitem{Po:SIGMA20} Ponge, R.: \emph{Noncommutative residue and canonical trace on noncommutative tori. Uniqueness results}.  SIGMA Symmetry Integrability Geom.\ Methods Appl. \textbf{16} (2020), 061, 31 pp..   

\bibitem{Po:Weyl} Ponge, R.: \emph{Connes' integration and Weyl's laws}. Preprint, \texttt{arXiv:2107.01242}, 29pp..

\bibitem{Po:MPAG22} Ponge, R.: \emph{Weyl's laws and Connes' integration formulas for matrix-valued $\LLogL$-Orlicz potentials}. Math.\ Phys.\ Anal.\ Geom.\ \textbf{25}, 10 (2022), 33 pp.. 

\bibitem{PS:Springer16} Prodan, E.;  Schulz-Baldes, H.: 
\emph{Bulk and boundary invariants for complex topological insulators: from K-theory to physics}. 
Mathematical Physics Studies. Springer, 2016. 

\bibitem{PW:RMP75} Pusz, W., Woronowicz, S.L.: \emph{Functional calculus for sesquilinear forms and the purification map}. Rep.\ Math.\ Phys.\ \textbf{8} (1975), 159--170. 

\bibitem{RS2:1975} Reed, M.; Simon, B.: \emph{Methods of modern mathematical physics. II. Fourier Analysis, Selfadjointness.} Academic Press, Inc., New York, 1975. xv+361 pp..

\bibitem{Ri:CM90} Rieffel, M.A.: \emph{Noncommutative tori-A case study of non-commuative differentiable manifolds}. \emph{Geometric and topological invariants of elliptic operators} (Brunswick, ME, 1988), Contemp. Math.\ 105, Amer.\ Math.\ Soc.; Providence, RI, 1990, pp.\ 191--211.

\bibitem{Ro:SIGMA13} Rosenberg, J.: \emph{Levi-Civita's theorem for noncommutative tori}. SIGMA Symmetry Integrability Geom.\ Methods Appl.\ \textbf{9} (2013), Paper 071, 9 pp..

\bibitem{Roz:1972} Rozenbljum, G.\ V.: \emph{Distribution of the discrete spectrum of singular differential operators.} 
Dokl. Akad. Nauk SSSR \textbf{202} (1972), 1012--1015. English transl.\ Soviet Math.\ Dokl.\ \textbf{13} (1972), 245--249.

\bibitem{Roz:1976} Rozenbljum, G.\ V.: \emph{Distribution of the discrete spectrum of singular differential operators.} 
Izv.\ Vys\u{s}.\ U\u{c}ebn.\ Zaved.\ Matematika 1976, no.\ 1(164), 75--86. English transl.\ Soviet Math.\ (Iz.\ VUZ) \textbf{20} (1976), 63--71. 

\bibitem{Ro:JST22} Rozenblum, G.: \emph{Eigenvalues of singular measures and Connes noncommutative integration}. 
J.\ Spectr.\ Theory \textbf{12} (2022), 259--300.

\bibitem{RS:SPMJ98} Rozenblyum, G.; Solomyak, M. \emph{The Cwikel-Lieb-Rozenblyum estimator for generators of positive semigroups and semigroups dominated by positive semigroups}. St.\ Petersburg Math.\ J.\ \textbf{9} (1998), no.\ 6 1195--1211. 

\bibitem{Sc:Springer12} Schm\"udgen, K.: \emph{Unbounded self-adjoint operators on Hilbert space.} Graduate Texts in Mathematics, 265. Springer, Dordrecht, 2012. xx+432 pp.

 \bibitem{Se:AM53} Segal, I.E.:  \emph{A noncommutative extension of abstract integration}. Ann.\ Math\. (2) \textbf{57} (1953), 401--457. 

\bibitem{SSUZ:AIM15} Semenov, E.; Sukochev, F.; Usachev, A.; Zanin, D.: \emph{Banach limits and traces on $\mathcal{L}_{1,\infty}$}. Adv.\ Math.\ {\bf 285} (2015), 568--628. 

\bibitem{Si:TAMS76} Simon, B.:\emph{Analysis with weak trace ideals and the number of bound states of Schr\"odinger operators}. Trans.\ Amer.\ Math.\ Soc.\ \textbf{224} (1976), no. 2, 367--380. 

\bibitem{Si:AMS05} Simon, B.:\emph{Trace ideals and their applications.} Second edition. Mathematical Surveys and Monographs, 120. American Mathematical Society, Providence, RI, 2005.

\bibitem{Si:AMS15a} Simon, B: \emph{A Comprehensive Course in Analysis, Part 3: Harmonic Analysis}. American Mathematical Society, Providence, RI, 2015.

\bibitem{Si:AMS15} Simon, B.: \emph{Operator Theory:  A Comprehensive Course in Analysis, Part 4}. American Mathematical Society, Providence, RI, 2015. 

\bibitem{So:PLMS95} Solomyak M.; \emph{Spectral  problems  related  to  the  critical  exponent  in  the  Sobolev  embedding theorem}. Proc.\ London\ Math.\ Soc. (3) \textbf{71} (1995), no.\ 1, 53--75.

\bibitem{Sp:Padova92} Spera, M.: \emph{Sobolev theory for noncommutative tori}. Rend.\ Sem.\ Mat.\ Univ.\ Padova \textbf{86} (1992), 143--156.

\bibitem{SZ:JOT18} Sukochev, F.;  Zanin, D.: \emph{A $C^*$-algebraic approach to the principal symbol}. J.\ Oper.\ Theory \textbf{80} (2018), 481--522.  

\bibitem{SZ:arXiv19}  Sukochev, F.; Zanin, D.: \emph{Local invariants of non-commutative tori}.  Preprint \texttt{arXiv:1910.00758}, 33 pp.. 

\bibitem{SZ:arXiv20} Sukochev, F.; Zanin, D.: \emph{Cwikel-Solomyak estimates on tori and Euclidean spaces}. Preprint \texttt{arXiv:2008.04494}, 31 pp.. 

\bibitem{SZ:PAMS21} Sukochev, F.; Zanin, D.: \emph{Optimal constants in non-commutative H\"older inequality for quasi-norms}. Proc.\ Amer.\ Math.\ Soc.\ \textbf{149} (2021), no.\ 9, 3813--3817.

\bibitem{SZ:arXiv21} Sukochev, F.; Zanin, D.: \emph{Connes integration formula without singular traces}. Preprint \texttt{arXiv:2103.08817}, 19 pp.. 

 \bibitem{Ta:JPCS18} Tao, J.: \emph{The theory of pseudo-differential operators on the noncommutative $n$-torus}. J.\ Phys.: Conf.\ Ser.\ \textbf{965} 012042 (2018), 1--12.

\bibitem{XXY:MAMS18} Xiong, X.; Xu, Q;   Yin, Z.: \emph{Sobolev, {B}esov and {T}riebel-{L}izorkin spaces on quantum tori}. Mem.\ Amer.\ Math.\ Soc.\ \textbf{252} (2018), no.\ 1203, vi+118 pp..

\bibitem{We:CMP96} Weidl, T.: \emph{On the Lieb-Thirring constants $L_{\gamma,1}$ for $\gamma \geq1/2$}. Comm.\ Math.\ Phys.\ \textbf{178} (1996), no.\ 1, 135--146

\bibitem{Wo:NCR} Wodzicki, M.: \emph{Noncommutative residue. I. Fundamentals}. $K$-theory, arithmetic and geometry (Moscow, 1984--1986),  pp.~320--399, 
 Lecture Notes in Math., Vol.\ 1289, Springer, Berlin, 1987.
  
\end{thebibliography}
\end{document}